\documentclass[11pt]{amsart}
\usepackage{amsmath, amssymb, amscd, mathrsfs, url, pinlabel,hyperref, verbatim}
\usepackage[margin=1.25in,marginparwidth=1in,centering,letterpaper,dvips]{geometry}
\usepackage{color,dcpic,latexsym,graphicx,epstopdf,comment}
\usepackage[all]{xy}
\usepackage{tikz}
\usepackage{enumitem,upquote}
\title{Instantons and L-space surgeries}

\author[John A. Baldwin]{John A. Baldwin}
\address{Department of Mathematics \\ Boston College}
\email{john.baldwin@bc.edu}

\author[Steven Sivek]{Steven Sivek}
\address{Department of Mathematics\\Imperial College London}
\email{s.sivek@imperial.ac.uk}

\newtheorem {theorem}{Theorem}
\newtheorem {lemma}[theorem]{Lemma}
\newtheorem {proposition}[theorem]{Proposition}
\newtheorem {corollary}[theorem]{Corollary}

\newtheorem {question}[theorem]{Question}

\numberwithin{equation}{section}
\numberwithin{theorem}{section}

\theoremstyle{definition}
\newtheorem{definition}[theorem]{Definition}

\newtheorem{remark}[theorem]{Remark}
\newtheorem*{remark*}{Remark}

\newtheorem{example}[theorem]{Example}

\newlist{pcases}{enumerate}{1}
\setlist[pcases]{
  label=\bf{Case~\arabic*:}\protect\thiscase.~,
  ref=\arabic*,
  align=left,
  labelsep=0pt,
  leftmargin=0pt,
  labelwidth=0pt,
  parsep=0pt
}
\newcommand{\case}[1][]{%
  \if\relax\detokenize{#1}\relax
    \def\thiscase{}%
  \else
    \def\thiscase{~#1}%
  \fi
  \item
}

\newcommand{\Z}{\mathbb{Z}}
\newcommand{\R}{\mathbb{R}}
\newcommand{\C}{\mathbb{C}}
\newcommand{\CP}{\mathbb{CP}}

\newcommand{\Q}{\mathbb{Q}}
\newcommand{\spc}{\operatorname{Spin}^c}

\newcommand{\ssm}{\smallsetminus}

\newcommand{\ltb}{\mathit{tb}}
\newcommand{\lsl}{\mathit{sl}}

\newcommand{\maxsl}{\overline{\lsl}}
\newcommand{\rank}{\operatorname{rank}}
\newcommand{\Hom}{\operatorname{Hom}}

\newcommand{\Img}{\operatorname{Im}}
\newcommand{\ind}{\operatorname{ind}}
\newcommand\bA{\mathbb{A}}
\newcommand\cB{\mathcal{B}}
\newcommand\cC{\mathcal{C}}

\newcommand{\sD}{\mathscr{D}}
\newcommand{\cE}{\mathcal{E}}

\newcommand{\cG}{\mathcal{G}}
\newcommand{\sK}{\mathscr{K}}
\newcommand{\cN}{\mathcal{N}}
\newcommand{\cS}{\mathcal{S}}

\newcommand{\cptwo}{\overline{\CP}^2}

\newcommand\ttb{\mathit{tb}}
\newcommand\ssl{\mathit{sl}}

\newcommand\hfk{\mathit{HFK}}
\newcommand\hfkhat{\widehat{\hfk}}

\newcommand\SHI{\mathit{SHI}}

\newcommand\SHItfun{\textbf{\textup{\underline{SHI}}}} 

\newcommand\iinvt{\Theta} 
\newcommand\fiinvt{{\iinvt^\#}} 

\DeclareFontFamily{U}{mathx}{\hyphenchar\font45}
\DeclareFontShape{U}{mathx}{m}{n}{
      <5> <6> <7> <8> <9> <10>
      <10.95> <12> <14.4> <17.28> <20.74> <24.88>
      mathx10
      }{}
\DeclareSymbolFont{mathx}{U}{mathx}{m}{n}
\DeclareFontSubstitution{U}{mathx}{m}{n}
\DeclareMathAccent{\widecheck}{0}{mathx}{"71}

\newcommand{\inr}{\operatorname{int}}
\newcommand{\img}{\operatorname{Im}}
\newcommand{\hfhat}{\widehat{\mathit{HF}}}

\newcommand{\ad}{\operatorname{ad}}

\newcommand{\id}{\operatorname{id}}
\newcommand{\tr}{\operatorname{tr}}

\newcommand{\pt}{\mathrm{pt}}
\newcommand{\PD}{\mathit{PD}}
\newcommand{\godd}{\mathrm{odd}}
\newcommand{\geven}{\mathrm{even}}
\newcommand{\mirror}[1]{\overline{#1}}

\newcommand{\torsion}{\mathrm{torsion}}
\newcommand{\Airr}{\tilde{A}}
\newcommand{\std}{\mathrm{std}}

\usetikzlibrary{calc}
\tikzset{every picture/.style=thick}

\begin{document}

\begin{abstract}
We prove that instanton L-space knots are fibered and strongly quasipositive. Our proof differs conceptually from   proofs of the analogous result in Heegaard Floer homology, and includes a new decomposition theorem for cobordism maps in framed instanton Floer homology akin to the $\textrm{Spin}^c$ decompositions of cobordism maps in other Floer homology theories. As our main application, we prove (modulo a mild nondegeneracy condition) that for $r$ a positive rational number and $K$ a nontrivial knot in the $3$-sphere, there exists an irreducible homomorphism
\[\pi_1(S^3_r(K)) \to SU(2)\]
unless $r \geq 2g(K)-1$ and $K$ is both fibered and strongly quasipositive, broadly generalizing results of Kronheimer and Mrowka. We also answer a  question of theirs from 2004,  proving that there is always an irreducible homomorphism from the fundamental group of 4-surgery on a nontrivial knot to $SU(2)$. In another application, we show that a slight enhancement of the A-polynomial detects infinitely many torus knots, including the trefoil.\end{abstract}

\maketitle

\section{Introduction} \label{sec:introduction}

The most important invariant of a 3-manifold is its fundamental group. One of the most fruitful approaches to understanding the fundamental group is to study its homomorphisms into simpler groups. $SU(2)$ is an especially convenient choice  because it is one of the simplest nonabelian Lie groups, and because gauge theory provides powerful tools for studying $SU(2)$ representations of $3$-manifold groups.
Given a $3$-manifold $Y$,  we will therefore be interested in the representation variety
\begin{equation}\label{eqn:variety}R(Y) = \Hom(\pi_1(Y),SU(2)).\end{equation} If this variety is to tell us anything about the $3$-manifold not captured by its first homology,  it must contain elements with nonabelian image. Let us introduce the following terminology for when this is \emph{not} the case.
\begin{definition}
A 3-manifold $Y$ is \emph{$SU(2)$-abelian} if every $\rho\in R(Y)$ has abelian image.\footnote{If $b_1(Y)=0$ then  $\rho\in R(Y)$ has abelian image iff it has cyclic image.}
\end{definition} 

$SU(2)$-abelian manifolds can be thought of as the simplest manifolds from the perspective of the variety \eqref{eqn:variety}. A basic question when studying representation varieties  is whether they contain irreducible homomorphisms; since a homomorphism into $SU(2)$ is reducible iff it is abelian, this question admits the following satisfying answer for the varieties studied here.

\begin{remark}$R(Y)$ contains an irreducible iff $Y$ is \emph{not} $SU(2)$-abelian.
\end{remark}

Little is known as to which 3-manifolds are $SU(2)$-abelian, even among Dehn surgeries on knots in $S^3$.
One of the first major advances in this direction was Kronheimer and Mrowka's landmark proof of the Property P conjecture \cite{km-p}, in which they showed that $S^3_1(K)$ is not $SU(2)$-abelian (and hence not a homotopy 3-sphere) for any nontrivial knot $K$.  They then substantially strengthened this result as follows.
\begin{theorem}[\cite{km-su2}] \label{thm:km-su2}
Suppose $K\subset S^3$ is a nontrivial knot. Then $S^3_r(K)$ is not $SU(2)$-abelian for any rational number $r$ with $|r| \leq 2$.
\end{theorem}

It is  natural to ask whether Theorem~\ref{thm:km-su2}   holds for  other values of $r$. Bearing in mind that $5$-surgery on the right-handed trefoil is a lens space, and hence $SU(2)$-abelian, Kronheimer and Mrowka posed this question for the next two integer values of $r$.

\begin{question}[\cite{km-su2}] \label{q:km-su2}
Suppose $K\subset S^3$ is a nontrivial knot. Can $S^3_3(K)$ or $S^3_4(K)$  be  $SU(2)$-abelian?
\end{question}

In this paper, we use instanton Floer homology to prove the following broad generalization (modulo a mild nondegeneracy condition) of Theorem \ref{thm:km-su2}. In particular, this generalization establishes a link between the genus, the smooth slice genus, and the $SU(2)$ representation varieties of Dehn surgeries on a knot. 

\begin{theorem} \label{thm:main-surgery}
Suppose $K \subset S^3$ is a nontrivial knot, and $r=\frac{m}{n} > 0$ is a rational number such that \[\Delta_K(\zeta^2) \neq 0\] for any $m$th root of unity $\zeta$. Then $S^3_r(K)$ is not $SU(2)$-abelian unless \[r \geq 2g(K)-1\] and $K$ is both fibered and strongly quasipositive.
\end{theorem}

\begin{remark}
As alluded to above, the assumption on the Alexander polynomial in Theorem \ref{thm:main-surgery} is equivalent to a certain nondegeneracy condition (see \S\ref{ssec:instantonLspace}); we remark that it always holds when $m$ is a prime power. Note that rationals with prime power numerators  are dense in the reals (see Remark \ref{rmk:dense}). 
\end{remark}

\begin{remark}
The right-handed trefoil is the only fibered, strongly quasipositive knot of genus 1. Since $r$-surgery on this trefoil is not $SU(2)$-abelian for any  $r\in[0,2]$, by Proposition \ref{prop:cyclic-2-3}, Theorem \ref{thm:main-surgery} recovers Theorem \ref{thm:km-su2} for $r$ with prime power numerators; e.g., for $r=1,2$.
\end{remark}

We use the techniques behind Theorem \ref{thm:main-surgery} in combination with   results  of Klassen \cite{klassen} and Lin \cite{lin} 
to give the following nearly complete answer to (a  more general version of) Question~\ref{q:km-su2}. 
\begin{theorem} \label{thm:main-34-surgery}
Suppose $K\subset S^3$ is a nontrivial knot of genus $g$.  Then
\begin{itemize}
\item $S^3_4(K)$ is not $SU(2)$-abelian,
\item $S^3_r(K)$ is not $SU(2)$-abelian for any $r=\frac{m}{n}\in (2,3)$ with $m$ a prime power,  and
\item $S^3_3(K)$ is not $SU(2)$-abelian unless $K$ is fibered and strongly quasipositive and $g=2$.
\end{itemize}
\end{theorem}

\begin{remark}
We expect that $S^3_3(K)$ is not $SU(2)$-abelian for \emph{any} nontrivial knot but do not know how to prove this at present.
\end{remark}

We highlight two additional applications of our work below before describing the gauge-theoretic results underpinning these theorems.

First, in \cite{sivek-zentner}, Sivek and Zentner  explored the question of which knots in $S^3$ are \emph{$SU(2)$-averse}, meaning that they are nontrivial and have infinitely many $SU(2)$-abelian surgeries. These include torus knots, which have infinitely many lens space surgeries \cite{moser}, and conjecturally nothing else.  Below, we provide strong new restrictions on $SU(2)$-averse knots and their \emph{limit slopes}, defined for an $SU(2)$-averse knot as the unique accumulation point of its   $SU(2)$-abelian surgery slopes.
\begin{theorem} \label{thm:main-su2-averse}
Suppose $K \subset S^3$ is an $SU(2)$-averse knot. Then $K$ is fibered, and either $K$ or its mirror is strongly quasipositive with limit slope strictly greater than $2g(K)-1$.
\end{theorem}

Finally, one of the first applications of Theorem~\ref{thm:km-su2} was a proof by Dunfield--Garoufalidis \cite{dunfield-garoufalidis} and Boyer--Zhang \cite{boyer-zhang} that the A-polynomial \cite{ccgls} detects the unknot.  The A-polynomial \[A_K(M,L) \in \Z[M^{\pm1},L^{\pm1}]\] is defined in terms of the $SL(2,\C)$ character variety of $\pi_1(S^3\setminus K)$, and the key idea behind these proofs is that Theorem~\ref{thm:km-su2} implicitly provides, for a nontrivial knot $K$, infinitely many representations
\[ \pi_1(S^3\setminus K) \to SU(2) \hookrightarrow SL(2,\C) \]
whose characters are distinct, even after restricting to the peripheral subgroup.

More recently, Ni and Zhang proved \cite{ni-zhang} that $\Airr_K(M,L)$ together with knot Floer homology detects all torus knots, where the former is a slight enhancement of the A-polynomial whose definition omits the curve of reducible characters in the $SL(2,\C)$ character variety (see \S\ref{sec:a-polynomial}). Neither the $A$-polynomial nor its enhancement (nor knot Floer homology) alone can  detect \emph{all} torus knots. Nevertheless, we use Theorem \ref{thm:main-su2-averse}, together with an argument inspired by the proofs of  the unknot detection result above, to prove that this enhanced $A$-polynomial  suffices in infinitely many cases, as follows.

\begin{theorem} \label{thm:main-a-poly}
Suppose $K$ is either a trefoil or a torus knot $T_{p,q}$ where $p$ and $q$ are distinct odd primes.  Then \[\Airr_{J}(M,L) = \Airr_{K}(M,L)\] iff $J$ is isotopic to $K$.
\end{theorem}

\begin{remark}We prove that $\Airr_K(M,L)$  detects many other torus knots as well; see Corollary~\ref{cor:torus-knot-detection} for a more complete list.
\end{remark}

\subsection{Instanton L-spaces}
\label{ssec:instantonLspace}
The theorems above rely on  new results about \emph{framed instanton homology},  defined by Kronheimer and Mrowka in \cite{km-yaft} and   developed further by Scaduto in \cite{scaduto}.  This theory associates to a closed, oriented 3-manifold $Y$ and a closed, embedded multicurve $\lambda \subset Y$ a  $\Z/2\Z$-graded $\C$-module \[I^\#(Y,\lambda) = I_{\textrm{odd}}^\#(Y,\lambda)\oplus I_{\textrm{even}}^\#(Y,\lambda).\]
This module depends, up to isomorphism, only on $Y$ and the homology class \[[\lambda] \in H_1(Y;\Z/2\Z).\] 
Let us write  $I^\#(Y)$ to denote $I^\#(Y,\lambda)$ with $[\lambda]=0$.  Scaduto proved \cite{scaduto} that if $b_1(Y) = 0$ then $I^\#(Y)$ has Euler characteristic $|H_1(Y;\Z)|$, which implies that \[\dim I^\#(Y) \geq |H_1(Y;\Z)|.\] This  inspires the following terminology, by analogy with Heegaard Floer homology.

\begin{definition} \label{def:main-l-space}
A rational homology 3-sphere $Y$ is an \emph{instanton L-space} if \[\dim I^\#(Y) = |H_1(Y;\Z)|.\] A knot $K\subset S^3$ is  an \emph{instanton L-space knot} if $S^3_r(K)$ is an instanton L-space for some rational number $r>0$.
\end{definition}

\begin{remark} \label{rmk:oddgradinglspace}
It follows easily from the discussion above that a rational homology $3$-sphere $Y$ is an instanton L-space iff $ I_{\godd}^\#(Y) = 0$.
\end{remark}

Our main Floer-theoretic result is the following.

\begin{theorem} \label{thm:main-surgery-reduction}
Suppose $K\subset S^3$ is a nontrivial instanton L-space knot.  Then
\begin{itemize}
\item $K$ is fibered and strongly quasipositive, and 
\item $S^3_r(K)$ is an instanton L-space for a rational number  $r$ iff $r \geq 2g(K)-1$.
\end{itemize}
\end{theorem}

Let us explain how Theorem \ref{thm:main-surgery-reduction} bears on the results claimed in the previous section.

The connection with $SU(2)$ representations comes from the  principle that $I^\#(Y)$ should be the Morse--Bott homology of a Chern--Simons functional with critical set  \[\textrm{Crit}(\mathit{CS})\cong R(Y).\] This heuristic holds true so long as the elements in $R(Y)$ are nondegenerate in the Morse--Bott sense.  In \cite{bs-stein}, we observed that if $Y$ is $SU(2)$-abelian and $b_1(Y)=0$, then  \[H_*(R(Y);\C)\cong\C^{|H_1(Y;\Z)|},\] which then implies that $Y$ is an instanton $L$-space if all elements of $R(Y)$ are  nondegenerate. For $SU(2)$-abelian Dehn surgeries, this nondegeneracy is equivalent to the condition on the Alexander polynomial in Theorem~\ref{thm:main-surgery}, by work of Boyer and Nicas \cite{boyer-nicas}; see \cite[\S4]{bs-stein}.

With this discussion in the background, let us now prove Theorems \ref{thm:main-surgery} and \ref{thm:main-su2-averse} assuming Theorem~\ref{thm:main-surgery-reduction}.

\begin{proof}[Proof of Theorem \ref{thm:main-surgery}]Suppose  $K\subset S^3$ is a nontrivial knot, and $r=\frac{m}{n} > 0$ is a rational  number such that $\Delta_K(\zeta^2) \neq 0$ for any $m$th root of unity $\zeta$. Suppose $S^3_r(K)$ is $SU(2)$-abelian. Then $S^3_r(K)$ is an instanton L-space, by \cite[Corollary~4.8]{bs-stein}. Theorem \ref{thm:main-surgery-reduction} then implies that $K$ is  fibered and strongly quasipositive, and $r\geq 2g(K)-1$. 
\end{proof}

\begin{proof}[Proof of Theorem \ref{thm:main-su2-averse}] Suppose $K\subset S^3$ is   $SU(2)$-averse  with limit slope $r>0$. Then $K$ has an instanton L-space surgery of slope $\lceil r \rceil - 1 > 0$, by \cite[Theorem~1.1]{sivek-zentner}. As $K$ is nontrivial by definition, Theorem \ref{thm:main-surgery-reduction} then implies that $K$ is fibered and strongly quasipositive, and  \[r>\lceil r \rceil - 1\geq 2g(K)-1.\] 
 If $r < 0$, then we apply the same argument to the mirror $\mirror{K}$, whose limit slope is $-r > 0$. \end{proof}

Theorem \ref{thm:main-surgery-reduction} will be unsurprising to those familiar with Heegaard Floer homology, given the conjectural isomorphism \cite[Conjecture~7.24]{km-excision} \[I^\#(Y)\cong_{\mathrm{conj.}}\hfhat(Y)\otimes \C,\] and the fact that nontrivial knots with positive  Heegaard Floer L-space surgeries are already known to be fibered \cite{osz-lens,ni-hfk} and strongly quasipositive \cite{hedden-positivity}, with L-space surgery slopes comprising \cite{osz-rational} \[[2g(K){-}1,\infty)\cap \Q.\]  On the other hand, all proofs in the literature of the fiberedness result \emph{alone} use some subset  of: the $\Z\oplus\Z$-filtered Heegaard Floer complex associated to a knot, the large  surgery formula, the $(\infty,0,n)$-surgery exact triangle for $n>1$, and, in all cases, the $\textrm{Spin}^c$ decomposition of the Heegaard Floer groups of rational homology spheres. None of this structure is   available in framed instanton homology. 

Our proof of Theorem~\ref{thm:main-surgery-reduction} is thus, by necessity, largely novel (and can  even be translated to  Heegaard Floer and monopole Floer homology\footnote{The analogue of Theorem~\ref{thm:main-surgery-reduction} is  known to hold in monopole Floer homology, but only by appeal to the isomorphism between monopole Floer homology and Heegaard Floer homology.} to give new, conceptually simpler proofs of the analogous theorems in those settings). One  of the   ingredients  is a new decomposition theorem for  cobordism maps in framed instanton homology, analogous  to the $\textrm{Spin}^c$ decompositions of cobordism maps in Heegaard and monopole Floer homology, described in the next section. We expect  this decomposition result to have  other applications as well.

\subsection{A decomposition of framed instanton homology}

The framed instanton homology of a $3$-manifold comes equipped with a collection of commuting operators
\[ \mu: H_2(Y;\Z) \to \mathrm{End}(I^\#(Y,\lambda)) \]
such that   the eigenvalues of $\mu([\Sigma])$ are even integers between $2-2g(\Sigma)$ and $2g(\Sigma)-2$ for a surface $\Sigma\subset Y$.  Following Kronheimer and Mrowka \cite[Corollary~7.6]{km-excision}, this gives rise to an eigenspace decomposition
\[ I^\#(Y,\lambda) = \bigoplus_{s:H_2(Y;\Z)\to 2\Z} I^\#(Y,\lambda;s), \]
where each summand $I^\#(Y,\lambda;s)$ is the simultaneous generalized $s(h)$-eigenspace of $\mu(h)$ for all $h \in H_2(Y;\Z)$.  Only finitely many of these summands are nonzero.

A smooth cobordism
\[ (X,\nu): (Y_0,\lambda_0) \to (Y_1,\lambda_1) \]
 induces a homomorphism 
\begin{equation}\label{eqn:cobmap} I^\#(X,\nu): I^\#(Y_0,\lambda_0) \to I^\#(Y_1,\lambda_1). \end{equation}
We extend the eigenspace decomposition discussed above  to cobordism maps, in a way which mirrors the $\spc$ decomposition of cobordism maps in, say, the \emph{hat} flavor of Heegaard Floer homology.  Our main theorem in this vein is Theorem \ref{thm:main-cobordism-decomposition} below. In stating it, we will use the following notation: 
given an inclusion $i: M\hookrightarrow N$ of two manifolds and a homomorphism $s: H_2(N;\Z) \to \Z$, we  write 
\[ s|_M = s \circ i_*: H_2(M;\Z) \to H_2(N;\Z) \to \Z \]
for the restriction of $s$ to $M$.

\begin{theorem} \label{thm:main-cobordism-decomposition}
Let $(X,\nu): (Y_0,\lambda_0) \to (Y_1,\lambda_1)$ be a cobordism with $b_1(X) = 0$.  Then there is a natural decomposition of the cobordism map \eqref{eqn:cobmap} into a sum
\[ I^\#(X,\nu) = \sum_{s: H_2(X;\Z) \to \Z} I^\#(X,\nu;s) \]
of maps of the form
\[ I^\#(X,\nu;s): I^\#(Y_0,\lambda_0;s|_{Y_0}) \to I^\#(Y_1,\lambda_1;s|_{Y_1}) \]
with the following properties.
\begin{enumerate}
\item \label{i:finite-support} $I^\#(X,\nu;s) = 0$ for all but finitely many $s$.
\item \label{i:adjunction} If $I^\#(X,\nu;s)$ is nonzero, then $s(h)+h\cdot h\equiv 0\pmod{2}$ for all $h \in H_2(X;\Z)$, and $s$ satisfies an adjunction inequality
\[ |s([\Sigma])| + [\Sigma] \cdot [\Sigma] \leq 2g(\Sigma)-2 \]
for every smoothly embedded, connected, orientable surface $\Sigma \subset X$ of genus at least $1$ and positive self-intersection.
\item \label{i:composition-law} If $(X,\nu)$ is a composition of two cobordisms
\[ (Y_0,\lambda_0) \xrightarrow{(X_{01},\nu_{01})} (Y_1,\lambda_1) \xrightarrow{(X_{12},\nu_{12})} (Y_2,\lambda_2) \]
where $b_1(X_{01}) = b_1(X_{12}) = 0$, then we have a composition law
\[ I^\#(X_{12},\nu_{12};s_{12}) \circ I^\#(X_{01},\nu_{01};s_{01}) = \sum_{\substack{s: H_2(X;\Z)\to\Z \\ s|_{X_{01}}=s_{01}\\ s|_{X_{12}}=s_{12}}} I^\#(X,\nu;s) \]
for all $s_{01}: H_2(X_{01};\Z) \to \Z$ and $s_{12}: H_2(X_{12};\Z) \to \Z$.
\item \label{i:blowup-formula} If $\tilde{X} = X \# \cptwo$ denotes the blow-up of $X$, with $e$ the exceptional sphere and $E$ its Poincar\'e dual, then
\[ I^\#(\tilde{X},\nu;s+kE) = \begin{cases} \frac{1}{2}I^\#(X,\nu;s) & k=\pm1 \\ 0 & k \neq \pm 1 \end{cases} \]
and similarly
\[ I^\#(\tilde{X},\nu+e;s+kE) = \begin{cases} \hphantom{-}\frac{1}{2}I^\#(X,\nu;s) & k=-1 \\ -\frac{1}{2}I^\#(X,\nu;s) & k=+1 \\ \hphantom{-}0 & k \neq \pm 1 \end{cases} \]
for all $s: H_2(X;\Z)\to \Z$ and all $k \in \Z$.  
\item \label{i:sign-change} $I^\#(X,\nu+\alpha;s) = (-1)^{\frac{1}{2}(s(\alpha)+\alpha\cdot\alpha)+\nu\cdot\alpha} I^\#(X,\nu;s)$ for all $\alpha \in H_2(X;\Z)$.
\end{enumerate}
\end{theorem}

We remark that property~\eqref{i:sign-change} requires $\alpha$ to be a class in $H_2(X)$ rather than in $H_2(X,\partial X)$, so that $\nu+\alpha$ still restricts to $Y_0$ and to $Y_1$ as $\lambda_0$ and $\lambda_1$ respectively and hence $I^\#(X,\nu+\alpha)$ is still a map of the form $I^\#(Y_0,\lambda_0) \to I^\#(Y_1,\lambda_1)$.

The key to our proof of this theorem is a cobordism analogue of Kronheimer and Mrowka's structure theorem \cite{km-structure} for the Donaldson invariants of closed 4-manifolds. Briefly, one can extend the cobordism map \eqref{eqn:cobmap}  to maps
 \[ D_{X,\nu}: I^\#(Y_0,\lambda_0) \otimes \bA(X) \to I^\#(Y_1,\lambda_1), \] where $\bA(X)$ is the graded algebra \[ \bA(X) = \mathrm{Sym}(H_0(X;\R) \oplus H_2(X;\R)), \] such that \[I^\#(X,\nu) = D_{X,\nu}(-\otimes 1).\]  Letting $x=[\pt]$ in $H_0(X;\Z)$,  one defines a formal power series 
 \[ \sD_X^\nu(h) = D_{X,\nu}\left(-\otimes\left(e^h + \tfrac{1}{2}xe^h\right)\right) \]
for each $h \in H_2(X;\R)$, following \cite{km-structure}, and we prove that this series can be expressed as a finite sum 
\[ e^{Q_X(h)/2} \sum_{j=1}^r a_{j} e^{K_j(h)}, \]
where the $K_j$ are \emph{basic classes}, which we think of as homomorphisms $H_2(X;\Z)\to\Z$. This has exactly the same form as Kronheimer and Mrowka's original structure theorem, except that the coefficients are now homomorphisms
\[ a_{j}: I^\#(Y_0,\lambda_0) \to I^\#(Y_1,\lambda_1) \] instead of rational numbers.
Up to scaling, we define $I^\#(X,\nu;K_j)$ to be the map $a_{j}$.

To prove this structure theorem for cobordisms, we adapt   Mu\~noz's \cite{munoz-basic} alternative proof of  Kronheimer and Mrowka's structure theorem for closed $4$-manifolds. The proofs of that theorem  in both \cite{km-structure} and \cite{munoz-basic}  require that the  $4$-manifold contains a surface of positive self-intersection. The cobordism analogue in the case $b_2^+(X)=0$  therefore requires a new argument (see \S\ref{ssec:structure-general}).

\subsection{On the proof of Theorem \ref{thm:main-surgery-reduction}}\label{ssec:proof}

We provide below a detailed sketch of the proof of our main Floer-theoretic result, Theorem \ref{thm:main-surgery-reduction}.

Our proof that instanton L-space knots are fibered has two main components. The first is the result that the framed instanton  homology of $0$-surgery on a knot detects fiberedness. For this, suppose $\Sigma$ is a genus-minimizing Seifert surface for a nontrivial knot $K\subset S^3$   with meridian $\mu$. Let $\hat\Sigma$  denote the capped-off surface in $S^3_0(K)$, and let \[s_i: H_2(S^3_0(K)) \to 2\Z\] be the homomorphism defined by $s_i([\hat\Sigma])=2i$, for each $i\in \Z$. We prove the following, largely using a combination of  arguments due to Kronheimer and Mrowka \cite{km-excision} and Ni \cite{ni-hm}; this result may be of independent interest.

\begin{theorem} \label{thm:odd-dim-g-1}
If $K \subset S^3$ is a nontrivial knot as above, then
\[ \dim I^\#_\godd(S^3_0(K),\mu;s_{g(K)-1}) \geq 1, \]
with equality iff $K$ is fibered.
\end{theorem}

The second component, which is conceptually novel, uses only the surgery exact triangle and formal properties of the decomposition of cobordism maps given by Theorem~\ref{thm:main-cobordism-decomposition},  as follows.  For each $k \geq 0$, there is  a surgery exact triangle
\[ \dots \to I^\#(S^3) \xrightarrow{I^\#(X_k,\nu_k)} I^\#(S^3_k(K)) \xrightarrow{I^\#(W_{k+1},\omega_{k+1})} I^\#(S^3_{k+1}(K)) \to \dots, \]
where we use $I^\#(S^3_0(K),\mu)$ instead of  $I^\#(S^3_0(K))$ when $k=0$. 
 Here, $X_k$ is the trace of $k$-surgery on $K$, and  $W_{k+1}$ is the trace of $-1$-surgery on a meridian of $K$ in $S^3_k(K)$, so  that  \begin{equation}\label{eqn:XW} X_k\cup_{S^3_k(K)} W_{k+1}\cong X_{k+1}\#\cptwo.\end{equation} We apply \eqref{eqn:XW} inductively, using  the blow-up formula and adjunction inequality  in Theorem \ref{thm:main-cobordism-decomposition}, to  eventually prove (Lemma~\ref{lem:ker-vn}) that, for any $n> 0$, the kernel of the composition
 \begin{equation} \label{eq:composition-w}
I^\#(S^3_0(K),\mu) \xrightarrow{I^\#(W_1,\omega_1)} I^\#(S^3_1(K)) \xrightarrow{I^\#(W_2,\omega_2)} \dots \xrightarrow{I^\#(W_n,\omega_n)} I^\#(S^3_n(K))
\end{equation}
 lies in the subspace of $I^\#(S^3_0(K),\mu)$ spanned by the  elements 
\[ y_i = I^\#(X_0,\nu_0; t_{0,i})({\mathbf 1}) \in I^\#(S^3_0(K),\mu; s_i), \]
for \[i=1-g(K),2-g(K),\dots,g(K)-1,\] where each homomorphism \[t_{0,i}: H_2(X_0) \to \Z\] is  defined by $t_{0,i}([\hat\Sigma])=2i$, and ${\mathbf 1}$ generates $I^\#(S^3) \cong \C.$ Noting that the composition  \eqref{eq:composition-w} preserves the $\Z/2\Z$-grading, it follows that if $S^3_n(K)$ is an instanton L-space then \begin{equation} \label{eq:main-top-eigenspace}
\dim I^\#_\godd(S^3_0(K),\mu; s_{g(K)-1}) = 1,
\end{equation}
which then implies that $K$ is fibered by Theorem \ref{thm:odd-dim-g-1}. Indeed, if \eqref{eq:main-top-eigenspace} does not hold, then the group  \[I^\#_\godd(S^3_0(K),\mu; s_{g(K)-1})\]  contains an element which is not a multiple of $y_{g(K)-1}$, which is thus sent by the composition \eqref{eq:composition-w}  to a nonzero element of \[I^\#_\godd(S^3_n(K)) = 0\] (see Remark \ref{rmk:oddgradinglspace}), a contradiction.  The same reasoning shows that $y_{g(K)-1}\neq 0$ when $K$ is an instanton L-space knot.

For strong quasipositivity, suppose $K$ is an instanton L-space knot. Since $K$ is fibered, it supports some contact structure $\xi_K$ on $S^3$. Then $K$ is strongly quasipositive iff $\xi_K$ is tight  \cite{hedden-positivity,baader-ishikawa}. To prove that this contact structure is tight, we first  recall from above that $y_{g(K)-1}\neq 0$. By duality, this implies that the map
\begin{equation} \label{eq:main-sqp-map}
I_{\textrm{even}}^\#(-S^3_0(K),\mu;s_{1-g(K)}) \cong \C \to I^\#(-S^3)
\end{equation}
induced by  turning $X_0$  upside down has nonzero image. The image of the analogous map in  Heegaard Floer homology defines the Heegaard Floer contact invariant $c(\xi_K)$ \cite{osz-contact}, and  $c(\xi_K) \neq 0$ implies that $\xi_K$ is tight. Unfortunately, we cannot  prove  that the nontriviality of \eqref{eq:main-sqp-map} implies the tightness of $\xi_K$  in general (the Heegaard Floer proof ultimately uses the description of $c(\xi_K)$ in terms of the knot Floer filtration for $K$, for which there is no analogue in the instanton setting). 

Our solution involves cabling. 
Namely, for  coprime integers $p$ and $q$ with $q\geq 2$ and \[\frac{p}{q}>2g(K)-1,\] a simple argument shows that the $(p,q)$-cable $K_{p,q}$   is also an instanton L-space knot, which implies as above that the map 
\begin{equation*} \label{eq:main-sqp-map-cable}
I_{}^\#(-S^3_0(K_{p,q}),\mu;s_{1-g(K_{p,q})})  \to I^\#(-S^3)
\end{equation*}
is nonzero. Using the  fact that such cables can be deplumbed, we  show that the nontriviality of this map implies that a slight variant of the contact class we defined in \cite{bs-instanton} is nonzero for the contact structure $\xi_{K_{1,q}}$ corresponding to the cable $K_{1,q}$. We  ultimately deduce from this that  $\xi_{K_{1,q}}$ is tight (we only work this out concretely in  the case $q=2$, which suffices for our application). Noting that $ \xi_K\cong\xi_{K_{1,q}} $ for any positive $q$, we conclude that $\xi_K$ is tight as well.

Finally, our characterization of the  L-space surgery slopes has two parts. First, we proved in \cite[Theorem~4.20]{bs-stein} that if the set of positive instanton L-space slopes for  a nontrivial knot $K$ is nonempty, then it has the form
\[ [N, \infty) \cap \Q\]  for some positive integer $N$, and we show in Proposition \ref{prop:first-l-space-slope} that \begin{equation}\label{eqn:Nbound}N\leq 2g(K)-1.\end{equation} To prove that this  inequality  is actually an \emph{equality}, we then employ an observation due to  Lidman, Pinz\'on-Caicedo, and Scaduto in  \cite{lpcs}. There, they use the $\Z/2\Z$-grading of our   contact invariant \cite{bs-instanton} together with our main technical result in \cite{bs-stein} to prove, under the assumptions that $K$ is a nontrivial instanton L-space knot with  maximal self-linking  number \begin{equation}\label{eqn:maxsllspace}\maxsl(K)=2g(K)-1,\end{equation} that \begin{equation}\label{eqn:Nequality}N\geq 2g(K)-1\end{equation} (this argument is actually a relatively easy part of their more general work in  \cite{lpcs}; we will therefore explain it in full in \S\ref{ssec:lspaceslopes}).  Our proof that  instanton L-space knots are strongly quasipositive is independent of their work; therefore, since strongly quasipositive knots  satisfy \eqref{eqn:maxsllspace},  the bounds \eqref{eqn:Nbound} and \eqref{eqn:Nequality} and hence the equality \[N=2g(K)-1\]  hold for \emph{all} nontrivial instanton L-space knots, as desired.

\begin{remark} 
Some of the arguments in our proof of Theorem \ref{thm:main-surgery-reduction} invoke the Giroux correspondence \cite{giroux-icm} between contact structures and open books. This correspondence is  generally accepted to be true, but since a complete proof  has yet to appear in the literature, we have chosen to indicate below which of our results depend on it.

We use the Giroux correspondence  to prove the strong quasipositivity claim in Theorem~\ref{thm:main-surgery-reduction} (and therefore in Theorems \ref{thm:main-surgery} and \ref{thm:main-su2-averse} as well), which is  then used to prove the claim in Theorem~\ref{thm:main-surgery-reduction} that positive instanton L-space slopes $r$ satisfy the bound \begin{equation*}\label{eqn:rbound}r \geq 2g(K)-1.\end{equation*} This bound is required for  the $3$-  and $4$-surgery cases of Theorem~\ref{thm:main-34-surgery}, and in Theorem~\ref{thm:main-a-poly} to prove that $\Airr_K(M,L)$ detects any of the claimed torus knots other than the trefoils, which are handled separately by Theorem~\ref{thm:apoly-trefoil}.
We remark that Theorem~\ref{thm:l-space-knots-are-fibered} and Proposition~\ref{prop:2-surgery-l-space} do not rely on this  correspondence. These assert that an instanton L-space knot  is fibered with genus equal to its smooth slice genus; and that if $S^3_1(K)$ or $S^3_2(K)$ is an instanton L-space, then $K$ is either the unknot or the right-handed trefoil.
\end{remark}

\subsection{Organization}

 We provide  some background on instanton Floer homology in \S\ref{sec:background}. The material in \S\ref{sec:fibered} and \S\ref{sec:compare-top-eigenspaces} is devoted to proving  the fiberedness detection result, Theorem \ref{thm:odd-dim-g-1}.  In \S\ref{sec:I-sharp}, we prove the structure theorem for cobordism maps on framed instanton homology, which we then use   in \S\ref{sec:eigenspace} for our decomposition result, Theorem~\ref{thm:main-cobordism-decomposition}.
In \S\ref{sec:L-spaces}, we prove that instanton L-space knots are fibered with Seifert genus equal to smooth slice genus (Theorem~\ref{thm:l-space-knots-are-fibered}).  We  prove in \S\ref{sec:sqp} that instanton L-space knots are   strongly quasipositive (Theorem \ref{thm:sqp}) and   complete the characterization of instanton L-space surgery slopes (Theorem \ref{thm:bound}), finishing the proof of   Theorem~\ref{thm:main-surgery-reduction}.
We conclude with several applications of Theorem~\ref{thm:main-surgery-reduction}. We have already shown  that this theorem  implies Theorems \ref{thm:main-surgery} and \ref{thm:main-su2-averse}.  In \S\ref{sec:small-surgeries}, we investigate $SU(2)$-abelian surgeries with small surgery slope,   and prove Theorem \ref{thm:main-34-surgery}. We study the A-polynomials of torus knots in \S\ref{sec:a-polynomial},  and prove Theorem~\ref{thm:main-a-poly}.

\subsection{Acknowledgments}

We thank Ken Baker, John Etnyre, Matt Hedden, Jen Hom, and Tom Mrowka for helpful conversations.  We thank Tye Lidman, Juanita Pinz\'on-Caicedo, and Chris Scaduto for the same, and also for discussing their work \cite{lpcs} with us while it was still in progress.  We also thank the referee for useful feedback.  JAB was supported by NSF CAREER Grant DMS-1454865.

\section{Background } \label{sec:background}

In this section, we provide    background on instanton Floer homology and establish some notational conventions. Much of our discussion below is adapted from \cite{km-excision} and \cite{scaduto}.
\subsection{Conventions} All manifolds in this paper are smooth, oriented, and compact, and all submanifolds are smoothly and properly embedded. 

\subsection{Instanton Floer homology} 
\label{ssec:instanton-floer}Let $(Y,\lambda)$ be an \emph{admissible pair}, meaning that $Y$ is a closed, connected $3$-manifold, and $\lambda\subset Y$ is a multicurve which intersects some  surface   in an odd number of points. We associate  to this pair:
\begin{itemize}
\item a Hermitian line bundle $w\to Y$ with $c_1(w)$ Poincar{\'e} dual to $\lambda$, and
\item a $U(2)$ bundle $E\to Y$ equipped with an isomorphism $\theta:\wedge^2 E \to w$.
\end{itemize}
The \emph{instanton Floer homology} $I_*(Y)_\lambda$ is roughly the   Morse homology, with $\C$-coefficients, of the Chern--Simons functional on the space \[\cB = \cC/\cG\] of $SO(3)$-connections on $\ad(E)$ modulo determinant-1 gauge transformations (the automorphisms of $E$ which respect $\theta$), as in \cite[\S5.6]{donaldson-book}. This   group has a relative $\Z/8\Z$-grading, which reduces to a canonical  $\Z/2\Z$-grading \cite{froyshov,donaldson-book}.

\begin{remark}
\label{rmk:isomorphism-homology}
Up to isomorphism, the group $I_*(Y)_\lambda$ depends only on  $Y$ and the homology class \[[\lambda]\in H_1(Y;\Z/2\Z)\] of the multicurve $\lambda$.
\end{remark}

Each homology class $h \in H_k(Y;\R)$ gives rise to a  class $\mu(h)\in H^{4-k}(\cB)$, and therefore to an endomorphism  \[\mu(h): I_*(Y)_\lambda \to I_{*+k-4}(Y)_\lambda\]  of degree $k-4$, as   in \cite{donaldson-kronheimer}. These endomorphisms are additive in the sense that \[\mu(h_1+h_2) = \mu(h_1)+\mu(h_2)\] for  $h_1,h_2 \in H_k(Y;\R)$, and   endomorphisms associated to   even-dimensional   classes commute.  
This  implies that for any collection of even-dimensional  classes, $I_*(Y)_\lambda$ is the  direct sum of the simultaneous generalized eigenspaces  of the associated operators.

\begin{remark}Going forward, we will often blur the distinction between closed submanifolds of $Y$ and the homology classes they represent, and  we will   use \emph{eigenspace} to mean \emph{generalized eigenspace}   unless  stated otherwise.
\end{remark}

After choosing an absolute lift of the relative $\Z/8\Z$-grading on $I_*(Y)_\lambda$, we can write an element of this group as \[v=(v_0,v_1,v_2,v_3,v_4,v_5,v_6,v_7),\] where $v_i$ is in grading $i\in\Z/8\Z$. Let \[\phi:I_*(Y)_\lambda \to I_*(Y)_\lambda\] be the map defined for $v$ as above by \begin{equation}\label{eq:eigenspace-isomorphism}\phi(v)=(v_0,v_1,iv_2,iv_3,-v_4,-v_5,-iv_6,-iv_7).\end{equation} This map gives rise to  isomorphisms between eigenspaces, as below.

\begin{lemma}
\label{lem:symmetry} For any $h\in H_2(Y;\R)$ and $m,n\in \C$, the map $\phi$ defines an isomorphism from  the  $(m,n)$-eigenspace of the operators $\mu(h),\mu(\pt)$ on $I_*(Y)_\lambda$ to the  $(im,-n)$-eigenspace.
\end{lemma}

\begin{proof}
This follows easily from the fact that $\mu(\alpha)$ has degree $-2$ and $\mu(\pt)$  degree $-4$.
\end{proof}

Using work of Mu\~noz \cite{munoz-ring}, Kronheimer and Mrowka prove the following in {\cite[Corollary~7.2 \& Proposition~7.5]{km-excision}}.

\begin{theorem}
\label{thm:km-surface}
Let $(Y,\lambda)$ be an admissible pair, and suppose $R\subset Y$  is a connected surface  of genus $g>0$. Then the simultaneous eigenvalues of the  operators $\mu(R),\mu(\pt)$  on  $I_*(Y)_\lambda$ are contained in the set
\[ \{ (i^r\cdot 2k, (-1)^r \cdot 2) \},\] for $0 \leq r \leq 3$ and $0 \leq k \leq g-1$. \end{theorem} 

This theorem implies that   $\mu(\pt)^2-4$ acts nilpotently on $I_*(Y)_\lambda$, but in some cases we know that it is in fact the zero operator,  per the following result of  Fr{\o}yshov \cite[Theorem~9]{froyshov} (the result stated below uses the observation ``$N_1=N_2=1$'' preceding the cited theorem, which follows from \cite{munoz-ring}).

\begin{theorem}
\label{thm:low-genus-mu} If $(Y,\lambda)$ is an admissible pair such that $\lambda$ intersects a surface of genus at most $2$ in an odd number of points, then $\mu(\pt)^2-4\equiv0$ on $I_*(Y)_\lambda$.
\end{theorem}

Theorem \ref{thm:km-surface} motivates the  definition below.

\begin{definition} \label{def:top-connected}
Let  $(Y,\lambda)$ be an admissible pair, and $R\subset Y$ a disjoint union \[R=R_1\cup \dots \cup R_n,\] where each $R_i$ is a connected surface in $Y$ of genus $g_i>0$. We define
\[ I_*(Y|R)_\lambda \subset I_*(Y)_\lambda \]
to be the $(2g_1{-}2,\dots,2g_n{-}2,2)$-eigenspace of the operators $\mu(R_1),\dots,\mu(R_n),\mu(\pt)$.

\end{definition}

\begin{remark}
\label{rmk:z2-grading} Note that the $\Z/2\Z$-grading on $I_*(Y)_\lambda$ descends to a canonical $\Z/2\Z$-grading on $I_*(Y|R)_\lambda$ since the operators $\mu(R_i)$ and $\mu(\pt)$ have even degree.
\end{remark}

We extend this definition to disconnected manifolds via tensor product (we will implicitly use this more general definition in the statement of  Theorem \ref{thm:excision}).

\begin{remark}
\label{rmk:spectrum-new-surface}
For any connected surface $\Sigma\subset Y$ of genus $g>0$,   the eigenvalues of $\mu(\Sigma)$ acting on $I_*(Y|R)_\lambda$  are contained in the set of even integers \[\{2-2g,4-2g,\dots,2g-2\},\]by Theorem \ref{thm:km-surface} and the fact that $I_*(Y|R)_\lambda$ is contained in the $2$-eigenspace of $\mu(\pt)$. We will therefore  refer to the $(2g-2)$-eigenspace of $\mu(\Sigma)$ as the \emph{top eigenspace} of this operator.
\end{remark}

\begin{example}
\label{ex:product}It follows from \cite{munoz-ring} that \[I_*(S^1\times R|R)_\gamma \cong \C\]  for $R$  a  connected surface of positive genus, and $\gamma = S^1\times\{\pt\}$.
\end{example}

We will make extensive use of the following nontriviality result due to Kronheimer and Mrowka    \cite[Theorem~7.21]{km-excision}. 

\begin{theorem}\label{thm:instanton-nonzero}
Let   $(Y,\lambda)$ be an admissible pair with $Y$  irreducible, and suppose   $R\subset Y$ is a connected surface which minimizes genus in its homology class, such that $\lambda\cdot R$ is odd. Then $I_*(Y|R)_\lambda$ is nonzero.
\end{theorem}

\subsection{Cobordism maps}
\label{ssec:cob-maps} Suppose $(Y_0,\lambda_0)$ and $(Y_1,\lambda_1)$ are admissible pairs, and let $(X,\nu)$ be a cobordism from the first pair to the second. We associate to this cobordism:
\begin{itemize}
\item a Hermitian line bundle $w\to X$  with $c_1(w)$ Poincar\'e dual to $\nu$, and 
\item a $U(2)$ bundle $E\to X$ equipped with an isomorphism $\theta:\wedge^2 E \to w$
\end{itemize}
which restrict to the bundle data associated to the admissible pairs at either end. One then defines a map
\begin{equation*} I_*(X)_\nu: I_*(Y_0)_{\lambda_0} \to I_*(Y_1)_{\lambda_1} \end{equation*}  in the standard way: given generators $a_i \in I_*(Y_i)_{\lambda_i},$ one considers an associated configuration space \[\cB(X,E,a_0,a_1)\] of $SU(2)$-connections on $E$ modulo determinant-1 gauge transformations, and defines
the coefficient \[ \langle I_*(X)_\nu(a_0), a_1 \rangle \]
to be a count of projectively anti-self-dual instantons in a 0-dimensional moduli space \[\mathcal{M}_0(X,E,a_0,a_1)\subset \cB(X,E,a_0,a_1).\]  The map $I_*(X)_\nu$ is homogeneous with respect to the $\Z/2\Z$ grading, and shifts this grading by the amount below, according to \cite[\S4.6]{scaduto} following \cite[\S4.5]{km-unknot}: \begin{equation}\label{eqn:grading-shift}
\deg(I_*(X)_\nu) = -\frac{3}{2}(\chi(X)+\sigma(X)) + \frac{1}{2}(b_1(Y_1) - b_1(Y_0)) \pmod{2}.
\end{equation}
More generally, one can define a map \begin{equation*}\label{eqn:algebra-map} \Psi_{X,\nu}: I_*(Y_0)_{\lambda_0} \otimes \bA(X) \to I_*(Y_1)_{\lambda_1}, \end{equation*}
where \[ \bA(X) = \mathrm{Sym}^*(H_0(X;\R)\oplus H_2(X;\R)) \otimes \Lambda^*H_1(X;\R) \] is a graded algebra
in which $H_k(X;\R)$ has grading $4-k$. Namely, a monomial $z = c_1c_2\dots c_k \in \bA(X)$ of degree $d$ gives rise to a class
\[ \mu(c_1) \cup \mu(c_2) \cup \dots \cup \mu(c_k) \in H^d(\cB(X,E,a_0,a_1)), \] following \cite{donaldson-kronheimer}, and the coefficient \[ \langle \Psi_{X,\nu}(a_0 \otimes z), a_1\rangle \]
 is the sum of the evaluations of this class on the $d$-dimensional components of $\mathcal{M}_d(X,E,a_0,a_1)$. In particular, \[I_*(X)_\nu = \Psi_{X,\nu}(-\otimes 1).\]
  
\begin{remark} \label{rmk:homology-X-dependence}
After choosing a homology orientation, as in \cite[\S3.8]{km-yaft}, the maps $\Psi_{X,\nu}$ depend only on $X$ and the homology class
\[ [\nu] \in H_2(X,\partial X; \Z). \]
Up to sign they only depend on the mod 2 class of $[\nu]$, as we will see in Remark~\ref{rem:nu-plus-even}.
\end{remark}

The map $\Psi_{X,\nu}$ interacts nicely with the actions of  $H_*(Y_i;\R)$ on $I_*(Y_i)_{\lambda_i}$ described in \S\ref{ssec:instanton-floer}. Namely, given $h_i\in H_*(Y_i;\R)$ for $i=0,1$,  we have the relations
\begin{align}
\label{eqn:relations-psi}\Psi_{X,\nu}(\mu(h_0)a \otimes z) &= \Psi_{X,\nu}(a \otimes h_0 z), \\
\mu(h_1) \Psi_{X,\nu}(a \otimes z) &= \Psi_{X,\nu}(a \otimes h_1 z).
\end{align}
Therefore, if $h_0$ and $h_1$ are homologous in $X$, we have
\[ \Psi_{X,\nu}(\mu(h_0) {-} \otimes -) = \mu(h_1) \Psi_{X,\nu}(- \otimes -). \]
In particular, the cobordism map $I_*(X)_\nu$ intertwines the actions of $\mu(h_0)$ and $\mu(h_1)$ in this case, and therefore respects the corresponding eigenspace decompositions.   It follows that if  $R_0 \subset Y_0$ and $R_1 \subset Y_1$ are surfaces of the same positive genus which are homologous in $X$ then $I_*(X)_\nu$ restricts to a map \[ I_*(X)_\nu: I_*(Y_0|R_0)_{\lambda_0} \to I_*(Y_1|R_1)_{\lambda_1}. \]

\subsection{The excision theorem}
\label{ssec:excision}

Let $Y$ be a closed $3$-manifold,  $\Sigma_1\cup \Sigma_2\subset Y$ a disjoint union of  connected surfaces  of the same positive genus, and $\lambda$  a multicurve in $Y$ which intersects each of $\Sigma_1$ and $\Sigma_2$ in the same odd number of points with the same sign, \[\#(\lambda\cap \Sigma_1)=|\lambda\cdot\Sigma_1|=|\lambda\cdot\Sigma_2| = \#(\lambda\cap \Sigma_2),\] according to the following two cases:
\begin{itemize}
\item if $Y$ is connected, we require that $\Sigma_1$ is not homologous to $\Sigma_2$;
\item if $Y$ has two components, we require that one $\Sigma_i$ is contained in each component. 
\end{itemize}
Let $\varphi:\Sigma_1 \to \Sigma_2$ be an orientation-reversing diffeomorphism such that \[\varphi(\lambda\cap \Sigma_1) = \varphi(\lambda\cap \Sigma_2).\]  Let $Y'$ be the manifold with four boundary components, \[\partial Y' =\Sigma_1 \cup -\Sigma_1 \cup \Sigma_2 \cup -\Sigma_2,\] obtained from $Y$ by cutting   along $\Sigma_1 \cup \Sigma_2$. Let $\tilde Y$ be the closed $3$-manifold obtained  from $Y'$ by gluing  $\Sigma_1$ to  $-\Sigma_2$ and $\Sigma_2$ to $-\Sigma_1$ using the diffeomorphism $h$, and let $\tilde{\lambda}$ denote the corresponding  multicurve in $\tilde Y$ obtained from $\lambda$. Denote by $\tilde \Sigma_i$  the image of $\Sigma_i$ in $\tilde Y$.

Kronheimer and Mrowka prove the following excision theorem  in \cite[Theorem~7.7]{km-excision}, stated in terms of the notation  above;  this theorem  generalizes work of Floer \cite{floer-surgery,braam-donaldson} in which the surfaces  $\Sigma_i$  are assumed to be tori.

\begin{theorem} \label{thm:excision}
   There is an isomorphism 
\[ I_*(Y|\Sigma_1\cup\Sigma_2)_\lambda \xrightarrow{\cong} I_*(\tilde{Y}|\tilde{\Sigma}_1\cup\tilde{\Sigma}_2)_{\tilde{\lambda}}, \] which is homogeneous with respect to the $\Z/2\Z$-grading, and which, for any   submanifold $R\subset Y$  disjoint from $\Sigma_1\cup\Sigma_2$,  intertwines the actions of $\mu(R)$ on either side.\end{theorem}

\begin{remark}
The fact that the isomorphism in Theorem \ref{thm:excision} is homogeneous with respect to the $\Z/2\Z$-grading comes from the fact that it is induced by a cobordism, and such maps are homogeneous, as in \eqref{eqn:grading-shift}.
\end{remark}

The following  is a consequence of Theorem~\ref{thm:excision} together with Example \ref{ex:product}, proved exactly as in \cite[Lemma~4.7]{km-excision}.

\begin{theorem} \label{thm:fibered-has-rank-1}
Let $(Y,\lambda)$ be an admissible pair such that $Y$ fibers over the circle with fiber a  connected surface $R$ of positive genus with $\lambda\cdot R = 1$. Then $I_*(Y|R)_\lambda \cong \C.$
\end{theorem}

We will prove a converse to this theorem in Section \ref{sec:fibered}, modulo a  technical assumption.

\subsection{Framed instanton homology} \label{ssec:framed-homology}  In this section, we review  the construction of framed instanton homology, which was defined by Kronheimer and Mrowka in \cite{km-yaft} and   developed further by Scaduto in \cite{scaduto}; it provides a way of assigning an instanton Floer  group to a  pair $(Y,\lambda)$ without the admissibility assumption (which requires, e.g.,  that $b_1(Y)>0$).%

Let us fix once and for all:
\begin{itemize}
\item  a basepoint $y_T$ in $T^3=T^2\times S^1$, and 
\item a curve $\lambda_T = \{\pt\}\times S^1$ disjoint from this basepoint.
\end{itemize}
Now suppose  $(Y,\lambda)$ is a pair consisting of a closed, connected $3$-manifold $Y$ and a multicurve $\lambda\subset Y$. Choose a basepoint $y \in Y$ disjoint from $\lambda$, and define \[\lambda^\#:= \lambda\cup \lambda_T\] in $Y\#T^3$, where the connected sum is performed at the basepoints $y$ and $y_T$. Then \[(Y\#T^3,\lambda^\#)\] is an admissible pair, which enables us to make the following definition.
\begin{definition}
\label{def:framed}The \emph{framed instanton homology} of $(Y,\lambda)$ is the group 
\[I^\#(Y,\lambda):=I_*(Y\#T^3|T^2)_{\lambda^\#},\]
  where $T^2$ refers to any  $T^2\times\{\pt\}\subset T^3$ disjoint from $y_T$.
\end{definition}

\begin{remark}
We will frequently write $I^\#(Y)$ for $I^\#(Y,\lambda)$ when $[\lambda]=0$ in $H_1(Y;\Z/2\Z)$, and will generally conflate $\lambda$ with its mod $2$ homology class, given  Remark \ref{rmk:isomorphism-homology}.
\end{remark}

\begin{remark} \label{rmk:simpletype3}
Theorem \ref{thm:low-genus-mu} implies that the operator \[\mu(\pt):I_*(Y\#T^3)_{\lambda^\#}\to I_*(Y\#T^3)_{\lambda^\#}\] satisfies $\mu(\pt)^2-4\equiv 0$. In particular, the framed instanton homology
\[I^\#(Y,\lambda)\subset I_*(Y\#T^3)_{\lambda^\#}\] is  the honest (i.e. not generalized) $2$-eigenspace of $\mu(\pt)$. Since this operator has degree $-4$, the relative $\Z/8\Z$-grading on $I_*(Y\#T^3)_{\lambda^\#}$ descends to a relative $\Z/4\Z$-grading on $I^\#(Y,\lambda)$.
\end{remark}

As in Remark \ref{rmk:z2-grading}, the framed instanton homology $I^\#(Y,\lambda)$ has a canonical $\Z/2\Z$-grading, \[I^\#(Y,\lambda) = I_{\textrm{odd}}^\#(Y,\lambda)\oplus I_{\textrm{even}}^\#(Y,\lambda),\] and Scaduto proves the following in \cite[Corollary 1.4]{scaduto}.

\begin{proposition}
\label{prop:euler}
The Euler characteristic of $I^\#(Y,\lambda)$ is given by 
\[\chi(I^\#(Y,\lambda)) = \begin{cases}
|H_1(Y;\Z)|& b_1(Y)=0\\
0& b_1(Y)>0.
\end{cases}\]
\end{proposition}

This motivates the following definition, as in the introduction.

\begin{definition} \label{def:main-l-space-2}
A rational homology 3-sphere $Y$ is  an \emph{instanton L-space} if \[\dim I^\#(Y) = |H_1(Y;\Z)|;\] or, equivalently, if \[I_{\textrm{odd}}^\#(Y)=0.\] A knot $K\subset S^3$ is  an \emph{instanton L-space knot} if $S^3_r(K)$ is an instanton L-space for some rational number $r>0$.
\end{definition}

\begin{remark}
\label{rmk:integralLspace}
As mentioned in the introduction, we proved in \cite[Theorem~4.20]{bs-stein} that if the set of positive instanton L-space slopes for  a nontrivial knot $K$ is nonempty, then it has the form
\[ [N, \infty) \cap \Q\]  for some positive integer $N.$
\end{remark}

Cobordisms induce maps on framed instanton homology as well. Namely, suppose \[(X,\nu,\gamma):(Y_0,\lambda_0,y_0) \to (Y_1,\lambda_1,y_1)\] is a cobordism where $\gamma\subset X$ is an arc from $y_0$ to $y_1$. Following \cite[\S7.1]{scaduto}, given framings of $y_0$ and $y_1$ and a compatible framing of $\gamma$, we form a new cobordism
\[ X^\# := X \bowtie (T^3\times [0,1]): Y_0 \# T^3 \to Y_1 \#T^3 \]
by removing tubular neighborhoods of $\gamma$ and $y_T\times [0,1]$ from $X$ and $T^3 \times [0,1]$ respectively, and gluing what remains along the resulting $S^2 \times [0,1]$.  We then define the  map
\[I^\#(X,\nu) : I^\#(Y_0,\lambda_0)\to I^\#(Y_1,\lambda_1)\] by \[I^\#(X,\nu)(a):= I_*(X^\#)_{\nu^\#}(a).\] Scaduto proves   \cite[Proposition 7.1]{scaduto} that the grading shift of this map agrees with that in \eqref{eqn:grading-shift}, as below.
\begin{proposition}
\label{prop:grading-shift}
The  map $I^\#(X,\nu)$ shifts the $\Z/2\Z$-grading by
\[ \deg(I^\#(X,\nu)) = -\frac{3}{2}(\chi(X)+\sigma(X)) + \frac{1}{2}(b_1(Y_1) - b_1(Y_0)) \pmod{2}. \] In particular, this map is homogeneous.
\end{proposition}

Note that there is a natural inclusion \[H_2(X;\R) \hookrightarrow H_2(X^\#;\R).\] Indeed, any class in $H_2(X;\R)$ can be represented by a multiple of some surface which avoids the path $\gamma$; and if such a surface bounds a 3-chain which intersects $\gamma$ transversely in finitely many points, then it also bounds a 3-chain in $X^\#$ obtained from the former by replacing the 3-balls normal to each intersection point with $\gamma$ with the corresponding punctured $T^3$ in $T^3\times[0,1]$.  Thus, we  have an inclusion \[\bA(X) \hookrightarrow \bA(X^\#),\] which enables us to  extend the cobordism map $I^\#(X,\nu)$ to a map 
\[ D_{X,\nu}: I^\#(Y_0,\lambda_0) \otimes \bA(X) \to I^\#(Y_1,\lambda_1)  \] defined by \[D_{X,\nu}(a\otimes z) := \Psi_{X^\#,\nu^\#}(a\otimes z).\] In particular, \[I^\#(X,\nu) = D_{X,\nu}(-\otimes 1).\] We note below  that these cobordism maps automatically satisfy   an analogue of the \emph{simple type} condition on the Donaldson invariants of closed $4$-manifolds. The proof of this lemma is a straightforward application of the relation \eqref{eqn:relations-psi} together with the fact that $\mu(\pt)$ acts as multiplication by $2$  on $I^\#(Y_0,\lambda_0)$, as  in Remark \ref{rmk:simpletype3}.

\begin{lemma} \label{lem:simple-type}
For $x=[\pt]\in H_0(X;\R)$ and any  $z \in \bA(X)$,  \[ D_{X,\nu}(- \otimes x^2z) = 4D_{X,\nu}(- \otimes z) \]
as maps  from $I^\#(Y_0,\lambda_0)$ to $I^\#(Y_1,\lambda_1)$. \qed
\end{lemma}

\begin{remark}
As indicated above, we will  omit the path $\gamma$ (and the basepoints) from the notation for these cobordism maps. In practice, we will only consider cobordisms built from handle attachments, in which $\gamma$ is implicitly understood to be a product arc. We may also omit $\nu$ from the notation for $I^\#(X,\nu)$ or $D_{X,\nu}$ when $[\nu]=0$ in $H_2(X,\partial X;\Z/2\Z)$, and will generally conflate $\nu$ with its mod $2$ homology class, given Remark \ref{rmk:homology-X-dependence}.
\end{remark}

Given a pair $(Y,\lambda)$ as above and a connected surface $R\subset Y$  of positive genus, we will use the notation  $I^\#(Y,\lambda|R)$ to refer to the $(2g(R)-2)$-eigenspace of $\mu(R)$ acting on $I^\#(Y,\lambda)$. In particular, \begin{equation}\label{eqn:Irelative}I^\#(Y,\lambda|R)=I_*(Y\#T^3|R)_{\lambda^\#}.\end{equation}  Moreover, given a cobordism \[(X,\nu):(Y_0,\lambda_0)\to (Y_1,\lambda_1)\] and connected surfaces $R_0\subset Y_0$ and $R_1\subset Y_1$ of the same positive genus which are homologous to a surface $R\subset X$, we will denote by 
\[I^\#(X,\nu|R):I^\#(Y_0,\lambda_0|R_0)\to I^\#(Y_1,\lambda_1|R_1)\] the map induced by $I^\#(X,\nu)$, restricted to the top eigenspaces of $\mu(R_0)$ and $\mu(R_1)$ (the induced map respects these  eigenspaces, as discussed at the end of \S\ref{ssec:cob-maps}).

\subsection{The eigenspace decomposition}
\label{ssec:eigenspace-decomposition} Following \cite[Corollary~7.6]{km-excision}, one can define a decomposition of framed instanton homology which bears some resemblance to the $\spc$ decompositions of monopole and Heegaard Floer homology, as below.

\begin{definition}
Given a homomorphism $s: H_2(Y;\Z) \to 2\Z$, let
\[ I^\#(Y,\lambda;s) = \bigcap_{h \in H_2(Y;\Z)} \left( \bigcup_{n \geq 1} \ker (\mu(h)-s(h))^n \right) \]
 be the simultaneous $s(h)$-eigenspace of the operators $\mu(h)$ on $I^\#(Y,\lambda)$, over all $h \in H_2(Y;\Z)$.
 \end{definition}

\begin{theorem} \label{thm:eigenspace-decomposition}
There is a direct sum decomposition
\[ I^\#(Y,\lambda) = \bigoplus_{s: H_2(Y;\Z) \to 2\Z} I^\#(Y,\lambda;s). \]
If $I^\#(Y,\lambda;s)$ is nonzero, then \[|s([\Sigma])| \leq 2g(\Sigma)-2\] for every connected surface $\Sigma \subset Y$ of positive genus.  Thus, only finitely many summands are nonzero.  Finally, there is an isomorphism
\[ I^\#(Y,\lambda;s) \cong I^\#(Y,\lambda;-s) \]
for each $s$, which preserves the $\Z/2\Z$-grading.
\end{theorem}

\begin{proof}Remark \ref{rmk:spectrum-new-surface} says that for any connected surface $\Sigma\subset Y$ of positive genus, the eigenvalues of the operator $\mu(\Sigma)$ acting on   \[I^\#(Y,\lambda):=I_*(Y\#T^3|T^2)_{\lambda^\#} \] are contained in the set of even integers between $2-2g(\Sigma)$ and $2g(\Sigma)-2$, which immediately implies the second claim of the theorem. The direct sum decomposition then follows from the fact that the operators $\mu(h)$ commute, and have even eigenvalues since each $h$ can be represented by a connected surface of positive genus.  For the last claim, let  \[\phi: I_*(Y\#T^3)_{\lambda^\#}\to I_*(Y\#T^3)_{\lambda^\#}\] be the map defined in \eqref{eq:eigenspace-isomorphism}.     Lemma \ref{lem:symmetry} implies that for every \[s:H_2(Y;\Z) \to 2\Z,\] the square  $\phi^2$ defines  an isomorphism  from the $(s(h),2)$-eigenspace of $\mu(h),\mu(\pt)$  acting on $I_*(Y\#T^3)_{\lambda^\#} $ to the $(-s(h),2)$-eigenspace; that is, an isomorphism \[ I^\#(Y,\lambda;s) \to I^\#(Y,\lambda;-s). \] It is immediate from the definition of $\phi$ that this map preserves the $\Z/2\Z$-grading.
\end{proof}

In \S\ref{sec:eigenspace}, we   extend the eigenspace decomposition in Theorem~\ref{thm:eigenspace-decomposition} to cobordism maps. 

\subsection{The surgery exact triangle} The following theorem is originally due to Floer \cite{floer-surgery,braam-donaldson}, though the formulation below is taken from \cite[Theorem 2.1]{scaduto}. 

\begin{theorem} \label{thm:exact-triangle-main}
Let $Y$ be a closed, connected 3-manifold,  $\lambda\subset Y$ a multicurve, and $K\subset Y$  a framed knot  with meridian $\mu\subset Y\ssm N(K)$. Then there is an exact triangle
\[ \dots \to I_*(Y)_\lambda \to I_*(Y_0(K))_{\lambda \cup \mu} \to I_*(Y_1(K))_\lambda \to I_*(Y)_\lambda \to \dots\]
whenever  $(Y,\lambda)$, $(Y_0(K),\lambda \cup \mu)$, and $(Y_1(K),\lambda)$ are all admissible pairs. Moreover, the maps in this triangle are  induced by the corresponding 2-handle cobordisms.
\end{theorem}

This implies the following  surgery exact triangle for framed instanton homology, without any admissibility hypotheses, as in \cite[\S7.5]{scaduto}. 

\begin{theorem} \label{thm:exact-triangle}
Let $Y$ be a closed, connected 3-manifold,  $\lambda\subset Y$ a multicurve, and $K\subset Y$  a framed knot with meridian $\mu\subset Y\ssm N(K)$. Then there is an exact triangle
\[ \dots \to I^\#(Y,\lambda) \to I^\#(Y_0(K),\lambda \cup \mu) \to I^\#(Y_1(K),\lambda) \to I^\#(Y,\lambda) \to \dots, \]
in which the maps are  induced by the corresponding 2-handle cobordisms.
\end{theorem}

Note that if $Y$ is a homology sphere, and $K$ has framing $n$ relative to its Seifert framing, then by taking \[ \lambda = \begin{cases} 0 & n\mathrm{\ odd} \\ \mu & n\mathrm{\ even}, \end{cases} \]
we ensure that the curves $\lambda$ and $\lambda \cup \mu$ appearing in Theorem~\ref{thm:exact-triangle} (and Theorem~\ref{thm:exact-triangle-main}) are zero in homology over $\Z/2\Z$.  In the case of $Y=S^3$, for instance, this choice of $\lambda$ therefore yields an exact triangle
\begin{equation} \label{eq:triangle-untwisted}
\dots \to I^\#(S^3) \to I^\#(S^3_n(K)) \to I^\#(S^3_{n+1}(K)) \to \dots.
\end{equation}

\subsection{The connected sum theorem} \label{ssec:connectedsum}

We will make use of the following version of Fukaya's connected sum theorem \cite{fukaya}, as applied by Scaduto in \cite{scaduto}.
\begin{theorem} \label{thm:connected-sum}
Let  $(Y,\lambda)$ be an admissible pair.  Then there is an isomorphism of relatively $\Z/4\Z$-graded $\C$-modules,
\[ I^\#(Y,\lambda) \cong \ker(\mu(\pt)^2-4) \otimes H_*(S^3;\C), \]
where $\mu(\pt)^2-4$ above is viewed as acting on four consecutive gradings of $I_*(Y)_\lambda$.
\end{theorem}

In light of Theorem \ref{thm:low-genus-mu}, this implies the following, as in \cite[\S9.8]{scaduto}.

\begin{corollary} \label{cor:connected-sum-low-genus}
Let  $(Y,\lambda)$ be an admissible pair such that $\lambda$ intersects a surface of genus at most 2 in an odd number of points. Then
\[ I^\#(Y,\lambda) \otimes H_*(S^4;\C) \cong I_*(Y)_\lambda \otimes H_*(S^3;\C) \]
as relatively $\Z/4\Z$-graded $\C$-modules.
\end{corollary}

\section{Instanton homology and fibered manifolds} \label{sec:fibered}

In this section, we prove a converse to Theorem~\ref{thm:fibered-has-rank-1}---namely, that if $I_*(Y|R)_\lambda \cong \C$, then $Y$ is fibered with fiber $R$---modulo a technical assumption which suffices for our applications. We  will use this converse, Theorem \ref{thm:fibered}, to show  that  instanton Floer homology of $0$-surgery on a knot detects whether the knot is fibered, as in Theorem \ref{thm:detect-fibered-knot} below, and likewise for framed instanton homology, as in Theorem \ref{thm:odd-dim-g-1}, proved in \S\ref{sec:compare-top-eigenspaces}. The latter will then be used in our proof that instanton L-space knots are fibered, as outlined in the introduction.

The statement of Theorem \ref{thm:fibered} requires the  following definition.

\begin{definition}
A compact 3-manifold $M$ with boundary $\Sigma_+ \sqcup \Sigma_-$ is a \emph{homology product} if both of the maps
\[ (i_\pm)_*: H_*(\Sigma_\pm; \Z) \to H_*(M; \Z) \]
induced by inclusion are isomorphisms.
\end{definition}

\begin{theorem} \label{thm:fibered}
Let $(Y,\lambda)$ be an admissible pair with $Y$  irreducible, and $R\subset Y$  a connected surface of positive genus with $\lambda \cdot R=1$.  If $Y \ssm N(R)$ is a homology product and \[I_*(Y|R)_\lambda \cong \C,\] then $Y$ is fibered over the circle with fiber $R$.
\end{theorem}

\begin{remark}
We would like to prove Theorem~\ref{thm:fibered} without the homology product assumption.  Indeed, the corresponding results in monopole and Heegaard Floer homology are stated without it  because one can show in those cases that rank 1 \emph{automatically} implies this  assumption.  The key input there is the fact that the Turaev torsion of $Y$ is generally equal to the Euler characteristic of  monopole and Heegaard Floer homology, but no such results are known for instanton homology.
\end{remark}

We will reduce Theorem~\ref{thm:fibered} to the special case in Proposition \ref{prop:fibered-v-prime} by a standard argument, exactly as  in \cite{ni-hfk, ni-hf, km-excision, ni-hm}. Before stating the proposition, another definition.

\begin{definition}
A  homology product $M$ with boundary $\Sigma_+ \sqcup \Sigma_-$ is \emph{vertically prime} if every closed, connected surface in $M$ in the same homology class and with the same genus as $\Sigma_+$ is isotopic to either $\Sigma_+$ or $\Sigma_-$.
\end{definition}

\begin{proposition} \label{prop:fibered-v-prime}
Theorem \ref{thm:fibered} holds under the additional assumption that $M=Y \ssm N(R)$ is vertically prime.
\end{proposition}

\begin{proof}
Our proof is a straightforward combination of the proof of \cite[Theorem~7.18]{km-excision}, which asserts that sutured instanton homology detects products, with  the work in \cite{ni-hm}, which uses the analogous result for sutured monopole homology plus  arguments from \cite{ni-hfk,ni-hfk-erratum} to conclude that monopole Floer homology detects fibered 3-manifolds. In particular, we follow Ni's argument and terminology from \cite{ni-hm} nearly to the letter.

Let $\cE \subset H_1(M)$ be the subgroup spanned by the homology classes of product annuli in $M$, whose boundary we write as $R_+ \sqcup R_-$.  Following Ni, we will first show that $\cE = H_1(M)$.  Suppose that $\cE \neq H_1(M)$. Then there exist essential simple closed curves \[\omega_- \subset R_-\textrm{ and }\omega_+ \subset R_+\] which are homologous in $M$ and satisfy $[\omega_\pm] \not\in \cE$. In this case, Ni fixes an arc $\sigma$ from $R_-$ to $R_+$, and  constructs for any sufficiently large  $m$  connected surfaces $S_1,S_2\subset M$ with \begin{align*}\partial S_1 &=   \omega_- -\omega_+ \textrm{ and } \sigma\cdot S_1 = m,\\
\partial S_2 &=   \omega_+-\omega_-  \textrm{ and } \sigma\cdot S_2 = m,
\end{align*}   such that decomposing $M$  (viewed as a sutured manifold with an empty suture) along either $S_1$ or $S_2$ produces a taut sutured manifold \cite{ni-hm, ni-hfk}.
Now choose a diffeomorphism  \[h: R_+ \to R_-\] such that $Z = M/h$ is a homology $R \times S^1$, $h(\omega_+) = \omega_-$, and \[h(R_+\cap \lambda) =h(R_-\cap \lambda),\] so that $\lambda$ extends to a closed curve $\bar{\lambda}\subset Z$. As in  the proof of \cite[Theorem~7.18]{km-excision}, we can arrange that 
 the closed surfaces $\bar{R} = R_+/h$, $\bar{S}_1 = S_1/h$, and $\bar{S}_2 = S_2/h$ in $Z$ satisfy
\begin{align} \label{eqn:identity1}[\bar{S}_1] &= m[\bar{R}] + [\bar{S}_0],\\
 [\bar{S}_2] &= m[\bar{R}] - [\bar{S}_0],\\
\label{eqn:identity3} \chi(\bar{S}_1) = \chi(\bar{S}_2) &= m\chi(\bar{R}) + \chi(\bar{S}_0),
  \end{align}
for some closed surface $\bar{S}_0 \subset Z$ with $\bar{\lambda}\cdot \bar{S}_0 = 0$ and \[2g(\bar{S}_0)-2 > 0\] which minimizes genus in its homology class ($\bar{S}_0 = S_0/h$ for some connected surface $S_0\subset M$ which provides a homology between $\omega_+$ and $\omega_-$ and satisfies $\sigma\cdot S_0=0$;  $S_0$ is not an annulus since  $[\omega_\pm] \not\in \cE$).

Following the proof of \cite[Theorem~7.18]{km-excision}, we let   $T_i \subset Z$ be the closed, connected surface obtained by smoothing out the circle of intersection $\bar{R} \cap \bar{S}_i$ for $i=1,2$. Then \[2g(T_i)-2 = (2g(\bar{R})-2) + (2g(\bar{S}_i)-2).\] Moreover, $T_i$ is genus-minimizing in its homology class since decomposing $Z$ along $T_i$ yields  the same taut sutured manifold as decomposing $M$ along $S_i$ (a decomposing surface must be genus-minimizing if the resulting decomposition is taut). 
 Let us suppose $m$ is even so that $\bar{\lambda}\cdot T_i$ is odd. Then Theorem~\ref{thm:instanton-nonzero} implies that $I_*(Z|T_i)_{\bar{\lambda}} \neq 0$ for each $i$.
  Let  $x_i$ be a nonzero element of $I_*(Z|T_i)_{\bar{\lambda}}$ for $i=1,2$. Since this element lies in the $(2g(T_i)-2)$-eigenspace of
\[ \mu(T_i) = \mu(\bar{R}) + \mu(\bar{S}_i), \]
and the eigenvalues of $\mu(\bar{R})$ and $\mu(\bar{S}_i)$ are at most $2g(\bar{R})-2$ and $2g(\bar{S}_i)-2$, it follows that $x_i$   lies in these top eigenspaces of $\mu(\bar{R})$ and $\mu(\bar{S}_i)$ as well.  Thus, we have    nonzero elements \[x_1,x_2\in I_*(Z|\bar{R})_{\bar{\lambda}}\] such that   $x_i$  lies in the $(2g(\bar{S}_i)-2)$-eigenspace of $\mu(\bar{S}_i)$.  It then follows from  the identities \eqref{eqn:identity1}-\eqref{eqn:identity3} that $x_1$ and $x_2$ also lie in the  top eigenspaces of $\mu(\bar{S}_0)$ and $-\mu(\bar{S}_0)$, respectively. But these eigenspaces are disjoint since $2g(\bar{S}_0)-2>0$.
Therefore, \[\dim I_*(Z|\bar{R})_{\bar{\lambda}} \geq 2.\] On the other hand, we have 
\[ I_*(Z|\bar{R})_{\bar{\lambda}} \cong I_*(Y|R)_\lambda \cong \C \]
by Theorem~\ref{thm:excision},  a contradiction.  We conclude that $\cE = H_1(M)$ after all.

In the terminology of  \cite[Corollary~6]{ni-hm}, the fact that  $\cE = H_1(M)$ implies that the map \[i_*: H_1(\Pi) \to H_1(M)\] is surjective, where  $(\Pi,\Psi)$ is the characteristic product pair for $(M,\partial M)$. According to Ni, we can therefore find an embedded $G \times [-1,1]$ inside $M$, where $G\subset R_+$ is a genus one surface with boundary obtained as the tubular neighborhood of two curves in $R_+$ intersecting in a single point. Let \[M' = M \ssm (\inr(G)\times [-1,1])\textrm{ and }\gamma' = \partial G \times \{0\} \subset \partial M'.\] Then $(M',\gamma')$ is a sutured manifold and a homology product. By definition, its sutured instanton homology is given by \[ \SHI(M',\gamma') := I_*(Y|R)_\lambda \cong \C. \] 
 $(M',\gamma')$ is taut since this module is nonzero, so  \cite[Theorem~7.18]{km-excision} asserts that $(M',\gamma')$  is a product sutured manifold.  In particular, \[M \cong R\times [-1,1],\] which implies that $Y$ is fibered with fiber $R$.
\end{proof}

As in \cite{ni-hm}, Theorem \ref{thm:fibered} follows easily from Proposition \ref{prop:fibered-v-prime} and excision.

\begin{proof}[Proof of Theorem~\ref{thm:fibered}] Suppose the hypotheses of the theorem are satisfied.
Take a maximal collection $R=R_0, R_1, \dots, R_{n-1}$ of disjoint, pairwise non-isotopic closed surfaces in $Y$ which have the same genus as $R$ and  are  homologous to $R$. Let us write
\[ Y \ssm (N(R_0) \sqcup N(R_1) \sqcup \dots \sqcup N(R_{n-1})) = M_0 \sqcup M_1 \sqcup \dots \sqcup M_{n-1}, \]
in which we have ordered the $R_i$ so that $\partial M_i = R_i \sqcup R_{i+1}$, interpreting the subscripts mod $n$.  For each $i$,   form a closed manifold $Y_i$ and curve $\lambda_i\subset Y_i$ from $M_i$ by taking a diffeomorphism $R_i \to R_{i+1}$ which identifies $\lambda\cap R_i$ with $\lambda\cap R_{i+1}$, and using this map to glue $R_i$ to $R_{i+1}$.  By repeated application of Theorem~\ref{thm:excision}, we have that
\[ I_*(Y|R)_\lambda \cong I_*(Y_0|R_0)_{\lambda_0} \otimes \dots \otimes I_*(Y_{n-1}|R_{n-1})_{\lambda_{n-1}}. \]
Then $I_*(Y|R)_\lambda \cong \C$ implies that $I_*(Y_i|R_i)_{\lambda_i} \cong \C$ for all $i$.  Moreover, the fact that $Y \ssm N(R)$ is a homology product implies that each  \[M_i \cong Y_i \ssm N(R_i)\] is as well, by \cite[Lemma~4.2]{ni-hf}. Finally, each $M_i$ is vertically prime since the collection $\{R_i\}$ is maximal, so Proposition~\ref{prop:fibered-v-prime} says that $M_i \cong R_i \times [-1,1]$ for all $i$.  It follows that $Y$ is fibered over $S^1$ with fiber $R$, as desired.
\end{proof}

The remainder of this section is devoted to proving Theorem \ref{thm:detect-fibered-knot}.

Suppose $K$ is a nontrivial knot in $S^3$, and let $\hat\Sigma \subset S^3_0(K)$ be the closed surface of genus $g=g(K)$ formed by capping off a genus-$g$ Seifert surface $\Sigma$ for $K$ with a disk. Let $\mu \subset S^3_0(K)$ be the image of the meridian of $K$ in the surgered manifold, with $\mu\cdot \hat\Sigma = 1$.  For each integer  \[j=1-g,2-g,\dots,g-1,\] let \[ I_*(S^3_0(K),\hat\Sigma,j)_\mu \subset I_*(S^3_0(K))_\mu \]
denote the  $(2j,2)$-eigenspace of the  operators $\mu(\hat\Sigma),\mu(\pt)$ (these eigenspaces are trivial for $|j| \geq g$ by Theorem \ref{thm:km-surface}).  Since the  $\Z/2\Z$-grading on $I_*(S^3_0(K))_\mu$ is fixed by $\mu(\hat\Sigma)$ and $\mu(\pt)$, each  $I_*(S^3_0(K),\hat\Sigma,j)_\mu$ inherits this grading and hence has a well-defined Euler characteristic.  Lim proved the following {\cite[Corollary~1.2]{lim}}, which says that these Euler characteristics are determined by the Alexander polynomial of $K$. 

\begin{theorem} \label{thm:lim}
Let $\Delta_K(t)$ be the Alexander polynomial of $K$, normalized so that \[\Delta_K(t) = \Delta_K(t^{-1})\textrm{ and }\Delta_K(1) = 1.\] Then 
\[ \frac{\Delta_K(t)-1}{t-2+t^{-1}} = \sum_{j=1-g}^{g-1} \chi(I_*(S^3_0(K),\hat\Sigma,j)_\mu) t^j. \]
\end{theorem}

We will use Lim's result in combination with Theorem \ref{thm:fibered} to prove the following fiberedness detection theorem, stated in terms of the notation above. 

\begin{theorem} \label{thm:detect-fibered-knot}
 The $\C$-module $I_*(S^3_0(K)|\hat\Sigma)_\mu$ is always nontrivial, and
\[ I_*(S^3_0(K)|\hat\Sigma)_\mu \cong \C \]
iff $K$ is fibered.
\end{theorem}

\begin{proof}
Gabai \cite{gabai-foliations3} proved that $S^3_0(K)$ is irreducible and that $\hat\Sigma$ minimizes genus within its homology class iff its genus is equal to the Seifert genus $g=g(K)$, so the nontriviality is a consequence of Theorem~\ref{thm:instanton-nonzero}.  He  also proved in \cite{gabai-foliations3} that $K$ is fibered with fiber $\Sigma$ iff $S^3_0(K)$ is fibered with fiber $\hat\Sigma$.  We already know that if $S^3_0(K)$ is fibered then $I_*(S^3_0(K)|\hat\Sigma)_\mu \cong \C$, by Theorem~\ref{thm:fibered-has-rank-1}. Let us prove the converse.

Suppose $I_*(S^3_0(K)|\hat\Sigma)_\mu \cong \C$.  We  claim that $S^3_0(K) \ssm N(\hat\Sigma)$ is a homology product.  Indeed, the Euler characteristic 
\[ \chi(I_*(S^3_0(K),\hat\Sigma,g-1)_\mu) = \chi(I_*(S^3_0(K)|\hat\Sigma)_\mu) \]
is $\pm 1$, so Theorem~\ref{thm:lim} says that
\[ \Delta_K(t) = 1 + (t - 2 + t^{-1}) \cdot (\pm t^{g-1} + \dots) = \pm t^g + \dots, \]
where   the omitted terms have strictly lower degree. The Alexander polynomial $\Delta_K(t)$ is therefore monic of degree $g$. This implies that $S^3_0(K) \ssm N(\hat\Sigma)$ is a homology product by \cite[Lemma~4.10]{ghiggini}.  We may then conclude  from Theorem~\ref{thm:fibered}  that $S^3_0(K)$ is fibered with fiber $\hat\Sigma$, and hence that $K$ is fibered with fiber $\Sigma$.
\end{proof}

\section{Framed instanton homology and fibered knots} \label{sec:compare-top-eigenspaces}

The  main goal of this section is to prove Theorem \ref{thm:odd-dim-g-1}, which asserts that framed instanton homology of $0$-surgery on a  nontrivial knot detects whether the knot is fibered. Note that Theorem \ref{thm:odd-dim-g-1}  follows immediately from  Theorem \ref{thm:detect-fibered-knot} if we can prove that the relevant $\C$-modules in the two theorems have the same dimensions, \begin{equation}\label{eqn:isomorphismframed} \dim I^\#_\godd(S^3_0(K),\mu;s_{g-1})= \dim I_*(S^3_0(K)|\hat\Sigma)_\mu,\end{equation} where $g=g(K)=g(\hat\Sigma)$. Recall from \S\ref{ssec:proof} that for each $i\in \Z$, \[s_i:H_2(S^3_0(K))\to 2\Z\] is the homomorphism defined by $s_i([\hat\Sigma]) = 2i.$ So, from \S\ref{ssec:eigenspace-decomposition},   \[I^\#(S^3_0(K),\mu;s_{g-1})\] is, by definition,   the $(2g-2)$-eigenspace of $\mu(\hat\Sigma)$ acting on \begin{equation}\label{eqn:module-top}I^\#(S^3_0(K),\mu) := I_*(S^3_0(K)\#T^3|T^2)_{\mu^\#}.\end{equation} Since the module in \eqref{eqn:module-top} is precisely the $2$-eigenspace of $\mu(\pt)$ acting on \[I_*(S^3_0(K)\#T^3)_{\mu^\#},\] per Remark \ref{rmk:simpletype3}, the top eigenspace of $\mu(\hat\Sigma)$ acting on the module in \eqref{eqn:module-top} is the same as \[I_*(S^3_0(K)\#T^3|\hat\Sigma)_{\mu^\#},\] by Definition \ref{def:top-connected}. In conclusion, we have that \begin{equation}\label{eqn:isomorphism-framed-g-1}I^\#(S^3_0(K),\mu;s_{g-1})= I_*(S^3_0(K)\#T^3|\hat\Sigma)_{\mu^\#}.\end{equation} The equality in \eqref{eqn:isomorphismframed}, which proves Theorem \ref{thm:odd-dim-g-1}, will then follow if we can show that there is an isomorphism \begin{equation}\label{eqn:isomorphism-double} I_*(S^3_0(K)\#T^3|\hat\Sigma)_{\mu^\#}\cong I_*(S^3_0(K)|\hat\Sigma)_{\mu}\otimes (\C_0\oplus \C_1)\end{equation} which is homogeneous with respect to the  $\Z/2\Z$-grading, where $\C_i$ is a copy of $\C$ in grading $i$; indeed,  \eqref{eqn:isomorphism-framed-g-1} and \eqref{eqn:isomorphism-double}  will  imply that\begin{equation*}\label{eqn:dim-equality}\dim I^\#_\godd(S^3_0(K),\mu;s_{g-1})=\dim I^\#_\geven(S^3_0(K),\mu;s_{g-1})=\dim I_*(S^3_0(K)|\hat\Sigma)_\mu.\end{equation*}
Our focus below  will therefore be on proving the isomorphism in \eqref{eqn:isomorphism-double}, though we will work in greater generality for most of this section.

Indeed, rather than considering   $(S^3_0(K),\mu)$ and $\hat\Sigma$, we  let  $(Y,\lambda)$ be an arbitrary admissible pair, and $R\subset Y$ a connnected surface of positive genus with $\lambda\cdot R$ odd. We note exactly as above that
\begin{equation}\label{eqn:Y-pt} I_*(Y\#T^3|T^2)_{\lambda^\#} \end{equation}
is the 2-eigenspace of $\mu(\pt)$ acting on $I_*(Y\#T^3)_{\lambda^\#}$, which means that the top eigenspace of  $\mu(R)$ acting on the module in \eqref{eqn:Y-pt} is simply  
\[ I_*(Y\#T^3|R)_{\lambda^\#}. \] The first of our two main propositions is the following.

\begin{proposition} \label{prop:sum-with-s1xs2}Suppose $(Y,\lambda)$ is an  admissible pair, and $R\subset Y$ is a connected surface of positive genus with $\lambda\cdot R=1$. 
Then there is an isomorphism  
\[ I_*(Y\#T^3|R)_{\lambda^\#} \cong I_*(Y\#(S^1\times S^2)|R)_{\lambda},\]  which is homogeneous with respect to the $\Z/2\Z$-grading, where the connected sum with $S^1\times S^2$ is performed away from  $R$ and $\lambda$.
\end{proposition}

\begin{proof} Recall from the definition of framed instanton homology in \S\ref{ssec:framed-homology} that \[\lambda^\#=\lambda \cup \lambda_T,\] where $\lambda_T=\{\pt\}\times S^1$ is a curve dual to $T^2=T^2\times\{\pt\}$ in  $T^3=T^2\times S^1$.
We can interpret
\[ I_*(Y\#T^3|T^2)_{\lambda\cup \lambda_T} \]
as a version of the sutured instanton homology $\SHI(Y \ssm B^3,S^1)$ defined in \cite{km-excision}, albeit one in which $Y\ssm B^3$ is equipped with the nontrivial bundle specified by $\lambda$.  Fixing a surface $R' \cong R$, it thus follows exactly as in the proof of invariance of $\SHI$ in \cite[\S7.4]{km-excision} that
\begin{equation} \label{eq:sum-with-s1xs2-1}
I_*(Y \# T^3|T^2)_{\lambda\cup\lambda_T} \cong I_*(Y \# (R'\times S^1)|R')_{\lambda\cup\lambda_{R'}},
\end{equation}
where $\lambda_{R'}=\{\pt\}\times S^1$ is a curve dual to $R' = R'\times\{\pt\}$ in  $R' \times S^1$. The isomorphism \eqref{eq:sum-with-s1xs2-1} is obtained as a composition of excision isomorphisms along tori that are disjoint from $R$. It  therefore  intertwines 
the actions of $\mu(R)$  on either side, as in Theorem \ref{thm:excision}.  In particular,  this isomorphism  identifies the  $2j$-eigenspace of $\mu(R)$ acting on $I_*(Y\#T^3|T^2)_{\lambda\cup\lambda_T}$ with the  $2j$-eigenspace of $\mu(R)$ acting on
\[ I_*(Y\#(R'\times S^1)|R')_{\lambda\cup\lambda_{R'}}. \]
In the case $j=g(R)-1$, this isomorphism becomes
\begin{equation} \label{eq:sum-with-s1xs2-2}
I_*(Y\#T^3|R)_{\lambda\cup\lambda_T} \cong I_*(Y\#(R'\times S^1)|R\cup R')_{\lambda\cup\lambda_{R'}}.
\end{equation}
This isomorphism is homogeneous with respect to the $\Z/2\Z$-grading, as in Theorem \ref{thm:excision}.

To prove the proposition, we now apply excision once more. Namely, we cut $Y \# (R' \times S^1)$ open along $R$ and $R' \times \{\pt\}$, and then glue the two resulting $R$ components of the boundary to the two $R'$ components via some identification $R\cong R'$, as illustrated in Figure~\ref{fig:zero-surgery-excision}.
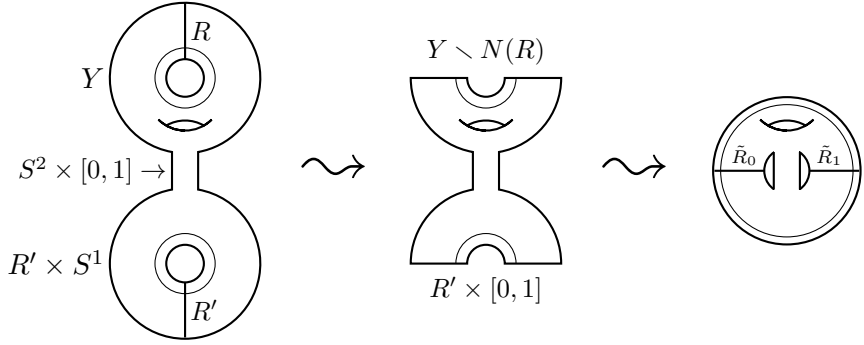
\begin{figure}
\begin{tikzpicture}[every node/.style={inner sep=0,outer sep=0}]
\path (-2,2) arc (90:260:1) node (A) {};
\draw (A) arc (260:-80:1) node (B) {} -- +(0,-0.5) arc (80:-90:1) node (C) {} arc (-90:-260:1) -- cycle node[midway,left,inner sep=2pt] {\small $S^2\times [0,1] \to$};
\draw (-2,1) circle (0.25);
\draw[thin] (-2,1) circle (0.4);
\draw (-2,0.3) arc (270:225:0.5) arc (225:315:0.5) arc (315:300:0.5) arc (60:120:0.5);
\draw (-2,1.25) -- (-2,2) node[midway,right,inner sep=2pt] {\small $R$};
\path (C) -- +(0,1) node (D) {};
\draw (D) circle (0.25);
\draw[thin] (D) circle (0.4);
\draw (C) -- +(0,0.75) node[midway,right,inner sep=2pt] {\small $R'$};
\node[left,inner sep=2pt] at (-3,1) {$Y$};
\node[left,inner sep=2pt] at ($(D)+(-1,0)$) {$R'\times S^1$};

\node (M) at ($(A)+(0,-0.25)$) {};
\node (N) at (0,0) {};
\node at (M -| N) {\Huge $\leadsto$};
\node (P) at (2,0) {};

\draw ($(A)+(4,0)$) arc (260:180:1) -- +(0.75,0) arc (180:360:0.25) -- +(0.75,0) arc (0:-80:1) -- +(0,-0.5) arc (80:0:1) -- +(-0.75,0) node (R) {} arc (0:180:0.25) -- +(-0.75,0) arc(180:100:1) -- cycle;
\draw (2,0.3) arc (270:225:0.5) arc (225:315:0.5) arc (315:300:0.5) arc (60:120:0.5);
\node[above,inner sep=5pt] at (2,1) {\small $Y \ssm N(R)$};
\node[below,inner sep=5pt] at (D -| P) {\small $R'\times [0,1]$};
\draw[thin] ($(2,1)+(0.4,0)$) arc (0:-180:0.4);
\draw[thin] ($(R -| P)+(0.4,0)$) arc (0:180:0.4);

\node (Q) at (4,0) {};
\node at (M -| Q) {\Huge $\leadsto$};

\node (E) at ($0.5*(-2,2)+0.5*(C)+(8,0)$) {};
\draw ($(E)+(0,0.98)$) arc (90:180:0.98) node (F) {} arc(180:360:0.98) node (G) {} arc(0:90:0.98);
\draw (F) -- +(0.68,0) node[pos=0.6,above,inner sep=2pt] {\tiny $\tilde R_0$} (G) -- +(-0.68,0) node[pos=0.6,above,inner sep=2pt] {\tiny $\tilde R_1$};
\draw (6,0.3) arc (270:225:0.5) arc (225:315:0.5) arc (315:300:0.5) arc (60:120:0.5);
\draw ($(A)+(8,0)$) -- +(0,-0.5);
\draw ($(B)+(8,0)$) -- +(0,-0.5);
\draw[thin] (E) circle (0.88);
\begin{scope}
\clip ($(A)+(7.5,0)$) rectangle ($(A)+(8,-0.5)$);
\draw (E) circle (0.3);
\end{scope}
\begin{scope}
\clip ($(B)+(8,0)$) rectangle ($(B)+(8.5,-0.5)$);
\draw (E) circle (0.3);
\end{scope}
\end{tikzpicture}
\caption{Relating $Y \# (R'\times S^1)$ to $Y \# (S^1\times S^2)$ by excision.  The thin curves are meant to represent $\lambda\cup\lambda_{R'}$ on the left, and $\lambda$ on the right.}
\label{fig:zero-surgery-excision}
\end{figure} We can assume this identification is such that the multicurve $\lambda \cup \lambda_T$ is cut and reglued to form $\lambda$ in the resulting manifold, which is $Y \# (S^1\times S^2)$. By Theorem~\ref{thm:excision}, there is an isomorphism
\begin{equation} \label{eq:sum-with-s1xs2-3}
I_*(Y\#(R'\times S^1)|R\cup R')_{\lambda\cup\lambda_{R'}} \cong I_*(Y\#(S^1\times S^2)|\tilde R_0 \cup \tilde R_1)_{\lambda},
\end{equation}
which is homogeneous with respect to the $\Z/2\Z$-grading, 
in which $\tilde R_0$ and $\tilde R_1$ are the images of $R$ and $R'$. 
We may identify one of these two surfaces---say $\tilde R_0$---with $R\subset Y\#(S^1\times S^2).$ Furthermore, note that $\tilde R_0$ and $\tilde R_1$ cobound a region in $Y \# (S^1\times S^2)$ with one of the $2$-spheres in the connected sum neck. Adding a trivial handle to this $2$-sphere, we obtain a torus $T$ in $Y \# (S^1\times S^2)$ such that \[[\tilde R_1] - [\tilde R_0] = [T]\] in homology. Since $\mu(T)$ is nilpotent, a generalized $(2g(R)-2)$-eigenvector for $\tilde R_0$ is also a generalized  $(2g(R)-2)$-eigenvector for $\tilde R_1$, and vice versa. We therefore have that \begin{equation} \label{eq:sum-with-s1xs2-4}
I_*(Y\#(S^1\times S^2)|\tilde R_0 \cup \tilde R_1)_{\lambda} \cong I_*(Y\#(S^1\times S^2)|\tilde R_0)_{\lambda}= I_*(Y\#(S^1\times S^2)|R)_{\lambda}.
\end{equation}
Combining \eqref{eq:sum-with-s1xs2-2}, \eqref{eq:sum-with-s1xs2-3}, and \eqref{eq:sum-with-s1xs2-4} then completes the proof of the proposition.
\end{proof}

Our second main proposition of this section is the following.
\begin{proposition} \label{prop:s1xs2-surgery-triangle} Suppose $(Y,\lambda)$ is an  admissible pair, and $R\subset Y$ is a connected surface of positive genus with $\lambda\cdot R$ odd. 
Then there is an isomorphism of $\Z/2\Z$-graded $\C$-modules
\[ I_*(Y\#(S^1\times S^2)|R)_{\lambda} \cong I_*(Y|R)_\lambda \otimes (\C_0 \oplus \C_1), \]
where  each $\C_i$ is a copy of $\C$ in grading $i$.
\end{proposition}

\begin{proof}
We model our argument after the computation of $I^\#(S^1\times S^2)$ in \cite[\S7.6]{scaduto}.  Let $U$ denote an unknot in $Y$. We apply Floer's surgery exact triangle,   Theorem \ref{thm:exact-triangle-main}, for surgeries on $U$ with framings $\infty$, $0$, and $1$ to get an exact triangle
\begin{equation*} \label{eq:sum-s1xs2-triangle}
\dots \to I_*(Y)_\lambda \xrightarrow{F} I_*(Y\#(S^1\times S^2))_{\lambda} \xrightarrow{G} I_*(Y)_\lambda \xrightarrow{H} \dots
\end{equation*}
(where we take the $\lambda$ in that theorem to be the union of our $\lambda$ with the $\mu$ in the theorem).
The maps $F$, $G$, and $H$ are induced by $2$-handle cobordisms where the $2$-handles are attached away from $R$, so these maps intertwine  the actions of $\mu(R)$ on these modules.  The exact triangle above therefore restricts to an exact triangle \begin{equation} \label{eq:sum-s1xs2-triangle-R}
\dots \to I_*(Y|R)_\lambda \xrightarrow{F} I_*(Y\#(S^1\times S^2)|R)_{\lambda} \xrightarrow{G} I_*(Y|R)_\lambda \xrightarrow{H} \dots
\end{equation}
Since $F$ has odd degree while $G$ has even degree, by the formula \eqref{eqn:grading-shift}, the proposition will follow if we can show that $H=0$.

The map $H$ is induced by a cobordism
\[ (X,\nu): (Y,\lambda) \to (Y,\lambda) \]
 built by attaching a $(-1)$-framed 2-handle to $Y \times [0,1]$ along the unknot $U\times \{1\}$. In other words,
\[ X \cong (Y \times [0,1]) \# \cptwo. \]
Letting $e \in \cptwo$ denote the exceptional sphere, the pairing $\nu\cdot e$ was determined to be odd in \cite[\S7.6]{scaduto}.  But then the induced cobordism map $H$ must be zero by a dimension-counting argument, exactly as in \cite[\S 4]{braam-donaldson}.  Namely, by neck-stretching along $S^3$, we can realize an instanton $A$ on $X$ by gluing instantons $A_{Y\times[0,1]}$ and $A_{\cptwo}$ on either summand.  The former are irreducible and the latter have stabilizer at most $S^1$, since the bundle specified by $\nu$ is nontrivial on $\cptwo$; and the unique flat connection on $S^3$ has 3-dimensional stabilizer, so we have
\[ \ind(A) = \ind(A_{Y\times[0,1]}) + \ind(A_{\cptwo}) + 3 \geq 0 + (-1) + 3 = 2. \]
Thus, there are generically no index-0 instantons on $X$, which implies that $H$ is zero.
\end{proof}

Combining Propositions~\ref{prop:sum-with-s1xs2} and \ref{prop:s1xs2-surgery-triangle}, we have the following immediate corollary.

\begin{corollary}\label{cor:odd-even-top}
Suppose $(Y,\lambda)$ is an  admissible pair, and $R\subset Y$ is a connected surface of positive genus with $\lambda\cdot R=1$. Then  there is an isomorphism of $\C$-modules, \[I_*(Y\#T^3|R)_{\lambda^\#}\cong  I_*(Y|R)_\lambda\otimes (\C_0\oplus\C_1),\]  which is homogeneous with respect to the $\Z/2\Z$-grading.
\end{corollary}

\begin{proof}[Proof of Theorem \ref{thm:odd-dim-g-1}]Applying Corollary \ref{cor:odd-even-top} to the case in which $Y=S^3_0(K)$, $R=\hat\Sigma$, and $\lambda = \mu$, we obtain the isomorphism \eqref{eqn:isomorphism-double} from which Theorem \ref{thm:odd-dim-g-1} follows, as discussed at the beginning of this section.
\end{proof}

With respect to the notation  \[I^\#(Y,\lambda|R):=I_*(Y\#T^3|R)_{\lambda^\#}\] of  \eqref{eqn:Irelative}, Corollary \ref{cor:odd-even-top} together with Theorem \ref{thm:fibered-has-rank-1} immediately imply the following.

\begin{proposition}\label{prop:fiberedframed}
Let $(Y,\lambda)$ be an admissible pair such that $Y$ fibers over the circle with fiber a  connected surface $R$ of positive genus with $\lambda\cdot R = 1$. Then \[I^\#_\geven(Y,\lambda|R) \cong I^\#_\godd(Y,\lambda|R) \cong  \C.\] 
\end{proposition}

\section{A structure theorem for cobordism maps} \label{sec:I-sharp}

Following \cite{km-structure}, we let $x=[\pt]$ in $H_0(X;\Z)$ and then define a formal power series by the formula
\[ \sD_X^\nu(h) = D_{X,\nu}\left(-\otimes\left(e^h + \tfrac{1}{2}xe^h\right)\right) \]
for each $h \in H_2(X;\R)$.  Our goal over the next several subsections is to prove the following structure theorem for the formal cobordism maps $\sD_X^\nu$.

\begin{theorem} \label{thm:structure-theorem}
Let $(X,\nu): (Y_0,\lambda_0) \to (Y_1,\lambda_1)$ be a cobordism with $b_1(X)=0$.  Then there is a finite collection (possibly empty) of \emph{basic classes}
\[ K_1,\dots,K_r: H_2(X;\Z) \to \Z, \]
each satisfying $K_j(\alpha) \equiv \alpha\cdot\alpha \pmod{2}$ for all $\alpha \in H_2(X;\Z)$; and nonzero elements
\[ a_1,\dots,a_r \in \Hom(I^\#(Y_0,\lambda_0), I^\#(Y_1,\lambda_1)) \]
with rational coefficients\footnote{Ultimately, framed instanton homology can be defined over $\Z$; what we mean here is that the $a_j$ have rational coefficients with respect to  rational bases for $I^\#(Y_0,\lambda_0)$ and $I^\#(Y_1,\lambda_1)$. We  remark that the results of this section hold over any field of characteristic zero.}, such that for any $\alpha \in H_2(X;\Z)$ we have
\[ \sD_X^{\nu+\alpha}(h) = e^{Q(h)/2} \sum_{j=1}^r (-1)^{\frac{1}{2}(K_j(\alpha)+\alpha\cdot\alpha)+\nu\cdot\alpha} \cdot a_j e^{K_j(h)}. \]
The basic classes satisfy an adjunction inequality
\[ |K_j([\Sigma])| + [\Sigma] \cdot [\Sigma] \leq 2g(\Sigma)-2 \]
for all smoothly embedded surfaces $\Sigma \subset X$ of genus $g(\Sigma) \geq 1$ and positive self-intersection.
\end{theorem}

Theorem~\ref{thm:structure-theorem} is a direct analogue of Kronheimer and Mrowka's structure theorem for the Donaldson invariants of closed 4-manifolds of simple type \cite{km-structure}, which was subsequently given different proofs by Fintushel and Stern \cite{fs-structure} and Mu\~noz \cite{munoz-basic}.  None of these proofs applies verbatim to an arbitrary cobordism: the proofs in \cite{km-structure,munoz-basic} require a smoothly embedded surface of positive self-intersection, while the one in \cite{fs-structure} requires $\pi_1(X) = 1$ and involves interaction between both $SO(3)$ and $SU(2)$ invariants.  We will prove Theorem \ref{thm:structure-theorem} by adapting Mu\~noz's proof, which is short and only requires two ingredients: a blow-up formula and a computation of Fukaya--Floer homology for products $\Sigma \times S^1$.  (In this paper we do not need to consider cobordisms where $\pi_1(X) \neq 1$, but we choose not to adapt Fintushel--Stern's proof because the greater generality will be useful elsewhere.)

Our proof of Theorem~\ref{thm:structure-theorem} is structured as follows.  We first verify  in \S\ref{ssec:blow-up} that Fintushel and Stern's blow-up formula \cite{fs-blowup} holds for  cobordism maps.  In \S\ref{ssec:structure-positive} we use this blow-up formula together with a modification of Mu\~noz's proof of the structure theorem in \cite{munoz-basic} to prove the structure theorem above in the case $b_2^+(X) > 0$.  In \S\ref{ssec:injective-cobordism}, we prove that the $2$-handle cobordism given by the  trace of $1$-surgery on the torus knot $T_{2,5}$ induces an injective map on $I^\#$, and  we use this in \S\ref{ssec:structure-general} to reduce the general case of Theorem~\ref{thm:structure-theorem} to the case where $b_2^+(X)$ is positive.

\begin{remark}\label{rmk:simpletypeb2}
Some of the results  which we  adapt from the closed $4$-manifold setting (like Fintushel--Stern's blow-up formula) require that the closed manifold have $b_2^+>0$ or $b_2^+>1$. These are  to rule out reducibles and to ensure that the Donaldson invariants are well-defined, respectively. The framed construction makes these constraints unnecessary for the cobordism maps studied here. Indeed, any ASD connection on a cobordism $X^\#$ limits at the ends $Y_i\#T^3$ to a flat connection on an admissible bundle, so it is automatically irreducible at the ends and hence on $X^\#$ itself. Furthermore, these maps are well-defined without any assumption on $b_2^+$. 

Some of these results also require the simple type assumption, but this is  automatically satisfied for the cobordisms $X^\#$ by Lemma \ref{lem:simple-type}.
\end{remark}

\subsection{The blow-up formula for cobordism maps} \label{ssec:blow-up}

The blow-up formula for  formal cobordism maps reads as follows.  
\begin{theorem} \label{thm:blowup-formula}
Let $(X,\nu): (Y_0,\lambda_0) \to (Y_1,\lambda_1)$ be a cobordism, and let $\tilde{X} = X \# \cptwo$ denote the blow-up of $X$ at a point.  Let $E$ denote the Poincar\'e dual of the class $e \in H_2(\tilde{X};\Z)$ of an exceptional sphere.  Then
\begin{align*}
\sD^\nu_{\tilde{X}} &= \sD^\nu_X \cdot e^{-E^2/2}\cosh E \\
\sD^{\nu+e}_{\tilde{X}} &= -\sD^\nu_X \cdot e^{-E^2/2}\sinh E
\end{align*}
as formal $\Hom(I^\#(Y_0,\lambda_0),I^\#(Y_1,\lambda_1)) $-valued functions on $H_2(X;\R)$.
\end{theorem}

This was proven by Fintushel and Stern for the Donaldson invariants of closed $4$-manifolds in \cite{fs-blowup}. We explain briefly  why their argument carries over  to the case of cobordisms.

\begin{proof}[Proof of Theorem \ref{thm:blowup-formula}]
The blow-up formula of \cite{fs-blowup} is a formal consequence of a handful of concrete identities given in \cite[\S2]{fs-blowup}, which constrain the behavior of $D_{X,\nu}$ in the presence of a sphere of self-intersection $-2$ or $-3$, and which relate $D_{X,\nu}$ to $D_{\tilde{X},\nu}$ and $D_{\tilde{X},\nu+e}$.  These identities were originally proven in the setting of closed $4$-manifolds but are  local, obtained via an analysis of ASD connections in neighborhoods of such spheres, and therefore hold for cobordism maps, so the blow-up formula remains valid in this setting.  The formulas here match the special case \cite[Theorem~5.2]{fs-blowup}, whose proof requires only the assumption that the closed $4$-manifold have simple type, which is automatically satisfied by the cobordisms $X^\#$ as noted in Remark \ref{rmk:simpletypeb2} and proved in  Lemma~\ref{lem:simple-type}.\end{proof}

\subsection{The structure theorem for $b_2^+$ positive} \label{ssec:structure-positive}

In this subsection we prove the following.

\begin{proposition} \label{prop:structure-thm-positive}
Theorem~\ref{thm:structure-theorem} holds for cobordisms $X$ with $b_2^+(X) > 0$.
\end{proposition}

Our proof of Proposition~\ref{prop:structure-thm-positive} follows (with some minor adjustments) Mu\~noz's proof in \cite[\S3]{munoz-basic} of \cite[Theorem~1.2]{munoz-basic}, a general structure theorem for closed 4-manifolds $X$ with strong simple type.  The key input in this proof is the following lemma {\cite[Lemma~3.1]{munoz-basic}}. 

\begin{lemma} \label{lem:munoz-hff}
Let $(X,\nu): (Y_0,\lambda_0) \to (Y_1,\lambda_1)$ be a cobordism with $b_1(X)=0$. Let $\Sigma \subset X$ be a surface of genus $g \geq 1$, with $\Sigma\cdot\Sigma=0$ and $\nu\cdot\Sigma$ odd.  Then for any other class $D \in H_2(X)$ there are power series $f_{r,D}(t)$, with $1-g \leq r \leq g-1$, such that
\begin{align*}
\sD_X^\nu(tD+s\Sigma) &= e^{Q(tD+s\Sigma)/2}\sum_{r=1-g}^{g-1} f_{r,D}(t)e^{2rs} \\
\sD_X^{\nu+\Sigma}(tD+s\Sigma) &= e^{Q(tD+s\Sigma)/2}\sum_{r=1-g}^{g-1} (-1)^{r+1}f_{r,D}(t)e^{2rs}.
\end{align*}
\end{lemma}

\begin{remark} \label{rem:munoz-3.1}
Mu\~noz states this lemma for closed $4$-manifolds in \cite[Lemma~3.1]{munoz-basic}, but it applies verbatim to the cobordism maps here. Indeed, the proof only uses the relations imposed by the relative invariants of a $\Sigma\times D^2$ neighborhood of $\Sigma\subset X$ in the Fukaya--Floer homology $\mathit{HFF}(\Sigma\times S^1,S^1)$. In particular, Mu\~noz uses his computation of this Fukaya--Floer homology from \cite{munoz-hff} to deduce that both series are annihilated by the operator
\[ \prod_{r=1-g}^{g-1} \left(\frac{\partial}{\partial s} - (2r+t(D\cdot\Sigma))\right), \] from which the lemma follows.

Mu\~noz's assumption of \emph{strong simple type} means that for any $\nu \subset X$ and $z \in \bA(X)$, the Donaldson invariants satisfy $D_{X,\nu}((x^2-4)z)=0$ and $D_{X,\nu}(\gamma z)=0$ for all $\gamma \in H_1(X)$.  In our situation, the first of these is Lemma~\ref{lem:simple-type}.  We avoid the second by requiring that $b_1(X)=0$: Mu\~noz needs strong simple type to assert that for $\Sigma$ as in the statement of the lemma, if we let $X_1 = X\setminus N(\Sigma)$ and $D_1 = D \cap X_1$, then the relative invariant
\[ \phi^\nu(X_1, e^{tD_1}) \in \mathit{HFF}(\Sigma\times S^1,S^1) \]
lies in the kernel of various elements denoted $\psi_i = \phi^\nu(N(\Sigma),\gamma_i e^{t\Delta})$, where $\gamma_i \in H_1(\Sigma)$.  As argued by Mu\~noz, this assertion follows from knowing that
\[ D_{X,\nu}(\gamma_i ze^{tD}) = D_{X,\nu+\Sigma}(\gamma_i ze^{tD}) = 0 \]
for all $z\in\bA(X)$.  But the condition $b_1(X)=0$ makes this automatic because $\gamma_i$ is nullhomologous in $X$, hence also in the larger cobordism $X^\#$ used to define $D_{X,\nu}$.  We should note that $X^\#$ may not itself have strong simple type since $H_1(X^\#) \neq 0$, but this does not matter because we are restricting our attention to surfaces which lie in $X$.
\end{remark}

We now prove Proposition~\ref{prop:structure-thm-positive} via a series of lemmas, corresponding to steps 2 through 5 of \cite[\S 3]{munoz-basic}.  (Step 1 is omitted because it establishes the strong simple type property, which is not necessary here because $H_1(X)=0$.)  Each of these is labeled with the corresponding step from \cite{munoz-basic}, and we indicate explicitly how our proofs differ from those in \cite{munoz-basic}.

Throughout this subsection we assume that
\[ (X,\nu): (Y_0,\lambda_0) \to (Y_1,\lambda_1) \]
is a cobordism with $b_1(X)=0$ and $b_2^+(X) > 0$.

\begin{lemma}[Step~2] \label{lem:structure-thm-for-nu}
For any cobordism $\nu: \lambda_0 \to \lambda_1$ in $X$, there are finitely many nonzero 
\[ a_{\nu,i} \in \Hom(I^\#(Y_0,\lambda_0), I^\#(Y_1,\lambda_1)) \]
and homomorphisms $K_{\nu,i}: H_2(X;\Z) \to \Z$ such that
\[ \sD_X^\nu(h) = e^{Q(h)/2} \sum_i a_{\nu,i} e^{K_{\nu,i}(h)} \]
for all $h \in H_2(X;\Z)$.
\end{lemma}

As originally stated, \cite[Step~2]{munoz-basic} only proves the corresponding property for a single $\nu$, but we will establish it for all $\nu$ simultaneously.

\begin{proof}[Proof of Lemma~\ref{lem:structure-thm-for-nu}]
Just as in \cite{munoz-basic}, it suffices to prove this for some blow-up of $X$ and some $\tilde\nu$ which is homologous to a sum of $\nu$ and some exceptional spheres, since if $\tilde{X}=X\#\CP^2$ has exceptional divisor $e$ then we have
\[
\sD_X^{\nu}(h) = \sD_{\tilde X}^{\nu}(h) \textrm{ and } \sD_X^{\nu}(h) = \left.-\tfrac{d}{dr}\sD_{\tilde X}^{\nu+e}(h+re)\right|_{r=0}
\]
by the blow-up formula of Theorem~\ref{thm:blowup-formula}.

We begin by finding a convenient basis of $H_2(X)$; the construction in \cite{munoz-basic} uses the fact that in the closed case, after blowing up we can arrange $Q_X = a(1) \oplus b(-1)$, but this is no longer true for $X$ a cobordism.  Instead, using the assumption that $b_2^+(X) > 0$, we let $\Sigma_1,\dots,\Sigma_k$ be any integral basis of $H_2(X)/\torsion$ with $\Sigma_1\cdot\Sigma_1 > 0$.  For $j \geq 2$, we replace each $\Sigma_j$ with $\Sigma_j + n\Sigma_1$ for a suitably large $n$, and then we still have an integral basis but now $\Sigma_j\cdot\Sigma_j > 0$ for all $j \geq 2$ as well.

Having done so, we take $n_i = \Sigma_i \cdot \Sigma_i$ for $1 \leq i \leq k$ and let $\tilde{X}$ be the blow-up of $X$ at $\sum_{i=1}^k n_i$ points, with
\[ e_{i,j} \textrm{ for } 1 \leq i \leq k,\ 1 \leq j \leq n_i \]
being the exceptional spheres.  For $1 \leq i \leq k$ we define a collection of homology classes by 
\begin{align*}
S_i^0 &= \Sigma_i - e_{i,1} - e_{i,2} - \dots - e_{i,n_i} \\
S_i^j &= \Sigma_i - e_{i,1} - \dots + e_{i,j} - \dots - e_{i,n_i}, \qquad 1 \leq j \leq n_i.
\end{align*}
These are integral classes of square zero which span $H_2(\tilde{X};\Q)$.  In fact, if we let
\[ H = \mathrm{span}_\Z(\{S_i^j \mid 1\leq i\leq k,\ 0\leq j\leq n_i \}) \subset H_2(\tilde{X};\Z)/\torsion, \]
then from the relations $2e_{i,j} = S_i^j - S_i^0$ for all $j \geq 1$, and
\[ 2\Sigma_i = 2S_i^0 + 2\sum_{j=1}^{n_i} e_{i,j} = (2-n_i)S_i^0 + 2 \sum_{j=1}^{n_i} S_i^j, \]
we deduce that 
\[ 2H_2(\tilde{X};\Z)/\torsion \subset H. \]
We also define a class $\tilde{\nu} \in H_2(\tilde{X},\partial\tilde{X})$ by 
\[
\tilde{\nu} = \nu + \sum_{i=1}^{k} c_i e_{i,1}, \textrm{ where }
c_i = \begin{cases} 0 & \nu\cdot\Sigma_i\mathrm{\ odd} \\ 1 & \nu\cdot\Sigma_i\mathrm{\ even}, \end{cases}
\]
so that for all $i$ and $j$ we have $\tilde{\nu} \cdot S_i^j \equiv 1 \pmod{2}$.

Now if each $S_i^j$ is represented by a smoothly embedded surface of genus $g_{i,j} \geq 1$, then just as in \cite{munoz-basic}, repeated application of Lemma~\ref{lem:munoz-hff} says that
\begin{equation} \label{eq:structure-for-nu-sum}
\sD_{\tilde{X}}^{\tilde{\nu}}\left(\sum_{i=1}^k\sum_{j=0}^{n_i} t_{i,j}S_i^j\right) = e^{Q(\sum t_{i,j}S_i^j)/2} \sum_{1-g_{i,j} \leq r_{i,j} \leq g_{i,j}-1} a_{\{r_{i,j}\}}e^{\sum 2r_{i,j}t_{i,j}}
\end{equation}
for some \[a_{\{r_{i,j}\}}: I^\#(Y_0,\lambda_0) \to I^\#(Y_1,\lambda_1).\]  The exponents $\sum 2r_{i,j} t_{i,j}$ are valued in $2\Z$ on all of $H$ and in particular on $2H_2(\tilde{X};\Z)/\torsion$, so they define homomorphisms $H_2(\tilde{X};\Z) \to \Z$.  Moreover, the $a_{\{r_{i,j}\}}$ can in fact be taken to have rational coefficients, because the same is true for each $D_{\tilde{X},\tilde{\nu}}(z)$ where $z$ has the form $h^d$ or $\frac{x}{2}h^d$ with $h \in H_2(\tilde{X};\Z)$.
\end{proof}

\begin{lemma}[Step~3] \label{lem:basic-classes-w2}
The basic classes $K_{\nu,\ell}$ of Lemma~\ref{lem:structure-thm-for-nu} satisfy
\[ K_{\nu,\ell}(h) + h\cdot h \equiv 0 \pmod{2} \]
for all $h \in H_2(X;\Z)$.
\end{lemma}

Our argument here differs from the one in \cite{munoz-basic}, because that one asserts that for any $x \in H_2(\tilde{X})$ there are some $i$ and $j$ such that $x \cdot S_i^j \neq 0$, and this need not be true when $X$ is a cobordism rather than a closed 4-manifold.

\begin{proof}[Proof of Lemma~\ref{lem:basic-classes-w2}]
Both terms are linear mod $2$ in $h$, so it suffices to prove the claim for the basis $\{\Sigma_1,\dots,\Sigma_k\}$ of $H_2(X;\Z)/\torsion$ which we used in the proof of Lemma~\ref{lem:structure-thm-for-nu}.  We borrow notation from that proof, blowing up $X$ to get $\tilde{X}$ and surfaces $S_i^j$ exactly as before.

We observe for each $i$ that if $K$ is any basic class of $(\tilde{X},\tilde{\nu})$, then equation~\eqref{eq:structure-for-nu-sum} says that $K(S_i^0)$ is even.  Letting $E_{i,j}$ denote the Poincar\'e duals of the exceptional spheres $e_{i,j}$ in $\tilde{X}$, the blow-up formula (Theorem~\ref{thm:blowup-formula}) says that
\[ K = K_{\nu,\ell} + \sum_{i=1}^k \sum_{j=1}^{n_i} E_{i,j} \]
is a basic class of $(\tilde{X},\tilde{\nu})$.  Since $S_i^0 = \Sigma_i - \sum_{j=1}^{n_i} e_{i,j}$ for all $i$, we have
\[ K(S_i^0) = K_{\nu,\ell}(\Sigma_i) + n_i = K_{\nu,\ell}(\Sigma_i) + \Sigma_i \cdot \Sigma_i \]
and hence this last expression is even for all $i$, as desired.
\end{proof}

\begin{lemma}[Step~4] \label{lem:structure-thm-signs}
Fix $\nu$ and let $a_{\nu,i}$ and $K_{\nu,i}$ be the coefficients and basic classes appearing in Lemma~\ref{lem:structure-thm-for-nu}.  For any class $\alpha \in H_2(X;\Z)$, we have
\[ \sD_X^{\nu+\alpha}(h) = e^{Q(h)/2} \sum_i (-1)^{\frac{1}{2}(K_{\nu,i}(\alpha)+\alpha\cdot \alpha)+\nu\cdot \alpha} \cdot a_{\nu,i} e^{K_{\nu,i}(h)}. \]
In other words, the basic classes $K_i = K_{\nu,i}$ do not depend on the particular choice of $\nu$, and their coefficients are related by
\[ a_{\nu+\alpha,i} = (-1)^{\frac{1}{2}(K_{\nu,i}(\alpha)+\alpha\cdot \alpha)+\nu\cdot \alpha} \cdot a_{\nu,i} \]
for all $i$.
\end{lemma}

Our argument here is nearly the same as in \cite{munoz-basic}, but with a different choice of signs in the definition of $\Sigma$ which avoids the need for the identity $D_{X,\nu+2\alpha} = (-1)^{\alpha^2}D_{X,\nu}$.

\begin{proof}[Proof of Lemma~\ref{lem:structure-thm-signs}]
We note that the exponent \[\frac{1}{2}(K_{\nu,i}(\alpha) + \alpha\cdot\alpha) + \nu\cdot\alpha\] is integral by Lemma~\ref{lem:basic-classes-w2}, and its reduction modulo 2 is linear in $\alpha$.  Since $H_2(X;\Z)/\torsion$ has an integral basis consisting of surfaces of positive self-intersection, as seen in the proof of Lemma~\ref{lem:structure-thm-for-nu}, it therefore suffices to prove the proposition when $N = \alpha\cdot \alpha$ is positive.  We will let $\tilde{X}$ be the $N$-fold blow-up of $X$, with exceptional spheres $e_1,\dots,e_N$ and $E_j = \PD(e_j)$ as usual.

Suppose first that $\nu\cdot \alpha$ is odd.  Letting $\Sigma = \alpha + \sum_j e_j$, so that $\nu\cdot\Sigma$ is odd and $\Sigma\cdot\Sigma = 0$, we have
\begin{align*}
\sD_{\tilde{X}}^{\nu} &= \sD_X^\nu \cdot e^{-\sum_j E_j^2/2} \prod_{j=1}^N \cosh E_j \\
\sD_{\tilde{X}}^{\nu+\Sigma} &= \sD_X^{\nu+\alpha} \cdot e^{-\sum_j E_j^2/2} \prod_{j=1}^N (-\sinh E_j)
\end{align*}
by the blow-up formula.  Lemma~\ref{lem:munoz-hff} tells us that $(\tilde{X},\nu)$ and $(\tilde{X},\nu+\Sigma)$ have the same basic classes, and these are
\[ K_{\nu,i} + \sum_{j=1}^N \sigma_j E_j \quad\mathrm{and}\quad K_{\nu+\alpha,i} + \sum_{j=1}^N \sigma_j E_j, \qquad \sigma_j \in \{\pm 1\} \]
respectively, so $(X,\nu)$ and $(X,\nu+\alpha)$ must have the same basic classes as well, say $K_{\nu,i}=K_{\nu+\alpha,i}$.  Moreover, the lemma says that the coefficient of $K = K_{\nu,i} + \sum_j E_j$ in $\sD_{\tilde{X}}^{\nu+\Sigma}$ is
\[ \frac{a_{\nu,i}}{2^N} \cdot (-1)^{\frac{1}{2}K(\Sigma) + 1} = \frac{a_{\nu,i}}{2^N} \cdot (-1)^{\frac{1}{2}(K_{\nu,i}(\alpha) - N) + 1}, \]
while by the blow-up formula it is $(-1)^N \cdot \frac{a_{\nu+\alpha,i}}{2^N}$.  Equating the two, and recalling that $N = \alpha\cdot\alpha$ and $\nu\cdot\alpha \equiv 1 \pmod{2}$, gives
\[ a_{\nu+\alpha,i} = \alpha_{\nu,i} \cdot (-1)^{\frac{1}{2}(K_{\nu,i}(\alpha) + \alpha\cdot\alpha) + \nu\cdot\alpha} \]
as claimed.

Now suppose instead that $\nu\cdot\alpha$ is even.  This time we take $\Sigma = \alpha - e_1 + \sum_{j=2}^N e_j$, so that $\Sigma\cdot\Sigma=0$ and $(\nu+e_1)\cdot\Sigma$ is odd.  Then the blow-up formula gives
\begin{align*}
\sD_{\tilde{X}}^{\nu+e_1} &= \sD_X^{\nu} \cdot e^{-\sum_j E_j^2/2}(-\sinh E_1)\prod_{j=2}^N \cosh E_j \\
\sD_{\tilde{X}}^{\nu+e_1+\Sigma} &= \sD_X^{\nu+\alpha} \cdot e^{-\sum_j E_j^2/2} (\cosh E_1) \prod_{j=2}^N (-\sinh E_j).
\end{align*}
Applying Lemma~\ref{lem:munoz-hff} again, we conclude exactly as before that $(X,\nu)$ and $(X,\nu+\alpha)$ have the same basic classes $K_{\nu,i} = K_{\nu+\alpha,i}$; and since $K=K_{\nu,i}+\sum_j E_j$ has coefficient $-a_{\nu,i}/2^N$ in $\sD_{\tilde{X}}^{\nu+e_1}$, the lemma also says that its coefficient in $\sD_{\tilde{X}}^{\nu+e_1+\Sigma}$ is
\[ -\frac{a_{\nu,i}}{2^N} \cdot (-1)^{\frac{1}{2}K(\Sigma) + 1} = 
\frac{a_{\nu,i}}{2^N} \cdot (-1)^{\frac{1}{2}(K_{\nu,i}(\alpha) - N+2)}, \]
while the blow-up formula gives this coefficient as $(-1)^{N-1} \cdot \frac{a_{\nu+\alpha,i}}{2^N}$.  Equating the two and using $N=\alpha\cdot\alpha$ and $\nu\cdot\alpha\equiv 0 \pmod{2}$ gives the desired conclusion.
\end{proof}

The last step is proved exactly as in \cite{munoz-basic}; we include the proof anyway for completeness.

\begin{lemma}[Step~5] \label{lem:adjunction}
Let $(X,\nu): (Y_0,\lambda_0) \to (Y_1,\lambda_1)$ be a cobordism with $b_1(X)=0$, and let $\Sigma \subset X$ be a smoothly embedded surface of genus $g \geq 1$ and positive self-intersection.  Then
\[ |K(\Sigma)| + \Sigma\cdot\Sigma \leq 2g(\Sigma)-2 \]
for every basic class $K$ of $X$.
\end{lemma}

\begin{proof}
Let $\tilde{X}$ be the $N$-fold blow-up of $X$ with exceptional spheres $e_1,\dots,e_N$, where $N = \Sigma \cdot \Sigma > 0$; and let $\tilde{\Sigma} \subset \tilde{X}$ denote the proper transform $\Sigma - \sum_{j=1}^N e_j$.  We let $\tilde{\nu}$ be whichever of $\nu$ and $\nu+e_1$ satisfies $\tilde\nu\cdot\tilde\Sigma \equiv 1 \pmod{2}$.  Then $\tilde{\Sigma}\cdot\tilde{\Sigma} = 0$, so we can apply Lemma~\ref{lem:munoz-hff} to $\sD_{\tilde{X}}^{\tilde{\nu}}$ and $\tilde\Sigma$ to see that 
\[ |\tilde{K}(\tilde{\Sigma})| \leq 2g-2 \]
for every basic class $\tilde{K}$ of $\tilde{X}$.  These basic classes include $K \pm \sum_{j=1}^N E_j$, and since
\[ \left(K\pm \sum_{j=1}^N E_j\right)(\tilde{\Sigma}) = K(\Sigma) \pm N = K(\Sigma) \pm \Sigma \cdot \Sigma \]
we must have $|K(\Sigma)| + \Sigma\cdot\Sigma \leq 2g-2$ as claimed.
\end{proof}

Lemmas~\ref{lem:structure-thm-for-nu}, \ref{lem:basic-classes-w2}, \ref{lem:structure-thm-signs}, and \ref{lem:adjunction} collectively prove Proposition~\ref{prop:structure-thm-positive}. \hfill\qed

\subsection{Injective maps induced by trace cobordisms} \label{ssec:injective-cobordism} Recall that the trace of $n$-surgery on a framed knot $K\subset S^3$ is the cobordism \[X_n(K):S^3\to S^3_n(K)\] obtained from  $S^3 \times [0,1]$ by attaching an $n$-framed $2$-handle along $K\times\{1\}$.
In this section, we   prove   that the map on framed instanton homology induced by the trace of $1$-surgery on the torus knot $T_{2,5}$ is injective; see Proposition \ref{prop:x1-t25-injective}. We will use this in the next section to establish the general case of Theorem~\ref{thm:structure-theorem}, for cobordisms with $b_2^+(X)=0$.  

We start with some lemmas about the dimension of $I^\#$ for  surgeries on $T_{2,5}$.

\begin{lemma} \label{lem:t25-1}
$\dim I^\#(S^3_k(T_{2,5})) = k$ for all $k \geq 3$.
\end{lemma}

\begin{proof}
Observe that $9$-surgery on $T_{2,5}$ is a lens space of order 9 \cite{moser}, and hence
\[ \dim I^\#(S^3_9(T_{2,5})) = 9. \]
Using \eqref{eq:triangle-untwisted}, we have an exact triangle of the form
\[ \dots \to I^\#(S^3) \xrightarrow{I^\#(X_k(T_{2,5}),\nu_k)} I^\#(S^3_k(T_{2,5})) \to I^\#(S^3_{k+1}(T_{2,5})) \to \dots \]
for all $k \geq 1$.  Each $X_k(T_{2,5})$ contains a surface $\Sigma_k$ of genus $g(T_{2,5})=2$ and self-intersection $k$, built by gluing a Seifert surface for $T_{2,5}$ to a core of the $k$-framed $2$-handle, and when $k > 2$ then this forces $I^\#(X_k(T_{2,5}),\nu_k)$ to vanish by the adjunction inequality of Lemma~\ref{lem:adjunction}.  Thus
\[ \dim I^\#(S^3_{k+1}(T_{2,5})) = \dim I^\#(S^3_k(T_{2,5})) + 1 \]
for all $k \geq 3$, and the lemma follows by induction.
\end{proof}

\begin{lemma} \label{lem:t25-2}
$\dim I^\#(S^3_1(T_{2,5})) = 5$.
\end{lemma}

\begin{proof}
We first note that
\[ S^3_1(T_{2,5}) \cong -\Sigma(2,5,9). \]
Since the dimension of $I^\#$ does not change upon reversing orientation, it suffices to prove that $I^\#(\Sigma(2,5,9))$ has dimension 5. 

To do so, we use Fintushel and Stern's computation from \cite[\S6]{fs-seifert} that, with coefficients in $\Z$,
\[ I_*(\Sigma(2,5,9)) \cong \Z^2_0 \oplus \Z^1_2 \oplus \Z^2_4 \oplus \Z^1_6, \]
where the subscripts denote the (absolute, in this case) $\Z/8\Z$ grading.  Fr{\o}yshov computed in \cite[Proposition~1]{froyshov-inequality}  that
\[ h(\Sigma(2,5,9)) = 1, \]
where $h$ is his invariant from \cite{froyshov}.  The $h$ invariant is defined by
\[ 2h(Y) = \chi(I_*(Y)) - \chi(\hat{I}_*(Y)) \]
for any homology 3-sphere $Y$, where $\hat{I}_*$ is reduced instanton homology \cite{froyshov}, and so 
\[ \chi(\hat{I}_*(\Sigma(2,5,9))) = 4. \]

Reduced instanton homology  satisfies
\[ \rank \hat{I}_q(Y) \leq \rank I_q(Y) \]
for all $q$, by definition, with equality if $q \equiv 2,3\pmod{4}$;  if $q \not\equiv 1\pmod{4}$ then $\hat{I}_q(Y)$ is a subgroup of $I_q(Y)$ (see \cite[\S3.3]{froyshov}).  Thus we have
\[ \hat{I}_*(\Sigma(2,5,9)) \cong V_0 \oplus \Z_2 \oplus V_4 \oplus \Z_6 \]
for some $V_0,V_4 \subset \Z^2$.  Moreover, there is a degree-4 endomorphism $u$ on $\hat{I}_*(Y)$ such that $u^2-64$ is nilpotent \cite[Theorem~10]{froyshov}, which implies that $u: V_0 \to V_4$ is an isomorphism.  From $\chi(\hat{I}_*(\Sigma(2,5,9))) = 4$ we conclude that
\[ \hat{I}_*(\Sigma(2,5,9)) \cong \Z_0 \oplus \Z_2 \oplus \Z_4 \oplus \Z_6 \]
is free of rank one in each even grading.

We now apply a corollary by Scaduto \cite[Corollary~1.5]{scaduto} of Fukaya's connected sum theorem, which says that since $Y = \Sigma(2,5,9)$ is $\pm1$-surgery on a genus-2 knot, we have an isomorphism with coefficients in $\C$,
\[ I^\#(Y) \cong H_*(\pt;\C) \oplus H_*(S^3;\C) \otimes \bigoplus_{j=0}^3 \hat{I}_j(Y) \]
as $\Z/4\Z$-graded $\C$-modules.  Thus, with $\C$-coefficients, we have
\[ I^\#(\Sigma(2,5,9)) \cong \C^2_0 \oplus \C_1 \oplus \C_2 \oplus \C_3, \]
which completes the proof.
\end{proof}

Finally, from \eqref{eq:triangle-untwisted} we have an exact triangle
\[ \dots \to I^\#(S^3) \to I^\#(S^3_2(T_{2,5})) \to I^\#(S^3_3(T_{2,5})) \to \dots, \]
which together with Lemma~\ref{lem:t25-1} implies that
\[ \dim I^\#(S^3_2(T_{2,5})) \leq 4. \]
Similarly, from Lemma~\ref{lem:t25-2} and the exact triangle
\begin{equation} \label{eq:t25-3}
\dots \to I^\#(S^3) \xrightarrow{I^\#(X_1(T_{2,5}),\nu_1)} I^\#(S^3_1(T_{2,5})) \to I^\#(S^3_2(T_{2,5})) \to \dots,
\end{equation}
we deduce that
\[ \dim I^\#(S^3_2(T_{2,5})) \geq 4, \]
hence the dimension must be equal to $4$.  Since equality holds iff the map $I^\#(X_1(T_{2,5}),\nu_1)$ is injective, we have proved the following.

\begin{proposition} \label{prop:x1-t25-injective}
The map $I^\#(X_1(T_{2,5}),\nu_1): I^\#(S^3) \to I^\#(S^3_1(T_{2,5}))$ in \eqref{eq:t25-3} is injective, where $X_1(T_{2,5})$ is the trace of 1-surgery on $T_{2,5}$. \qed
\end{proposition}

\subsection{The structure theorem in general} \label{ssec:structure-general}

We now deduce the general case of Theorem~\ref{thm:structure-theorem} from the case where $b_2^+(X)$ is positive.  In this subsection we take
\[ (X,\nu): (Y_0,\lambda_0) \to (Y_1,\lambda_1) \]
to be a cobordism with $b_1(X)=0$ but with no restrictions on $b_2^+(X)$.

We begin by taking a 3-ball in $Y_1$ which avoids $\lambda_1$, identifying a $T_{2,5}$ knot inside this ball, and letting
\[ Y'_1 = Y_1 \# S^3_1(T_{2,5}) \]
denote the result of $1$-surgery along this $T_{2,5}$.  The trace of this surgery is a cobordism
\[ (Z, \lambda_1 \times[0,1] \cup \nu_1): (Y_1,\lambda_1) \to (Y'_1,\lambda_1). \]
built by attaching a 2-handle to $Y_1\times[0,1]$ along $T_{2,5} \times \{1\}$, where $\nu_1$ is the cobordism of Proposition \ref{prop:x1-t25-injective} suitably interpreted.

\begin{lemma} \label{lem:trace-t25-injective}
The induced map
\[ I^\#(Z,\lambda_1\times[0,1] \cup \nu_1): I^\#(Y_1,\lambda_1) \to I^\#(Y'_1,\lambda_1) \]
is injective.
\end{lemma}

\begin{proof}
Letting $X_n(K): S^3 \to S^3_n(K)$ be the trace of $n$-surgery on $K$, as in \S\ref{ssec:injective-cobordism}, we have
\[ Z \cong (Y_1 \times [0,1]) \bowtie X_1(T_{2,5}). \]
The K\"unneth isomorphism
\[ I^\#(Y \# Y',\lambda+\lambda') \xrightarrow{\cong} I^\#(Y,\lambda) \otimes I^\#(Y',\lambda') \]
 is natural with respect to split cobordisms \cite[\S 7.7]{scaduto}, meaning in this case that the diagram
\[ \xymatrix{
I^\#(Y_1\#S^3) \ar[rrr]^-{I^\#((Y_1\times[0,1]) \bowtie X_1(T_{2,5}))} \ar[d]_{\cong} &&& I^\#(Y_1\# S^3_1(T_{2,5})) \ar[d]^{\cong} \\
I^\#(Y_1) \otimes I^\#(S^3) \ar[rrr]^-{\mathrm{Id}\otimes I^\#(X_1(T_{2,5}))} &&& I^\#(Y_1) \otimes I^\#(S^3_1(T_{2,5}))
} \]
commutes.  (We omit the various $\lambda_1$, $\lambda_1\times[0,1]$, and $\nu_1$ from the diagram for readability.)  The map $I^\#(X_1(T_{2,5}),\nu_1)$ is injective by Proposition~\ref{prop:x1-t25-injective}, so it follows that the top arrow $I^\#(Z,\lambda_1\times[0,1] \cup \nu_1)$ is injective as well.
\end{proof}

We now modify the cobordism $X$ by attaching $Z$ to get
\[ (X',\nu') = (X,\nu) \cup_{(Y_1,\lambda_1)} (Z,\lambda_1\times[0,1] \cup \nu_1): (Y_0,\lambda_0) \to (Y'_1,\lambda_1). \]
Then $b_1(X') = b_1(X) = 0$, and $H_2(X') \cong H_2(X) \oplus H_2(Z)$, where $H_2(Z) \cong \Z$ is generated by a surface $F$ built by gluing a Seifert surface for $T_{2,5}$ to the core of the 2-handle.  Let $\Sigma_1,\dots,\Sigma_k$ be an integral basis of $H_2(X;\Z)/\torsion$.  Since $F\cdot F = 1$, we have $b_2^+(X') > 0$, and hence by Proposition~\ref{prop:structure-thm-positive} there are finitely many basic classes
\[ K'_1,\dots,K'_r: H_2(X') \to \Z \]
and homomorphisms
\[ a'_1,\dots,a'_r: I^\#(Y_0,\lambda_0) \to I^\#(Y'_1,\lambda_1) \]
such that
\[ \sD_{X'}^{\nu'}(sF+t_1\Sigma_1+\dots+t_k\Sigma_k) = e^{Q(sF+\sum_i t_i\Sigma_i)/2} \sum_{j=1}^r a'_j e^{sK'_j(F)+\sum_i t_i K'_j(\Sigma_i)}. \]

For convenience, we will work with modified Donaldson series as in \cite{fs-structure}, namely 
\[ \sK_{X'}^{\nu'} = e^{-Q_{X'}/2} \sD_{X'}^{\nu'} \,\textrm{ and } \,\sK_X^\nu = e^{-Q_X/2} \sD_X^\nu, \]
so that 
\begin{equation} \label{eq:KXprime-structure}
\sK_{X'}^{\nu'}(sF+t_1\Sigma_1+\dots+t_k\Sigma_k) = \sum_{j=1}^r a'_j e^{sK'_j(F)+\sum_i t_i K'_j(\Sigma_i)}.
\end{equation}
Letting $S = t_1\Sigma_1+\dots+t_k\Sigma_k$, we observe that
\begin{align*}
\sK_{X'}^{\nu'}(S) &= e^{-Q_{X'}(S)/2} D_{X',\nu'}({-}\otimes(1+\tfrac{x}{2})e^S) \\
&= D_{Z,\lambda_1\times[0,1]\cup\nu_1}({-}\otimes 1) \circ e^{-Q_X(S)/2} D_{X,\nu}({-}\otimes(1+\tfrac{x}{2})e^S)
\end{align*}
or equivalently
\begin{equation} \label{eq:KXprime-composition}
\sK_{X'}^{\nu'}(t_1\Sigma_1+\dots+t_k\Sigma_k) = I^\#(Z,\lambda_1\times[0,1]\cup\nu_1) \circ \sK_X^\nu(t_1\Sigma_1+\dots+t_k\Sigma_k)
\end{equation}
as formal series in the variables $t_i$.

\begin{lemma} \label{lem:general-structure-exists}
For fixed $(X,\nu)$, there are finitely many basic classes $K_{\nu,i}: H_2(X;\Z) \to \Z$ and nonzero maps
\[ a_{\nu,i} \in \Hom(I^\#(Y_0,\lambda_0), I^\#(Y_1,\lambda_1)), \]
each with rational coefficients, such that
\[ \sD_X^\nu(h) = e^{Q(h)/2} \sum_i a_{\nu,i} e^{K_{\nu,i}(h)}. \]
\end{lemma}

\begin{proof}
For each $i=1,\dots,k$, we define a finite set
\[ A_i = \{ K'_j(\Sigma_i) \mid 1 \leq j \leq r \} \subset \Z \]
and an operator
\[ \delta_i = \prod_{c \in A_i} \left(\frac{\partial}{\partial t_i} - c\right), \]
and it follows immediately from equation~\eqref{eq:KXprime-structure} that
\[ \delta_i \sK_{X'}^{\nu'}(sF+t_1\Sigma_1+\dots+t_k\Sigma_k) = 0. \]
Setting $s=0$, we deduce from equation~\eqref{eq:KXprime-composition} that
\[ I^\#(Z;\lambda_1\times[0,1] \cup\nu_1) \circ \delta_i \sK_{X}^{\nu}(t_1\Sigma_1+\dots+t_k\Sigma_k) = 0, \]
and Lemma~\ref{lem:trace-t25-injective} says that $I^\#(Z;\lambda_1\times[0,1]\cup\nu_1)$ is injective, so in fact
\[ \delta_i \sK_X^\nu(t_1\Sigma_1 + \dots + t_k\Sigma_k) = 0, \qquad i=1,\dots,k. \]
This means that $\sK_X^\nu = e^{-Q_X/2} \sD_K^\nu$ has the form
\[ \sK_X^\nu\left(\sum_{i=1}^k t_i\Sigma_i\right) = \sum_{\underline{c}=(c_1,\dots,c_k) \in A_1\times\dots\times A_k} a_{\underline{c}} e^{c_1t_1 + \dots + c_k t_k} \]
for some homomorphisms $a_{\underline{c}}$, which have rational coefficients just as in the case $b_2^+(X) > 0$, as desired.
\end{proof}

For the following lemma, given a homomorphism $K': H_2(X';\Z) \to \Z$, we will write $K'|_X$ to denote the composition
\[ H_2(X;\Z) \xrightarrow{i_*} H_2(X';\Z) \xrightarrow{K'} \Z, \]
in which $i: X \hookrightarrow X'$ is the obvious inclusion.

\begin{lemma} \label{lem:general-restriction}
If $K$ is a basic class for $(X,\nu)$, then there is a basic class $K'$ for $(X',\nu')$ such that $K = K'|_X$.
\end{lemma}

\begin{proof}
We take $A_i = \{K'_j(\Sigma_i)\}$ for $i=1,\dots,k$ as before, and we recall from the proof of Lemma~\ref{lem:general-structure-exists} that every basic class of $(X,\nu)$ has the form
\[ K_{\nu,j}(t_1\Sigma_1 + \dots + t_k\Sigma_k) =  c_1t_1 + \dots + c_kt_k \]
for some integers $c_i \in A_i$.  We thus define an operator
\begin{equation} \label{eq:delta-K}
\delta_K = \prod_{i=1}^k \prod_{\substack{c \in A_i \\ c \neq K(\Sigma_i)}} \left( \frac{ \frac{\partial}{\partial t_i} - c } {K(\Sigma_i) - c} \right),
\end{equation}
and it follows that
\[ \delta_K e^{K_{\nu,j}(t_1\Sigma_1 + \dots + t_k\Sigma_k)} = \begin{cases} e^{K(t_1\Sigma_1+\dots+t_k\Sigma_k)} & K_{\nu,j} = K \\ 0 & K_{\nu,j} \neq K. \end{cases} \] 

In particular, if $K$ is a basic class of $(X,\nu)$ with coefficient $a\neq 0$ in $\sK_X^\nu$, then
\[ \delta_K \sK_X^\nu(t_1\Sigma_1+\dots+t_k\Sigma_k) = ae^{K(t_1\Sigma_1+\dots+t_k\Sigma_k)} \neq 0, \]
and so by Lemma~\ref{lem:trace-t25-injective} and equation~\eqref{eq:KXprime-composition} we see that
\[ \delta_K \sK_{X'}^{\nu'} (t_1\Sigma_1 + \dots + t_k\Sigma_k) = \sum_{j=1}^r a'_j \cdot \delta_K e^{K'_j(t_1\Sigma_1 + \dots + t_k\Sigma_k)} \]
is nonzero as well, where the equality follows from \eqref{eq:KXprime-structure}.  Each $K'_j(\Sigma_i)$ belongs to the set $A_i$, so if the $j$th term is nonzero for some fixed $j$ then again we must have $K'_j(\Sigma_i)=K(\Sigma_i)$ for all $i$, and so $K'_j|_X = K$.
\end{proof}

\begin{lemma} \label{lem:general-basic-w2}
The basic classes $K$ for $(X,\nu)$ satisfy $K(h) + h\cdot h \equiv 0 \pmod{2}$ for all $h \in H_2(X;\Z)$.
\end{lemma}

\begin{proof}
Write $K = K'|_X$ for some basic class $K'$ on $X'$, which we can do by Lemma \ref{lem:general-restriction}.  Then since $b^+(X') > 0$, we may apply Proposition~\ref{prop:structure-thm-positive} to $X'$ to conclude that
\[ K(h) + h\cdot h = K'(h) + h\cdot h \equiv 0\pmod{2}. \qedhere \]
\end{proof}

\begin{lemma} \label{lem:general-sign-change}
Let $K_{\nu,i}$ be the basic classes for $(X,\nu)$, with coefficients $a_{\nu,i}$, so that
\[ \sD_X^\nu(h) = e^{Q(h)/2} \sum_i a_{\nu,i} e^{K_{\nu,i}(h)}. \]
Then for all $\alpha \in H_2(X;\Z)$, we have
\[ \sD_X^{\nu+\alpha}(h) = e^{Q(h)/2} \sum_i (-1)^{\frac{1}{2}(K_{\nu,i}(\alpha) + \alpha\cdot\alpha) + \nu\cdot\alpha} \cdot a_{\nu,i} e^{K_{\nu,i}(h)}. \]
In particular, the set of basic classes does not depend on $\nu$.
\end{lemma}

\begin{proof}
Fix $\alpha \in H_2(X;\Z)$ and let $K$ be a basic class for $(X,\nu)$, with coefficient $a_{\nu} \neq 0$ in $\sD_X^\nu$ and coefficient $a_{\nu+\alpha}$ (possibly zero) in $\sD_X^{\nu+\alpha}$.  We write \[\epsilon_\alpha = (-1)^{\frac{1}{2}(K(\alpha) + \alpha\cdot\alpha)+\nu\cdot\alpha},\] and set \[S=t_1\Sigma_1+\dots+t_k\Sigma_k\] for readability.

Recalling that $\sK_X^\nu = e^{-Q_X/2} \sD_X^\nu$ and likewise for $\sK_{X'}^{\nu'}$, we now observe by \eqref{eq:KXprime-composition} that
\[ I^\#(Z,\lambda_1\times[0,1]\cup\nu_1) \circ \delta_K (\sK_X^{\nu+\alpha}(S) - \epsilon_\alpha\sK_X^{\nu}(S)) = \delta_K (\sK_{X'}^{\nu'+\alpha}(S) - \epsilon_\alpha\sK_{X'}^{\nu'}(S)), \]
where $\delta_K$ is the operator defined in \eqref{eq:delta-K}.  On the right side, we have
\[ \delta_K \sK_{X'}^{\nu'+\alpha}(S) - \epsilon_\alpha\delta_K \sK_{X'}^{\nu'}(S) = \sum_{K'_j|_X=K} (a'_{\nu'+\alpha,j} - \epsilon_\alpha a'_{\nu',j})e^{K'_j(S)}, \]
and each of the coefficients $a'_{\nu'+\alpha,j} - \epsilon a'_{\nu',j}$ is zero by Lemma~\ref{lem:structure-thm-signs} and the fact that $K'_j(\alpha) = K(\alpha)$, so this vanishes.  On the left side, Lemma~\ref{lem:trace-t25-injective} says that $I^\#(Z,\lambda_1 \times [0,1]\cup\nu_1)$ is injective, so we deduce that
\[ \delta_K (\sK_X^{\nu+\alpha}(S) - \epsilon_\alpha\sK_X^{\nu}(S)) = a_{\nu+\alpha} e^{K(S)} - \epsilon_\alpha a_{\nu} e^{K(S)} \]
is zero as well.  Thus $a_{\nu+\alpha} = \epsilon_\alpha a_{\nu}$ and the lemma follows.
\end{proof}

\begin{proof}[Proof of Theorem~\ref{thm:structure-theorem}]
The case where $b_2^+(X) > 0$ is Proposition~\ref{prop:structure-thm-positive}.  If instead $b_2^+(X) = 0$ then the adjunction inequality is vacuously true, while the rest of the theorem is a combination of Lemmas~\ref{lem:general-structure-exists}, \ref{lem:general-basic-w2}, and \ref{lem:general-sign-change}.
\end{proof}

\section{A decomposition of cobordism maps} \label{sec:eigenspace}

The   goal of this section is to use the structure theorem of the previous section (Theorem \ref{thm:structure-theorem}) to prove Theorem \ref{thm:main-cobordism-decomposition}, which extends the eigenspace decomposition of framed instanton homology in Theorem~\ref{thm:eigenspace-decomposition} to a similar decomposition for cobordism maps,  akin to the $\textrm{Spin}^c$ decompositions of cobordism maps in Heegaard and monopole Floer homology. 

We will assume throughout this section that \[(X,\nu): (Y_0,\lambda_0) \to (Y_1,\lambda_1)\] is a cobordism with $b_1(X) = 0$.  Recall that Theorem \ref{thm:main-cobordism-decomposition} says that  there is a  decomposition of the induced  cobordism map as a sum
\[ I^\#(X,\nu) = \sum_{s: H_2(X;\Z) \to \Z} I^\#(X,\nu;s) \] of maps 
\[ I^\#(X,\nu;s): I^\#(Y_0,\lambda_0;s|_{Y_0}) \to I^\#(Y_1,\lambda_1;s|_{Y_1}) \]
which satisfy five properties. The proof of this theorem below will make reference to these properties as they are numbered in the theorem statement.

\begin{proof}[Proof of Theorem \ref{thm:main-cobordism-decomposition}]
By Theorem~\ref{thm:structure-theorem}, we can write
\begin{equation} \label{eq:structure-thm-for-decomposition}
\sD_X^\nu(h) = e^{Q(h)/2} \sum_{j=1}^r a_{\nu,j} e^{K_j(h)},
\end{equation}
where the $a_j$ are nonzero  homomorphisms $I^\#(Y_0,\lambda_0) \to I^\#(Y_1,\lambda_1)$ and the basic classes $K_j$ are elements of $\Hom(H_2(X;\Z),\Z)$.  Then for any homomorphism $s: H_2(X;\Z) \to \Z$, we define
\[ I^\#(X,\nu;s): I^\#(Y_0,\lambda_0) \to I^\#(Y_1,\lambda_1) \]
to be $\frac{1}{2}a_{\nu,j}$ if $K_j = s$ for some $j$, and $I^\#(X,\nu;s) = 0$ otherwise.

Properties~\eqref{i:finite-support}, \eqref{i:adjunction}, and \eqref{i:sign-change} are immediate from the definition of $I^\#(X,\nu;s)$ and Theorem~\ref{thm:structure-theorem}, since $I^\#(X,
\nu;s)$ is nonzero precisely when $s$ is a basic class.  Property~\eqref{i:blowup-formula} follows from the blow-up formula for $\sD_X^\nu$, Theorem~\ref{thm:blowup-formula}.  
It remains to verify the following:
\begin{itemize}
\item The identity $I^\#(X,\nu) = \sum_s I^\#(X,\nu;s)$;
\item The fact that each $I^\#(X,\nu;s)$ is zero on all $I^\#(Y_0,\lambda_0;s_0)$ except for $s_0 = s|_{Y_0}$, and that its image lies in $I^\#(Y_1,\lambda_1;s|_{Y_1})$;
\item Property~\eqref{i:composition-law}, the composition law.
\end{itemize}
For the first of these, we set $h=0$ in \eqref{eq:structure-thm-for-decomposition} to get
\begin{equation} \label{eq:evaluate-x/2-cobordism}
D_{X,\nu}\left({-} \otimes \left(1+\tfrac{x}{2}\right)\right) = \sum_{j=1}^r a_{\nu,j} = 2\left(\sum_{s:H_2(X;\Z)\to\Z} I^\#(X,\nu;s)\right).
\end{equation}
But $x=\mu(\pt)$ acts on $I^\#(Y_0,\lambda_0)$ as multiplication by $2$, as in Remark \ref{rmk:simpletype3}, so for all $a \in I^\#(Y_0,\lambda_0)$ we have
\[ D_{X,\nu}\left(a \otimes \left(1+\tfrac{x}{2}\right)\right) = D_{X,\nu}\left(\left(1+\tfrac{\mu(\pt)}{2}\right)a \otimes 1\right) = I^\#(X,\nu)(2a), \]
and hence the left side of \eqref{eq:evaluate-x/2-cobordism} is $2I^\#(X,\nu)$, establishing the claim.  The second and third items above are proved below as Proposition~\ref{prop:eigenspace-restriction} and Proposition~\ref{prop:composition-law}, respectively.
\end{proof}

\begin{remark}
If $I^\#(X,\nu;s)$ is nonzero, then $s|_{Y_0}$ must take values in $2\Z$, since for any class $h \in H_2(Y_0;\Z)$ we have
\[ s|_{Y_0}(h) = s((i_0)_*h) \equiv (i_0)_*h \cdot (i_0)_*h = h\cdot h=0 \pmod{2} \]
by property~\eqref{i:adjunction} of Theorem~\ref{thm:main-cobordism-decomposition}, which we have already proved.
\end{remark}

\begin{remark} \label{rem:nu-plus-even}
Properties \eqref{i:adjunction} and \eqref{i:sign-change} of Theorem~\ref{thm:main-cobordism-decomposition} imply that
\[ I^\#(X,\nu+2\alpha;s) = (-1)^{\alpha\cdot\alpha} I^\#(X,\nu;s) \]
for all $s$.  Indeed, both sides are zero unless $s(\alpha) \equiv \alpha\cdot\alpha \pmod{2}$, and if they are nonzero then by \eqref{i:sign-change} the exponent on the right should be
\[ \tfrac{1}{2}(s(2\alpha) + (2\alpha)\cdot(2\alpha)) + \nu\cdot(2\alpha) \equiv s(\alpha) \equiv \alpha\cdot\alpha \pmod{2}. \]
\end{remark}

\begin{proposition} \label{prop:eigenspace-restriction}
We have the following:\begin{itemize}
\item For any $s_0: H_2(Y_0;\Z) \to 2\Z$, the map $I^\#(X,\nu;s)$ is zero on $I^\#(Y_0,\lambda_0;s_0)$ unless $s_0=s|_{Y_0}$.
\item The image of $I^\#(X,\nu;s)$ lies in $I^\#(Y_1,\lambda_1;s|_{Y_1})$.
\end{itemize}
In other words, for each $s$ we can interpret $I^\#(X,\nu;s)$ as a map
\[ I^\#(X,\nu;s): I^\#(Y_0,\lambda_0;s|_{Y_0}) \to I^\#(Y_1,\lambda_1;s|_{Y_1}). \]
\end{proposition}

\begin{proof}
We fix a non-torsion class $\Sigma \in H_2(Y_0;\Z)$ and extend it to a rational basis $\Sigma_1 = \Sigma, \Sigma_2,\dots,\Sigma_k$ of $H_2(X;\Z)$.  (The inclusion $Y = \partial X \hookrightarrow X$ gives an injection $H_2(Y;\Q) \to H_2(X;\Q)$, by the long exact sequence of the pair $(X,Y)$ and the fact that $H_3(X,Y;\Q) \cong H^1(X;\Q) = 0$.) We then write
\[ \sD_X^\nu(t_1\Sigma_1 + \dots + t_k\Sigma_k) = e^{Q(\sum t_i\Sigma_i)/2} \sum_{j=1}^r a_j e^{K_j(t_1\Sigma_1+\dots+t_k\Sigma_k)} \]
by Theorem~\ref{thm:structure-theorem}, and observe that
\[ Q\left(\sum_{i=1}^k t_i\Sigma_i\right) = Q\left(\sum_{i=2}^k t_i\Sigma_i\right) \]
does not depend on $t_1$, since $H_2(Y)$ is in the kernel of the intersection pairing on $X$.

We now take an arbitrary element $z \in I^\#(Y_0,\lambda_0;s_0)$ and an integer $n \geq 1$, and we apply the operator $\left(\frac{\partial}{\partial t_1}-s_0(\Sigma_1)\right)^n$ to both sides of
\[ \sD_X^\nu(t_1\Sigma_1+\dots+t_k\Sigma_k)(z) = D_{X,\nu}\left(z\otimes \left(1+\tfrac{x}{2}\right)e^{t_1\Sigma_1+\dots+t_k\Sigma_k}\right). \]
On the left side, we have
\begin{multline} \label{eq:restriction-domain-left}
\left(\frac{\partial}{\partial t_1} - s_0(\Sigma_1)\right)^n \sD^\nu_X(t_1\Sigma_1+\dots+t_k\Sigma_k)(z) \\= e^{Q(\sum t_i\Sigma_i)/2} \sum_{j=1}^r a_j(z)\left(K_j(\Sigma_1)-s_0(\Sigma_1)\right)^n e^{K_j(t_1\Sigma_1+\dots+t_k\Sigma_k)}.
\end{multline}
On the right side, we have
\begin{multline*}
D_{X,\nu}\left(z\otimes \left(1+\tfrac{x}{2}\right)\left(\Sigma_1-s_0(\Sigma_1)\right)^ne^{t_1\Sigma_1+\dots+t_k\Sigma_k}\right) \\
= D_{X,\nu}\left((\mu(\Sigma_1)-s_0(\Sigma_1))^n z\otimes \left(1+\tfrac{x}{2}\right)e^{t_1\Sigma_1+\dots+t_k\Sigma_k}\right),
\end{multline*}
which is equivalently
\begin{multline} \label{eq:restriction-domain-right}
\big(\sD_X^\nu(t_1\Sigma_1+\dots+t_k\Sigma_k)\big)\big((\mu(\Sigma_1)-s_0(\Sigma_1))^nz\big) \\= e^{Q(\sum t_i\Sigma_i)/2} \sum_{j=1}^r a_j\big((\mu(\Sigma_1)-s_0(\Sigma_1))^nz\big) e^{K_j(t_1\Sigma_1+\dots+t_k\Sigma_k)}.
\end{multline}

The various functions $e^{K_j(t_1\Sigma_1+\dots+t_k\Sigma_k)}$ for $1 \leq j \leq r$ are linearly independent as power series in $t_1,\dots,t_k$, so they must have the same coefficients in \eqref{eq:restriction-domain-left} and \eqref{eq:restriction-domain-right}, i.e.,
\[ a_j(z) \cdot (K_j(\Sigma_1)-s_0(\Sigma_1))^n = a_j\big((\mu(\Sigma_1)-s_0(\Sigma_1))^nz\big). \]
The right side is identically zero for $n$ large, since $z$ by definition belongs to the (generalized) $s_0(\Sigma_1)$-eigenspace of $\mu(\Sigma_1)$.  Thus the left side is zero for $n$ large as well, which implies that it is also zero when $n=1$, so in fact we have
\begin{equation} \label{eq:mu-action-in-kernel}
a_j(z) \cdot (K_j(\Sigma_1)-s_0(\Sigma_1)) = a_j\big((\mu(\Sigma_1)-s_0(\Sigma_1))z\big) = 0
\end{equation}
for all $j$.  We conclude that if $K_j(\Sigma_1) \neq s_0(\Sigma_1)$ then $a_j(z) = 0$, and hence $I^\#(X,\nu;s)(z)=0$ for all $s$ which do not satisfy $s(\Sigma_1) = s_0(\Sigma_1)$.  But $\Sigma_1 = \Sigma$ was an arbitrary non-torsion class in $H_2(Y_0;\Z)$, so if $I^\#(X,\nu;s)(z) \neq 0$ then $s|_{Y_0} = s_0$ on all of $H_2(Y_0;\Z)$.

The proof that $I^\#(X,\nu;s)$ sends $I^\#(Y_0,\lambda_0)$ into $I^\#(Y_1,\lambda_1;s|_{Y_1})$ is similar.  We choose our basis $\{\Sigma_k\}$ so that $\Sigma_1$ is a given non-torsion element of $H_2(Y_1;\Z)$, and then apply $\frac{\partial}{\partial t_1}$
to both sides of
\[ e^{Q(\sum t_i\Sigma_i)/2}\sum_{j=1}^r a_j(z)e^{K_j(t_1\Sigma_1+\dots+t_k\Sigma_k)} =
D_{X,\nu}\left(z\otimes \left(1+\tfrac{x}{2}\right)e^{t_1\Sigma_1+\dots+t_k\Sigma_k}\right) \]
to get
\[ e^{Q(\sum t_i\Sigma_i)/2}\sum_{j=1}^r K_j(\Sigma_1)\cdot a_j(z)e^{K_j(t_1\Sigma_1+\dots+t_k\Sigma_k)} \]
on the left side, and
\begin{align*}
D_{X,\nu}\left(z\otimes \left(1+\tfrac{x}{2}\right)\Sigma_1 e^{t_1\Sigma_1+\dots+t_k\Sigma_k}\right)
&= \mu(\Sigma_1) D_{X,\nu}\left(z\otimes \left(1+\tfrac{x}{2}\right) e^{t_1\Sigma_1+\dots+t_k\Sigma_k}\right) \\
&= \mu(\Sigma_1) \cdot (\sD_X^\nu(t_1\Sigma_1+\dots+t_k\Sigma_k))(z) \\
&= e^{Q(\sum t_i\Sigma_i)/2}\sum_{j=1}^r \mu(\Sigma_1)a_j(z) e^{K_j(t_1\Sigma_1+\dots+t_k\Sigma_k)}
\end{align*}
on the right.  Again, the linear independence of the $e^{K_j(\sum t_i\Sigma_i)}$ tells us that
\begin{equation} \label{eq:mu-action-on-image}
\mu(\Sigma_1)a_j(z) = K_j(\Sigma_1) a_j(z)
\end{equation}
for all $j$, so if $I^\#(X,\nu;s)(z)$ is nonzero, then $s=K_j$ for some $j$ and we have
\[ I^\#(X,\nu;s)(z) = \frac{1}{2}a_j(z) \in \ker( \mu(\Sigma_1)-s(\Sigma_1) ). \]
Since $\Sigma_1$ was an arbitrary non-torsion element of $H_2(Y_1;\Z)$, we conclude that the image of $I^\#(X,\nu;s)$ lies in
\[ \bigcap_{h \in H_2(Y_1;\Z)} \ker(\mu(h)-s(h)) \subset I^\#(Y_1,\lambda_1;s|_{Y_1}). \qedhere \]
\end{proof}

\begin{remark} \label{rem:actual-eigenvectors}
We can deduce from the proof of Proposition~\ref{prop:eigenspace-restriction}, and in particular equation~\eqref{eq:mu-action-in-kernel}, that for all $z \in I^\#(Y_0,\lambda_0;s_0)$ and all $\Sigma \in H_2(Y_0;\Z)$ we have
\[ (\mu(\Sigma)-s_0(\Sigma))z \in \ker I^\#(X,\nu;s) \]
for all $s$.  Similarly, equation~\eqref{eq:mu-action-on-image} shows that the image of a map $I^\#(X,\nu;s)$ consists of actual simultaneous eigenvectors, as opposed to generalized ones.
\end{remark}

\begin{proposition} \label{prop:composition-law}
Let $(X,\nu)$ denote the composite cobordism
\[ (Y_0,\lambda_0) \xrightarrow{(X_{01},\nu_{01})} (Y_1,\lambda_1) \xrightarrow{(X_{12},\nu_{12})} (Y_2,\lambda_2), \]
where $b_1(X) = b_1(X_{01}) = b_1(X_{12}) = 0$.  Then for all homomorphisms $s_{01}: H_2(X_{01};\Z) \to \Z$ and $s_{12}: H_2(X_{12};\Z) \to \Z$, we have
\[ I^\#(X_{12},\nu_{12};s_{12}) \circ I^\#(X_{01},\nu_{01};s_{01}) = \sum_{\substack{s: H_2(X;\Z)\to\Z \\ s|_{X_{01}} = s_{01},\ s|_{X_{12}}=s_{12}}} I^\#(X,\nu;s). \]
\end{proposition}

\begin{proof}
Let $\Sigma_1,\dots,\Sigma_k$ and $\Sigma'_1,\dots,\Sigma'_{k'}$ be integral bases of $H_2(X_{01};\Z)$ and $H_2(X_{12};\Z)$, respectively, modulo torsion, and write
\begin{align*}
\sD_{X_{01}}^{\nu_{01}} &= e^{Q/2} \sum_{j=1}^r a_j e^{K_j}, &
\sD_{X_{12}}^{\nu_{12}} &= e^{Q/2} \sum_{j=1}^{r'} a'_j e^{K'_j}.
\end{align*}
We define some finite sets of integers by
\begin{align*}
A_i &= \{ K_j(\Sigma_i) \mid 1 \leq j \leq r \}, &
A'_i &= \{ K'_j(\Sigma'_i) \mid 1 \leq j \leq r' \}
\end{align*}
and let $d_i = s_{01}(\Sigma_i)$ and $d'_i = s_{12}(\Sigma'_i)$ for all $i$; here $i$ ranges from $1$ to $k$ for $A_i$ and $d_i$, and from $1$ to $k'$ for $A'_i$ and $d'_i$.  Then we define some differential operators by
\begin{align*}
\delta_{s_{01}} &= \prod_{i=1}^k \prod_{c \in A_i \ssm \{s_{01}(\Sigma_i)\}} \left( \frac{\frac{\partial}{\partial t_i}-c}{d_i-c} \right), &
\delta_{s_{12}} &= \prod_{i=1}^{k'} \prod_{c \in A'_i \ssm \{s_{12}(\Sigma'_i)\}} \left( \frac{\frac{\partial}{\partial t'_i}-c}{d'_i-c} \right)
\end{align*}
so that if we write $\sK = e^{-Q/2}\sD$ for each cobordism then
\[ \delta_{s_{01}} \sK_{X_{01}}^{\nu_{01}}(t_1\Sigma_1 + \dots + t_k\Sigma_k) = \begin{cases} a_j e^{K_j(t_1\Sigma_1 + \dots + t_k\Sigma_k)} & s_{01}=K_j \\ 0 & s_{01} \not\in \{K_1,\dots,K_r\}. \end{cases} \]
Upon setting $(t_1,\dots,t_k)=(0,\dots,0)$, this becomes
\[ \left.\delta_{s_{01}} \sK_{X_{01}}^{\nu_{01}}(t_1\Sigma_1 + \dots + t_k\Sigma_k)\right|_{t_1=\dots=t_k=0} = 2I^\#(X_{01},\nu_{01};s_{01}), \]
and the same argument gives
\[ \left.\delta_{s_{12}} \sK_{X_{12}}^{\nu_{12}}(t'_1\Sigma'_1+\dots+t'_{k'}\Sigma'_{k'})\right|_{t'_1=\dots=t'_{k'}=0} = 2I^\#(X_{12},\nu_{12};s_{12}). \]

Letting $S = t_1\Sigma_1 + \dots + t_k\Sigma_k$ and $S'=t'_1\Sigma'_1+\dots+t'_{k'}\Sigma'_{k'}$ for convenience, we now compute
\begin{align*}
\sD_{X_{12}}^{\nu_{12}}(S') \circ \sD_{X_{01}}^{\nu_{01}}(S) &= D_{X_{12},\nu_{12}}\left({-}\otimes\left(1+\tfrac{x}{2}\right)e^{S}\right) \circ D_{X_{01},\nu_{01}}\left({-}\otimes\left(1+\tfrac{x}{2}\right)e^{S'}\right) \\
&= D_{X,\nu}\left({-}\otimes \left(1+\tfrac{x}{2}\right)^2e^{S+S'}\right) \\
&= 2 \sD_X^\nu(S+S'),
\end{align*}
since the extra factor of $1+\frac{x}{2}$ in the second line acts on $I^\#(Y_0,\lambda_0)$ as multiplication by $2$.  We have $S \cdot S' = 0$ and hence $Q_X(S+S') = Q_{X_{01}}(S) + Q_{X_{12}}(S')$, so dividing both sides by $4\exp(Q_X(S+S')/2)$ gives
\[ \tfrac{1}{2}\sK_{X_{12}}^{\nu_{12}}(S') \circ \tfrac{1}{2}\sK_{X_{01}}^{\nu_{01}}(S) = \tfrac{1}{2}\sK_X^\nu(S+S'). \]
Now applying $\delta_{s_{12}} \circ \delta_{s_{01}}$ and then setting all $t_i$ and $t'_i$ to zero turns the left side into
\[ I^\#(X_{12},\nu_{12};s_{12}) \circ I^\#(X_{01},\nu_{01};s_{01}). \]
On the right side, the operator $\delta_{s_{12}} \circ \delta_{s_{01}}$ fixes all terms in the series $\frac{1}{2}\sK_X^\nu(S+S')$ of the form $\frac{1}{2}a\cdot e^{K(S+S')}$, where the basic class $K$ satisfies 
\[ K(\Sigma_i)=s_{01}(\Sigma_i),\ 1 \leq i \leq k \quad\mathrm{and}\quad K(\Sigma'_i)=s_{12}(\Sigma'_i),\ 1 \leq i \leq k', \]
or equivalently $K|_{X_{01}} = s_{01}$ and $K|_{X_{12}}=s_{12}$; and it replaces all other terms with $0$.  Setting the $t_i$ and $t'_i$ to zero then gives the sum of $I^\#(X,\nu;s)$ over all $s$ such that $s|_{X_{01}}=s_{01}$ and $s|_{X_{12}}=s_{12}$, as desired.
\end{proof}

Weaker versions of Proposition~\ref{prop:composition-law} still hold in the case where $b_1(X) = 0$ but one of $b_1(X_{01})$ and $b_1(X_{12})$ is positive, even though we have not proved a structure theorem for the corresponding cobordism map on $I^\#$.  We include one such statement for completeness.

\begin{proposition} \label{prop:weak-composition-law}
Let $(X,\nu)$ denote the composite cobordism
\[ (Y_0,\lambda_0) \xrightarrow{(X_{01},\nu_{01})} (Y_1,\lambda_1) \xrightarrow{(X_{12},\nu_{12})} (Y_2,\lambda_2), \]
where $b_1(X) = b_1(X_{01}) = 0$ and $X_{12}$ is a rational homology cobordism.  Then
\[ I^\#(X,\nu;s) = I^\#(X_{12},\nu_{12}) \circ I^\#(X_{01},\nu_{01};s|_{X_{01}}) \]
for all homomorphisms $s: H_2(X;\Z) \to \Z$.
\end{proposition}

\begin{proof}
We take a rational basis $\Sigma_1,\dots,\Sigma_k$ of $H_2(X_{01};\Z)$, which by hypothesis is also a rational basis of $H_2(X;\Z)$.  Then, following the proof of Proposition~\ref{prop:composition-law}, we can write
\[ \sD_X^\nu(t_1\Sigma_1 + \dots + t_k \Sigma_k) = D_{X_{12},\nu_{12}}({-}\otimes 1) \circ \sD_{X_{01}}^{\nu_{01}}(t_1\Sigma_1 + \dots + t_k \Sigma_k) \]
and divide both sides by $e^{Q(\sum_i t_i\Sigma_i)/2}$ to get
\[ \sK_X^\nu(t_1\Sigma_1 + \dots + t_k \Sigma_k) = I^\#(X_{12},\nu_{12}) \circ \sK_{X_{01}}^{\nu_{01}}(t_1\Sigma_1 + \dots + t_k \Sigma_k). \]
We apply the differential operator 
\[ \prod_{i=1}^k \prod_{\substack{c \in A_i \\ c \neq s(\Sigma_i)}} \left(\frac{\frac{\partial}{\partial t_i}-c}{s(\Sigma_i)-c}\right) \]
to both sides, where $A_i$ is the set of values of $K(\Sigma_i)$ as $K$ ranges over all basic classes of $X$, and then set $t_1=\dots=t_k=0$ and divide by $2$ to get the desired relation.
\end{proof}

Finally, the adjunction inequality in Theorem~\ref{thm:main-cobordism-decomposition} gives us the following general result.

\begin{proposition} \label{prop:adjunction-inequality}
Let $(X,\nu): (Y_0,\lambda_1) \to (Y_1,\lambda_1)$ be a cobordism with $b_1(X)=0$.  Suppose that $X$ contains a smoothly embedded, closed surface $S$ satisfying one of the following:
\begin{itemize}
\item $[S]\cdot [S] \geq \max(2g(S)-1,1)$; or
\item $S$ is a sphere with $[S]\cdot[S]=0$ and $[S] \cdot [F] \neq 0$ for some smooth, closed surface $F \subset X$.
\end{itemize}
Then the cobordism map $I^\#(X,\nu)$ is identically zero.
\end{proposition}

\begin{proof}
In the first case, if $g(S) \geq 1$ then Theorem~\ref{thm:main-cobordism-decomposition} says that $I^\#(X,\nu)$ is a sum of various maps $I^\#(X,\nu;s)$, all of which are zero by part~\eqref{i:adjunction} of that theorem. If instead $S$ is a sphere then we attach a handle to get a torus $T$ of positive self-intersection and apply the same argument to $T$.

In the second case, we follow \cite[\S6(ii)]{km-embedded2}: replacing $F$ with $-F$ if needed, we arrange that $n = [S]\cdot[F]$ is positive and that $S \cap F$ consists of $n$ points.  We form another surface $S'$ by taking the union of $F$ with $d>0$ parallel, disjoint copies of $S$ and smoothing out the $dn$ points of intersection.  The result has genus
\[ g(S') = 1 - \frac{\chi(S')}{2} = 1-\frac{(2-2g(F)) + 2d - 2dn}{2}  = g(F) + d(n-1) \]
and self-intersection
\[ [S']\cdot[S'] = ([F]+d[S]) \cdot ([F]+d[S]) = [F]\cdot[F] + 2dn. \]
For $d$ sufficiently large, we have $[S]\cdot[S] \geq 1$, and also
\[ [S']\cdot[S']-(2g(S')-1) = 2d + ([F]\cdot[F] - (2g(F)-1)) > 0, \]
so we can apply the first case to $S'$ to conclude that $I^\#(X,\nu)=0$.
\end{proof}

\section{Instanton L-space knots are fibered} \label{sec:L-spaces}

Our primary goal in this section is to  prove that  instanton L-space knots are fibered, with Seifert genus equal to  smooth slice genus (Theorem \ref{thm:l-space-knots-are-fibered}). The equality $g=g_s$ of Seifert and smooth slice genus will also follow from our result, proved in the next section, that instanton L-space knots are strongly quasipositive. However, our proof of strong quasipositivity  uses  the Giroux correspondence, whereas our proof that $g=g_s$ here does not. We will also prove in this section some  results related to the bound on L-space slopes in Theorem \ref{thm:main-surgery-reduction}, which are enough to conclude, without using the Giroux correspondence, that the framed instanton homology of surgeries detects the trefoils among nontrivial knots. We will prove the general $2g-1$ bound  in the next section (again, using the Giroux correspondence).

Our proof that instanton L-space knots are fibered follows the outline described in the introduction. Let us suppose below that $K\subset S^3$ is a nontrivial knot of genus $g=g(K)>0$, with minimal genus Seifert surface $\Sigma$. For each  integer $k\geq 0$,  consider the  $2$-handle cobordisms 
\begin{align*}
X_k&:S^3 \to S^3_k(K)\\
W_{k+1}&:S^3_{k}(K)\to S^3_{k+1}(K),
\end{align*}
 where $X_k$ is the trace of $k$-surgery on $K$, and $W_{k+1}$ is the trace of $-1$-surgery on a meridian of $K$ in $S^3_{k}(K)$. A handleslide reveals that \begin{equation}\label{eqn:XWblowup}X_k \cup_{S^3_k(K)} W_{k+1} \cong X_{k+1} \# \cptwo.\end{equation} By Theorem~\ref{thm:exact-triangle} and \eqref{eq:triangle-untwisted}, the maps induced by $X_k$ and $W_{k+1}$ fit into  exact triangles
\begin{equation} \label{eq:triangle-0-cobordisms}
\dots \to I^\#(S^3) \xrightarrow{I^\#(X_0,\nu_0)} I^\#(S^3_0(K),\mu) \xrightarrow{I^\#(W_1,\omega_1)} I^\#(S^3_1(K)) \to \dots
\end{equation}
and
\begin{equation} \label{eq:triangle-n-cobordisms}
\dots \to I^\#(S^3) \xrightarrow{I^\#(X_k,\nu_k)} I^\#(S^3_k(K)) \xrightarrow{I^\#(W_{k+1},\omega_{k+1})} I^\#(S^3_{k+1}(K)) \to \dots
\end{equation}
for $k \geq 1$, for some $\nu_k$ and $\omega_{k+1}$.   Observe that the maps induced by $X_k$ and $W_{k+1}$ shift the mod 2 grading by $1$ and $0$, respectively, for all $k \geq 0$,  by Proposition~\ref{prop:grading-shift}.

Let us denote by
\[ \Sigma_k \subset X_k \]
 the surface of genus $g$ and self-intersection $k$ obtained by gluing a core of the 2-handle used to form $X_k$ to the minimal genus Seifert surface $\Sigma$ for $K$. The construction in \cite[\S3.3]{scaduto} tells us that
\begin{equation} \label{eq:nu-dot-sigma}
[\nu_k] \cdot [\Sigma_k] \equiv \begin{cases} 1 & k=0 \\ k & k \geq 1 \end{cases} \pmod{2}.
\end{equation}
A homomorphism \[s: H_2(X_k;\Z) \to \Z\] is determined by its evaluation on $[\Sigma_k]$, and Theorem~\ref{thm:main-cobordism-decomposition} says that the map $I^\#(X_k,\nu_k;s)$ is nonzero only if \[s([\Sigma_k]) \equiv \Sigma_k\cdot\Sigma_k = k \pmod{2}.\]  With this as motivation, we denote by
\[ t_{k,i}: H_2(X_k;\Z) \to \Z \]
 the unique homomorphism sending $[\Sigma_k]$ to $2i-k$.  The adjunction inequality of Theorem~\ref{thm:main-cobordism-decomposition} implies that for $k\geq 1$ the map
\[ I^\#(X_k,\nu_k;t_{k,i}): I^\#(S^3) \to I^\#(S^3_k(K)) \]
is nontrivial only if
\begin{equation} \label{eq:t_i-adjunction}
|t_{k,i}([\Sigma_k])| + k \leq 2g-2 \ \Longleftrightarrow\ 1-g+k \leq i \leq g-1.
\end{equation}
For each integer $k\geq 0$ and every integer $i$, let us define
\begin{equation} \label{eq:define-zki}
z_{k,i} := I^\#(X_k,\nu_k;t_{k,i})({\bf 1}),
\end{equation}
where ${\bf 1}$ is a fixed generator of $I^\#(S^3) \cong \C$.  For $k\geq 1$, these elements  belong to $I^\#_\godd(S^3_k(K))$ since the map induced by $X_k$ shifts the mod 2 grading by 1, as noted above.  Moreover,  for $k = 0$, we have that
\[ z_{0,i} \in I^\#_\godd(S^3_0(K),\mu;s_i) \]
where \[s_i: H_2(S^3_0(K);\Z) \to 2\Z\] is the unique homomorphism sending $[\hat\Sigma]$ to $2i$, for $\hat\Sigma$  the  capped-off Seifert surface in $S^3_0(K)$ (note that $[\hat\Sigma] = [\Sigma_0]\in H_2(X_0;\Z)$). For convenience, we define \[y_i :=z_{0,i}\in I^\#_\godd(S^3_0(K),\mu;s_i).\] Note that $y_i$ is nonzero only if \[1-g\leq i\leq g-1,\]  by the adjunction inequality of Theorem \ref{thm:eigenspace-decomposition}. As outlined in the introduction, our proof that instanton L-space knots are fibered relies on  understanding  the kernel of  a composition of the $W_k$ maps in terms of these  $y_i$ (see Lemma \ref{lem:ker-vn}). We prove the technical results we will need for this in the subsection below.

\subsection{The $W_k$ cobordism maps}

Our goal in this section is to understand the  images of the  elements $z_{k,i}$ defined in \eqref{eq:define-zki} for $k\geq 0$ under the maps
\begin{align*}
I^\#(W_1,\omega_1) &: I^\#(S^3_0(K),\mu) \to I^\#(S^3_1(K)) \\
I^\#(W_{k+1},\omega_{k+1}) &: I^\#(S^3_k(K)) \to I^\#(S^3_{k+1}(K)).
\end{align*}
Our main result along these lines is  the following.

\begin{proposition} \label{prop:w-z-image}
There are constants $\epsilon_k = \pm 1$ for all $k \geq 0$, depending only on $k$, such that
\[ I^\#(W_1,\omega_1)(z_{0,i}) = \frac{\epsilon_0}{2}\cdot (-1)^i \left(z_{1,i} + z_{1,i+1}\right) \]
and also
\[ I^\#(W_{k+1},\omega_{k+1})(z_{k,i}) = \frac{\epsilon_k}{2} \left(z_{k+1,i} - z_{k+1,i+1}\right) \]
for all $k \geq 1$.
\end{proposition}

This proposition will follow from two lemmas below. To begin, recall from \eqref{eqn:XWblowup} that \[X_k \cup_{S^3_k(K)} W_{k+1} \cong X_{k+1} \# \cptwo.\] 
Furthermore, letting $e \subset \cptwo$ be the exceptional sphere,  this diffeomorphism identifies $[\Sigma_k]$ on the left with $[\Sigma_{k+1}]-[e]$ on the right as elements of $H_2(X_{k+1}\#\cptwo) \cong \Z^2$.  

\begin{lemma} \label{lem:generator-w}
 $H_2(W_{k+1};\Z)\cong\Z$, and this group is generated by the class of a surface $F_{k+1}$ such that
\[ [F_{k+1}] = [\Sigma_{k+1}] - (k+1)[e]\textrm{ in } H_2(X_{k+1}\#\cptwo;\Z), \]
and hence $F_{k+1}\cdot F_{k+1} = -k(k+1)$.
\end{lemma}

\begin{proof}
When $k=0$, we  can see  that $H_2(W_1;\Z) \cong \Z$  by turning $W_1$ upside-down so that it is obtained  from the homology sphere $-S^3_1(K)$ by attaching a $0$-framed $2$-handle.  The inclusions $S^3_0(K) \hookrightarrow X_0$ and $S^3_0(K) \hookrightarrow W_1$ induce isomorphisms $\Z \to \Z$ on second homology, so the generators $[\Sigma_0]$ and $[F_1]$ are both the images of a generator of $H_2(S^3_0(K))$, hence they agree up to sign in $H_2(X_1\#\cptwo)$.  We can therefore take $[F_1] = [\Sigma_0] = [\Sigma_1] - [e]$.

Let us  assume from here on that $k \geq 1$.  The Mayer--Vietoris sequence for $X_k \cup W_{k+1}$, together with the fact that $H_2(S^3_k(K);\Z)=0$, gives an exact sequence
\begin{equation}\label{eqn:mvexact} 0 \to H_2(X_k) \oplus H_2(W_{k+1}) \to H_2(X_{k+1}\#\cptwo) \to \Z/k\Z. \end{equation}
A generator $[F_{k+1}]$ of $H_2(W_{k+1})$ must be orthogonal to $[\Sigma_k] = [\Sigma_{k+1}]-[e]$ in $H_2(X_{k+1}\#\cptwo)$ since they can be represented by disjoint surfaces, the former in $W_{k+1}$ and the latter in $X_k$.  Thus, $[F_{k+1}]$ must be an integer multiple of $[\Sigma_{k+1}]-(k+1)[e]$, which has self-intersection $-k(k+1)$.
Now, the intersection forms on $H_2(X_k) \oplus H_2(W_{k+1})$ and $H_2(X_{k+1}\#\cptwo)$ have Gram matrices \[\left(\begin{smallmatrix} k & 0 \\ 0 & [F_{k+1}]^2\end{smallmatrix}\right)\textrm{ and }\left(\begin{smallmatrix} k+1 & 0 \\ 0 & -1\end{smallmatrix}\right),\] respectively, with determinants $[F_{k+1}]^2 \cdot k$ and $-(k+1)$.  The former is a sublattice of the latter of some index $I \leq k$, by \eqref{eqn:mvexact}. We therefore have that 
\[ k(k+1)\cdot k \leq |[F_{k+1}]^2 \cdot k| = |I^2(-(k+1))| \leq k^2(k+1). \]
Since the left and right sides are equal,  we must have $[F_{k+1}]^2 = -k(k+1)$ as claimed, and up to reversing orientation it follows that $[F_{k+1}] = [\Sigma_{k+1}] - (k+1)[e]$.
\end{proof}

Noting that $W_{k+1}$ has even intersection form, let us define 
\[ s_{k+1,j}: H_2(W_{k+1};\Z) \to \Z \]
to be the unique homomorphism sending $[F_{k+1}]$ to $2j$.  In order to understand the images of the  $z_{k,i}$ under the map induced by $W_{k+1}$, we  first describe the composition
\begin{equation} \label{eq:compose-x-w}
I^\#(W_{k+1},\omega_{k+1};s_{k+1,j}) \circ I^\#(X_k,\nu_k;t_{k,i})
\end{equation}
in terms of the maps induced by $X_{k+1}$, as per the lemma below.

\begin{lemma} \label{lem:compose-x-w-one-s}
For $k=0$, there is a  constant $\epsilon_0 = \pm1$ such that the composition \eqref{eq:compose-x-w} is equal to
\[ \frac{\epsilon_0}{2} \cdot (-1)^i\left( I^\#(X_1,\nu_1;t_{1,i}) + I^\#(X_1,\nu_1;t_{1,i+1}) \right) \]
when $j=i$, and it is zero for $j\neq i$.
For $k\geq 1$, the composition \eqref{eq:compose-x-w} is instead equal to
\[ \begin{cases}
\phantom{-}\frac{\epsilon_k}{2} I^\#(X_{k+1},\nu_{k+1};t_{k+1,i}) & j=i \\
-\frac{\epsilon_k}{2} I^\#(X_{k+1},\nu_{k+1};t_{k+1,i+1}) & j=i-k, 
\end{cases} \]
where $\epsilon_k = \pm1$ depends only on $k$, and it is zero otherwise.
\end{lemma}

\begin{proof}

We first aim to understand when the composition  \eqref{eq:compose-x-w} is nonzero.  Scaduto showed in \cite[\S3.4]{scaduto} that  $(\nu_k \cup \omega_{k+1}) \cdot e$ is odd in  \[X_k \cup_{S^3_k(K)} W_{k+1} \cong X_{k+1} \# \cptwo,\] so, up to homology, we can write
\[ \nu_k \cup \omega_{k+1} = \nu'_{k+1} \cup a_k e \]
for some properly embedded surface $\nu'_{k+1} \subset X_{k+1}$ and some odd $a_k$.  We then have that
\[ [\nu_k]\cdot[\Sigma_k] =  [\nu_k\cup \omega_{k+1}]\cdot[\Sigma_k]=[\nu'_{k+1} \cup a_ke] ([\Sigma_{k+1}]-[e]) \equiv [\nu'_{k+1}]\cdot[\Sigma_{k+1}] + 1 \pmod{2}. \]
Combining this with \eqref{eq:nu-dot-sigma}  gives
\begin{equation}\label{eqn:nu'mod2} [\nu'_{k+1}\cup a_ke] \equiv \begin{cases} [\nu_1] + [e] + [\Sigma_1] & k=0 \\ [\nu_{k+1}] + [e] & k \geq 1 \end{cases} \pmod{2}. \end{equation}
 Theorem~\ref{thm:main-cobordism-decomposition} says that the composition \eqref{eq:compose-x-w} is equal to
\begin{equation}\label{eqn:sumscomp} \sum_s I^\#(X_{k+1} \# \cptwo,\nu'_{k+1}+a_ke; s), \end{equation}
where we sum over all \[s: H_2(X_{k+1}\#\cptwo;\Z)\to\Z\] with $s|_{X_k} = t_{k,i}$ and $s|_{W_{k+1}} = s_{k+1,j}$, meaning that \begin{equation}\label{eqn:sSF}s([\Sigma_k]) = 2i-k\textrm{ and }s([F_{k+1}]) =2j.\end{equation} Moreover, properties \eqref{i:adjunction} and \eqref{i:blowup-formula} of Theorem~\ref{thm:main-cobordism-decomposition} say that each summand is zero unless \begin{equation}\label{eqn:sblowup}s([\Sigma_{k+1}]) \equiv k+1 \textrm{ (mod 2)} \textrm{ and }s([e]) = \pm 1.\end{equation} The constraints in \eqref{eqn:sblowup} imply that we can write each $s$ which might contribute a nonzero summand in \eqref{eqn:sumscomp} as 
\[ s = t_{k+1,\ell} + cE, \]
where $E$ is Poincar\'e dual to $[e]$ and $c=\pm 1$.  Using Lemma~\ref{lem:generator-w}, we compute
\begin{align}
\label{eqn:sSell}s([\Sigma_k]) = s([\Sigma_{k+1}] - [e]) &= (2\ell - (k+1)) + c \\
s([F_{k+1}]) = s([\Sigma_{k+1}] - (k+1)[e]) &= (2\ell - (k+1)) + (k+1)c.
\end{align}
By \eqref{eqn:sSF}, this implies that
\[ 2j - (2i-k) = s([F_{k+1}]) - s([\Sigma_k]) = kc, \]
or, equivalently, that  $k(c-1) = 2(j-i)$.

Suppose first that $k\neq 0$ and $c=-1$.  Let $\alpha_k$ be the class in $H_2(X_{k+1}\#\CP^2;\Z)$ satisfying
\[ 2\alpha_k = [\nu'_{k+1} \cup a_ke] - ([\nu_{k+1}]+[e]) \] (such a class exists by \eqref{eqn:nu'mod2}).
Then $k=i-j$ and \[2i-k = s([\Sigma_k]) = 2(\ell-1) - k,\] where the latter equality is by  \eqref{eqn:sSell}, so that $\ell=i+1$.  Remark~\ref{rem:nu-plus-even} and the blow-up formula (property \eqref{i:blowup-formula} of Theorem~\ref{thm:main-cobordism-decomposition}) then say that
\begin{align*}
I^\#(X_{k+1}\#\cptwo,\nu'_{k+1}+a_ke;t_{k+1,i+1}-E) &= -\epsilon_k I^\#(X_{k+1}\#\cptwo,\nu_{k+1}+e;t_{k+1,i+1}-E) \\
&= -\frac{\epsilon_k}{2} I^\#(X_{k+1},\nu_{k+1};t_{k+1,i+1}),
\end{align*}
where $\epsilon_k = (-1)^{\alpha_k\cdot\alpha_k+1}$.  Similarly, if either $k=0$ or $c=1$ then we must have $i=j$, and also
\[ 2i-k = s([\Sigma_k]) = (2\ell-(k+1)) + c =  \begin{cases} 2\ell - k & c=1 \\ 2(\ell-1) - k & (k,c)=(0,-1). \end{cases} \]
Thus, either $c=1$ and $\ell = i$, or $(k,c)=(0,-1)$ and $\ell=i+1$.  In the first case, if $k \geq 1$ then the relevant map is
\[ I^\#(X_{k+1}\#\cptwo,\nu'_{k+1}+ a_ke;t_{k+1,i}+E) = \frac{\epsilon_k}{2} I^\#(X_{k+1},\nu_{k+1};t_{k+1,i}), \]
and if $k = 0$ then it is
\begin{equation} \label{eq:x1-nu1-1}
I^\#(X_1\#\cptwo,\nu'_1+ a_0e;t_{1,i}+E).
\end{equation}
In the second case the relevant map is
\begin{equation} \label{eq:x1-nu1-2}
I^\#(X_1\#\cptwo,\nu'_1+ a_0e;t_{1,i+1}-E).
\end{equation}

This completes the proof of the lemma except for the signs when $k=0$.  We let $\alpha_0$ be the class in $H_2(X_1\#\CP^2;\Z)$ satisfying
\[ 2\alpha_0 = [\nu'_1 + a_0e] - ([\nu_1] + [e] + [\Sigma_1]) \] (which exists by \eqref{eqn:nu'mod2}),
and set $\epsilon_0 = (-1)^{\alpha_0\cdot\alpha_0}$.  Then
\begin{align*}
I^\#(X_1\#\cptwo,\nu'_1+ a_0e;t_{1,j}\pm E) &= \epsilon_0 I^\#(X_1\#\cptwo,\nu_1+e+\Sigma_1;t_{1,j}\pm E) \\
&= \mp \frac{\epsilon_0}{2} I^\#(X_1,\nu_1+\Sigma_1;t_{1,j}) \\
&= \mp \frac{\epsilon_0}{2} \cdot (-1)^\sigma I^\#(X_1,\nu_1;t_{1,j})
\end{align*}
where \[\sigma = \frac{1}{2}(t_{1,j}(\Sigma_1)+\Sigma_1\cdot\Sigma_1)+\nu_1\cdot\Sigma_1 = j + (\nu_1\cdot\Sigma_1).\]  Since \eqref{eq:nu-dot-sigma} says that $\nu_1\cdot\Sigma_1 \equiv 1\pmod{2}$, we have shown that
\[ I^\#(X_1\#\cptwo,\nu'_1+ a_0e;t_{1,j}\pm E) = \pm \frac{\epsilon_0}{2}\cdot (-1)^j I^\#(X_1,\nu_1;t_{1,j}) \]
and so we combine this with equations \eqref{eq:x1-nu1-1} and \eqref{eq:x1-nu1-2} to complete the proof.
\end{proof}

\begin{proof}[Proof of Proposition~\ref{prop:w-z-image}]
Sum \eqref{eq:compose-x-w} over all $j$, applying Lemma~\ref{lem:compose-x-w-one-s} to determine each composition, and evaluate the result on the generator ${\bf 1} \in I^\#(S^3)$.
\end{proof}

\subsection{Instanton L-space knots are fibered with $g=g_s$} Our goal in this section is to prove Theorem~\ref{thm:l-space-knots-are-fibered}, which  asserts that  instanton L-space knots are fibered with Seifert genus equal to smooth slice genus.
The proposition below is the first of  two main ingredients in the proof of this theorem. Adopting the notation from above, recall that we defined \[ y_i := z_{0,i} = I^\#(X_0,\nu_0;t_{0,i})({\mathbf 1}) \in I_\godd^\#(S^3_0(K),\mu; s_i), \] and that \[ I^\#(S^3_0(K),\mu;s_i) = 0 \quad \mathrm{for\ }|i| > g-1, \] by the adjunction inequality of Theorem~\ref{thm:eigenspace-decomposition}.

\begin{proposition} \label{prop:l-space-eigenspaces}
If $K \subset S^3$ as above is a nontrivial instanton L-space knot then
\[ I^\#_\godd(S^3_0(K),\mu;s_i) = \C \cdot y_i \]
for all $i$.  In particular,
\[ \dim I^\#_\godd(S^3_0(K),\mu;s_i) = 0 \mathrm{\ or\ } 1 \]
depending on whether $y_i=0$ or $y_i \neq 0$, respectively.
\end{proposition}

To prepare for the proof of Proposition~\ref{prop:l-space-eigenspaces}, we let $(V_k,\bar\omega_k)$ denote the composition \[(V_k,\bar\omega_k)=(W_k,\omega_k)\circ\dots (W_1,\omega_1): = (S^3_0(K),\mu)\to (S^3_k(K),0),\] which induces the map \[ I^\#(V_k,\bar\omega_k) = I^\#(W_k,\omega_k) \circ \dots \circ I^\#(W_1,\omega_1): I^\#(S^3_0(K),\mu) \to I^\#(S^3_k(K)), \] and define for each integer $i$ the element
\[ c_{k,i} := I^\#(V_k,\bar\omega_k)(y_i) \in I^\#(S^3_k(K)). \]
Repeated application of Proposition~\ref{prop:w-z-image} then tells us that
\[ 2^k c_{k,i} = \sum_{j=0}^k d_{j,i} z_{k,i+j}, \]
where the coefficient $d_{j,i}$ is a sum of $k \choose j$ signs.  In particular, we have
\begin{equation} \label{eq:c-in-terms-of-z}
2^k c_{k,i} = \pm z_{k,i} + \sum_{j=1}^k d_{j,i} z_{k,i+j},
\end{equation}
and since this system of equations has an invertible triangular matrix of coefficients, we see that each $z_{k,i}$ is a linear combination of the various $c_{k,j}$.

\begin{lemma} \label{lem:ker-vn}
For all integers $n \geq 1$, the kernel of the map
\[ I^\#(V_n,\bar\omega_n) = I^\#(W_n,\omega_n) \circ \dots \circ I^\#(W_1,\omega_1): I^\#(S^3_0(K),\mu) \to I^\#(S^3_n(K)) \]
lies in the span of the elements $y_{1-g},\dots,y_{g-1}$.  It is equal to this span if $n \geq 2g-1$.
\end{lemma}

\begin{proof}
The case $n=1$ follows from the exact triangle \eqref{eq:triangle-0-cobordisms}, since
\[ \ker I^\#(V_1,\bar\omega_1) = \ker I^\#(W_1,\omega_1) =\Img I^\#(X_0,\nu_0) \]
is spanned by $I^\#(X_0,\nu_0)({\mathbf 1}) = \sum_i y_i$.

Suppose now that $n \geq 1$.  Suppose the kernel of $I^\#(V_n,\bar\omega_n)$ lies in the span of  the elements $y_{1-g},\dots,y_{g-1}$. We will deduce that the same is true of the kernel of $I^\#(V_{n+1},\bar\omega_{n+1})$, which will then prove the first part of the lemma for all $n$, by induction.

 From the exact triangle \eqref{eq:triangle-n-cobordisms}, the kernel of $I^\#(W_{n+1},\omega_{n+1})$ is generated by
\[ I^\#(X_n,\nu_n)({\mathbf 1}) = \sum_{i=1-g+n}^{g-1} z_{n,i} \]
(here, the range of indices in the sum comes from \eqref{eq:t_i-adjunction}); since each $z_{n,i}$ is a linear combination of the $c_{n,j}$, as argued above, we may write this element as
\[ I^\#(X_n,\nu_n)({\mathbf 1}) = \sum_i z_{n,i} = \sum_j a_{n,j} c_{n,j} = I^\#(V_n,\bar\omega_n)\left(\sum_j a_{n,j}y_j\right) \]
for some coefficients $a_{n,j}$.  In particular, any element $x$ of the kernel of
\[ I^\#(V_{n+1},\bar\omega_{n+1}) = I^\#(W_{n+1},\omega_{n+1}) \circ I^\#(V_n,\bar\omega_n) \]
satisfies $I^\#(V_n,\bar\omega_n)(x) = c\cdot I^\#(V_n,\bar\omega_n)\left(\sum_j a_{n,j}y_j\right)$ for some constant $c$, and then
\[ x - c\sum_j a_{n,j}y_j \in \ker I^\#(V_n,\bar\omega_n). \] It follows by induction that $x$ is a linear combination of $y_{1-g}, \dots, y_{g-1}$, as desired.

Now suppose $n \geq 2g-1$. It remains to show that $I^\#(V_n,\bar\omega_n)(y_i)=0$ for all $i$ in this case.  We can write \[I^\#(V_n,\bar\omega_n)(y_i) = c_{n,i}\] as a linear combination of the elements \[z_{n,j} = I^\#(X_n,\nu_n;t_{n,j})({\mathbf 1}),\] as before. The adjunction inequality says that the map $I^\#(X_n,\nu_n;t_{n,j})$ is zero for $n \geq 2g-1$, so $z_{n,j}=0$ for all $j$ and hence $c_{n,i}=0$ for all $i$ as well.
\end{proof}

\begin{proof}[Proof of Proposition~\ref{prop:l-space-eigenspaces}]
Suppose for some $i$ that there is an element
\[ x \in I^\#_\godd(S^3_0(K),\mu;s_i) \]
which is not a multiple of $y_i$.  Then $x$ is not in the span of $y_{1-g},\dots,y_{g-1}$ since the \[y_j\in I^\#_\godd(S^3_0(K),\mu;s_j)\] belong to different direct summands. Now, each map \[I^\#(V_n,\bar\omega_n): I^\#(S^3_0(K),\mu) \to I^\#(S^3_n(K))\] has even degree, since it is a composition of the even-degree maps $I^\#(W_k,\omega_k)$ for $1 \leq k \leq n$.  Thus, we have that
\[ I^\#(V_n,\bar\omega_n)(x) \in I^\#_\godd(S^3_n(K)) \]
for all $n \geq 1$, and since $x$ is not in the span of $y_{1-g},\dots,y_{g-1}$, Lemma~\ref{lem:ker-vn} tells us that each $I^\#(V_n,\bar\omega_n)(x)$ must be nonzero.  In particular, we have that
\[ \dim I^\#_\godd(S^3_n(K)) > 0 \]
for all $n \geq 1$, and so none of the $S^3_n(K)$ can be an instanton L-space.  This is a contradiction, since if $K$ is an instanton L-space knot then it has some positive integral L-space surgery, by Remark \ref{rmk:integralLspace}; so each $I^\#_\godd(S^3_0(K),\mu;s_i)$ must in fact be spanned by $y_i$.
\end{proof}

\begin{remark} \label{rem:l-space-eigenspaces-bundle}
Nothing in the proof of Proposition~\ref{prop:l-space-eigenspaces} requires that we use the nontrivial bundle on $S^3_0(K)$  specified by $\mu$; we only need this  to conclude in Theorem~\ref{thm:l-space-knots-are-fibered} that $K$ is fibered.  If we use elements $\tilde{y}_i \in I^\#(S^3_0(K);s_i)$ instead and replace the bundle $\nu_0$ on $X_0$ by $\tilde{\nu}_0$ accordingly, none of the argument changes save  perhaps some signs.  Thus, it is also true that if $K$ is a nontrivial instanton L-space knot then each $\C$-module
\[ I^\#_\godd(S^3_0(K);s_i) \]
is  spanned by the single element $\tilde{y}_i = I^\#(X_0,\tilde{\nu}_0;t_{0,i})({\mathbf 1})$.
\end{remark}

The second main ingredient in the proof of Theorem~\ref{thm:l-space-knots-are-fibered} is the following improvement to Proposition~
\ref{prop:l-space-eigenspaces}. Namely, Theorem~\ref{thm:eigenspace-decomposition} tells us that $I^\#_\godd(S^3_0(K),\mu;s_i) = 0$ for all $|i| \geq g(K)$, but the proposition below says that we can sharpen this bound if we know that $K$ has an instanton L-space surgery.

\begin{proposition} \label{prop:l-space-slice-bound}
Suppose $K\subset S^3$ as above is a nontrivial instanton $L$-space knot with smooth slice genus $g_s(K)$. Then 
\[ I^\#_\godd(S^3_0(K),\mu;s_i) = 0 \]
for all $i \geq \max(g_s(K),1)$.
\end{proposition}

\begin{proof}
Let \[g'_s = \max(g_s(K),1).\]  Then the cobordism $X_1$ contains a smoothly embedded, closed surface $S$ homologous to $\Sigma_1$, built by gluing the core of the 2-handle used to form $X_1$ to a genus-$g'_s$ surface in $S^3\times[0,1]$ with boundary $K\times\{1\}$.  We can  thus improve the inequality \eqref{eq:t_i-adjunction} by replacing $g=g(\Sigma_1)$ with $g'_s=g(S)$, so that
\[ I^\#(X_1,\nu_1;t_{1,i}) = 0 \quad\mathrm{for\ all}\quad i \geq g'_s. \]
This implies, in particular,  that \[z_{1,i} = I^\#(X_1,\nu_1;t_{1,i})({\bf 1})=0\] for all $i \geq g'_s$.  By Proposition~\ref{prop:w-z-image}, we have
\[ I^\#(W_1,\omega_1)(y_i) = \pm \frac{1}{2}(z_{1,i} + z_{1,i+1}), \]
and so
\begin{equation} \label{eq:W-y_i-zero}
I^\#(W_1,\omega_1)(y_i) = 0
\end{equation}
for all $i \geq g'_s$.

Now fix some $i \geq g'_s$, and suppose that $y_i \neq 0$.  Theorem~\ref{thm:eigenspace-decomposition} says that
\[ I^\#(S^3_0(K),\mu;s_i) \cong I^\#(S^3_0(K),\mu;s_{-i}) \]
as $\Z/2\Z$-graded $\C$-modules, so $I^\#_\godd(S^3_0(K),\mu;s_{-i})$ is also nonzero; by Proposition~\ref{prop:l-space-eigenspaces}, it must be spanned by $y_{-i}$, hence $y_{-i} \neq 0$ as well.  (We note that $y_i \neq y_{-i}$, because $i \geq g'_s \geq 1$.)  The exactness of the triangle \eqref{eq:triangle-0-cobordisms}, together with \eqref{eq:W-y_i-zero}, says that
\[ y_i \in \img(I^\#(X_0,\nu_0)) = \C\cdot I^\#(X_0,\nu_0)({\mathbf 1}), \]
so that $y_i = c\left(\sum_j y_j\right)$ for some $c\neq 0$, or, equivalently,
\[ (c-1)y_i + \sum_{\substack{j\neq i\\y_j\neq0}} cy_j = 0. \]
But the nonzero $y_j$ are linearly independent, as they belong to different direct summands \[y_j\in I^\#_\godd(S^3_0(K),\mu;s_j),\] and the left hand side above is not identically zero because the term $cy_{-i}$ is nonzero.  This is a contradiction, so $y_i = 0$ after all.
\end{proof}

Finally, we apply the fact that framed instanton homology of $0$-surgery detects fiberedness (Theorem~\ref{thm:odd-dim-g-1}) to obtain the promised  restrictions on instanton L-space knots.

\begin{theorem} \label{thm:l-space-knots-are-fibered}
Suppose  $K \subset S^3$ is a nontrivial instanton L-space knot. Then $K$ is fibered, and its Seifert genus is equal to its smooth slice genus.
\end{theorem}

\begin{proof}
Theorem~\ref{thm:odd-dim-g-1} and Proposition \ref{prop:l-space-eigenspaces} combine to say that 
\[ 1 \leq \dim I^\#_\godd(S^3_0(K),\mu;s_{g(K)-1}) \leq 1, \]
respectively, and since the first inequality is then an equality we conclude that $K$ must be fibered.  For the claim about the slice genus of $K$, we deduce from Proposition~\ref{prop:l-space-slice-bound} that
\[ g(K)-1 < \max(g_s(K),1) \]
since $I^\#_\godd(S^3_0(K),\mu;s_{g(K)-1})$ is nonzero, and hence that
\[ g_s(K) \leq g(K) \leq \max(g_s(K),1). \]
If $g_s(K)=0$ then we have $g(K) \leq 1$, but the only fibered knots of genus $1$ are the trefoils and the figure eight, which satisfy $g_s(K) = g(K) = 1$.  Otherwise $\max(g_s(K),1) = g_s(K)$ and so we must have $g(K) = g_s(K)$.
\end{proof}

\subsection{The first L-space slope}
\label{ssec:first-slope}

Supposing $K$ is a nontrivial instanton L-space knot, our goal in this subsection is to  determine in terms of \[\dim I^\#(S^3_0(K),\mu)\]   which rational surgeries on $K$ are instanton L-spaces (see Proposition \ref{prop:first-l-space-slope}).  We begin with the following lemma, which is a direct analogue of \cite[Proposition~7.2]{kmos}.

\begin{lemma} \label{lem:canceling-2-handles}
Suppose $K \subset S^3$ is a knot. For all $k \geq 0$, let
\[ (X'_{k+1},\eta_{k+1}): S^3_{k+1}(K) \to S^3 \]
be the cobordisms which induce the unlabeled maps in the exact triangles \eqref{eq:triangle-0-cobordisms} (for $k=0$) and \eqref{eq:triangle-n-cobordisms} (for $k \geq 1$).  Then the composition
\[ I^\#(S^3_{k+1}(K)) \xrightarrow{I^\#(X'_{k+1},\eta_k)} I^\#(S^3) \xrightarrow{I^\#(X_{k+1},\nu_{k+1})} I^\#(S^3_{k+1}(K)) \]
is zero for all $k \geq 0$.
\end{lemma}

\begin{proof}
This is essentially Lemma~4.13 and Remark~4.14 of \cite{bs-stein}, which in turn follows the proof of \cite[Proposition~6.5]{km-embedded2}.   The composition is induced by a cobordism
\[ X: S^3_{k+1}(K) \to S^3 \to S^3_{k+1}(K), \]
in which we attach a $0$-framed $2$-handle $H_\mu$ to $S^3_{k+1}(K) \times [0,1]$ along a meridian $\mu$ of $K \times \{1\}$ and then attach a $(k+1)$-framed $2$-handle $H_K$ to $K$ in the resulting $S^3$.  The cobordism $X$ contains a smoothly embedded $2$-sphere $S$ of self-intersection zero, as the union of a cocore of $H_\mu$ and a core of $H_K$.

We wish to apply Proposition~\ref{prop:adjunction-inequality} to $S$, so we must construct a surface $F$ with $[S]\cdot[F] \neq 0$.  The disjoint union of $k+1$ parallel cores of $H_\mu$ is bounded by a nullhomologous link in $S^3_{k+1}(K) \times \{1\}$, namely $k+1$ disjoint copies of $K$, so let $F$ be the union of these cores with a Seifert surface of $(k+1)K$ and then we have $[S] \cdot [F] = k+1 > 0$ as desired.
\end{proof}

\begin{lemma} \label{lem:induct-injectivity}
If $I^\#(X_{k+1},\nu_{k+1})$ is injective for some $k \geq 0$, then so is $I^\#(X_k,\nu_k)$.
\end{lemma}

\begin{proof}
The injectivity of $I^\#(X_{k+1},\nu_{k+1})$ implies by Lemma~\ref{lem:canceling-2-handles} that $I^\#(X'_{k+1},\eta_k)$ is zero; hence, $I^\#(X_k,\nu_k)$ is injective by the exactness of either \eqref{eq:triangle-0-cobordisms} or \eqref{eq:triangle-n-cobordisms}, depending on whether $k=0$ or $k \geq 1$.
\end{proof}

We can now prove the main proposition of this subsection.

\begin{proposition} \label{prop:first-l-space-slope}
Let $K\subset S^3$ be a nontrivial instanton L-space knot.  Then the following are equivalent for any rational $r>0$:
\begin{enumerate}
\item $r \geq \dim I^\#_\godd(S^3_0(K),\mu)$, \label{i:slope-vs-odd-dim}
\item $S^3_r(K)$ is an instanton L-space, \label{i:slope-is-l-space}
\item the map \[I^\#(X_m,\nu_m): I^\#(S^3) \to I^\#(S^3_m(K))\] is zero for $m=\lfloor r\rfloor$, where we interpret the codomain as  $I^\#(S^3_0(K),\mu)$ when $m=0$.\label{i:xm-is-zero}
\end{enumerate}
In particular, $S^3_{2g(K)-1}(K)$ is an instanton L-space.
\end{proposition}

\begin{proof}
Let $N \geq 1$ be the smallest  integral  L-space slope, which is positive by Remark \ref{rmk:integralLspace}.  Moreover, this remark says that $S^3_r(K)$ is an instanton L-space iff $r\geq N$. Since $S^3_{N-1}(K)$ is not an instanton L-space, the map $I^\#(X_{N-1},\nu_{N-1})$ appearing in the exact triangle
\[ \dots \to I^\#(S^3) \xrightarrow{I^\#(X_{N-1},\nu_{N-1})} I^\#(S^3_{N-1}(K)) \xrightarrow{I^\#(W_N,\omega_N)} I^\#(S^3_N(K)) \to \dots \]
must be injective: otherwise it would be zero, leading to
\[ \dim I^\#(S^3_{N-1}(K)) = \dim I^\#(S^3_N(K)) - 1 = N-1, \]
which is a contradiction.  (When $N=1$, we interpret $I^\#(S^3_{N-1}(K))$ here as $I^\#(S^3_0(K),\mu)$; this has positive dimension because \[I^\#_\godd(S^3_0(K),\mu;s_{g-1})\] is nonzero by Theorem~\ref{thm:odd-dim-g-1}.) 
Lemma~\ref{lem:induct-injectivity} now says that $I^\#(X_k,\nu_k)$ is injective for $0 \leq k \leq N-1$.  This shows that (3)$\implies$(2), since if $S^3_r(K)$ is not an instanton L-space then \[0\leq m=\lfloor r \rfloor\leq N-1,\] and we have argued in this case that $I^\#(X_m,\nu_m)$ is nontrivial. 

The argument above also shows, by the exactness of  \eqref{eq:triangle-0-cobordisms} and \eqref{eq:triangle-n-cobordisms}, that  the  maps
\[ I^\#(W_{k+1},\omega_{k+1}), \quad 0 \leq k \leq N-1 \]
are   surjective with 1-dimensional kernel; the kernel is the image of $I^\#(S^3)$ under the odd-degree map $I^\#(X_k,\nu_k)$, and hence is supported entirely in odd grading.  The composition
\[ I^\#(V_N,\bar\omega_N): I^\#(S^3_0(K),\mu) \to I^\#(S^3_N(K)) \]
is thus surjective as well, with kernel an $N$-dimensional subspace of $I^\#_\godd(S^3_0(K),\mu)$.  Since $I^\#(V_N,\bar\omega_N)$ has even degree, and $I^\#_\godd(S^3_N(K))=0$, we must have
\[ I^\#_\godd(S^3_0(K),\mu) = \ker I^\#(V_N,\bar\omega_N), \]
and in particular $N = \dim I^\#_\godd(S^3_0(K),\mu)$. This shows that (1)$\iff$(2), since $S^3_r(K)$ is an instanton L-space iff $r\geq N$.

We also know that  the map $I^\#(X_N,\nu_N)$ must be zero, because otherwise we deduce from the corresponding exact triangle that
\[ \dim I^\#(S^3_{N+1}(K)) = \dim I^\#(S^3_N(K)) - 1 = N-1 < N+1, \]
which is impossible since $I^\#(S^3_{N+1}(K))$ has Euler characteristic $N+1$.  Thus,  Lemma~
\ref{lem:induct-injectivity} implies that $I^\#(X_k,\nu_k) = 0$ for all $k \geq N$. This shows that (2)$\implies$(3), since if $S^3_r(K)$ is an instanton L-space then $\lfloor r \rfloor\geq N.$

For the final claim that $S^3_{2g-1}(K)$ is an instanton L-space, we know that each eigenspace \[I^\#_\godd(S^3_0(K),\mu;s_j)=0\]  for $|j| \geq g(K)$, and each of the $2g(K)-1$ eigenspaces with $|j| \leq g(K)-1$ has dimension at most 1 by Proposition~\ref{prop:l-space-eigenspaces}.  We therefore have that \[\dim I^\#_\godd(S^3_0(K),\mu)\leq 2g(K)-1,\] and we apply the equivalence  \eqref{i:slope-vs-odd-dim}$\iff$\eqref{i:slope-is-l-space} to complete the proof of the claim.
\end{proof}

Recall from Proposition~\ref{prop:l-space-eigenspaces} that if $K$ is a nontrivial instanton L-space knot then
\[ I^\#_\godd(S^3_0(K),\mu;s_i) = \C\cdot y_i \]
for all $i$, where \[y_i = I^\#(X_0,\nu_0;t_{0,i})({\mathbf 1}).\]  Again, the nonzero  $y_i$ are linearly independent because they belong to different eigenspaces, so Proposition~\ref{prop:first-l-space-slope} says that the smallest instanton L-space slope is precisely the number of $y_i$ that are nonzero. Note  that
\[ y_i \neq 0 \mathrm{\ iff\ } y_{-i} \neq 0 \] for each $i$,
since the eigenspaces $I^\#_\godd(S^3_0(K),\mu;s_{\pm i})$ are isomorphic  by Theorem~\ref{thm:eigenspace-decomposition} and are spanned by $y_{\pm i}$ by  Proposition~\ref{prop:l-space-eigenspaces}.  In particular, Theorem~\ref{thm:odd-dim-g-1} implies that 
\begin{equation}\label{eqn:bothnonzero}y_{g(K)-1} \neq 0 \textrm{ and } y_{1-g(K)} \neq 0 \end{equation} in this case.

\begin{proposition} \label{prop:2-surgery-l-space}
Suppose  $K\subset S^3$ is a nontrivial knot, and that either $S^3_1(K)$ or $S^3_2(K)$ is an instanton L-space.  Then $K$ is the right-handed trefoil.
\end{proposition}

\begin{proof}
Let $g=g(K)$, and suppose first that $S^3_1(K)$ is an instanton L-space.  Proposition~\ref{prop:first-l-space-slope} implies that \[\dim I^\#_\godd(S^3_0(K),\mu) \leq 1,\] and yet $y_{g-1}$ and $y_{1-g}$ are both nonzero as in \eqref{eqn:bothnonzero}.  This is only possible if $s_{g-1}=s_{1-g}$, or equivalently if $g=1$.  Since $K$ has genus $1$ and it is fibered by Theorem~\ref{thm:l-space-knots-are-fibered}, it can only be a trefoil or the figure-eight knot.  But in \cite[\S4]{bs-stein} we noted that 1-surgery on the left-handed trefoil and figure-eight are Seifert fibered and not the Poincar\'e homology sphere, so these are not instanton L-spaces by \cite[Corollary~5.3]{bs-stein}.  Thus, $K$ is the right-handed trefoil.

Now suppose instead that $S^3_2(K)$ is an instanton L-space, but $S^3_1(K)$ is not.  In this case Proposition~\ref{prop:first-l-space-slope} tells us that $g \geq 2$, and that
\[ \dim I^\#_\godd(S^3_0(K),\mu) = 2. \]
The elements $y_{g-1}$ and $y_{1-g}$ are nonzero, and they lie in distinct eigenspaces since $g \geq 2$, so all of the other $I^\#_\godd(S^3_0(K),\mu;s_j)$ must vanish, hence
\[ y_j=0 \mathrm{\ for\ }2-g \leq j \leq g-2. \]
We now apply Proposition~\ref{prop:w-z-image} to say that
\[ I^\#(W_1,\omega_1)\left(\sum_{j=0}^{g-2} y_{2-g+2j}\right) = \frac{\epsilon_0}{2}\cdot \sum_{j=0}^{g-2} (-1)^{2-g+2j}(z_{1,2-g+2j}+z_{1,3-g+2j}), \]
or equivalently 
\begin{align*}
I^\#(W_1,\omega_1)\big(y_{2-g}+y_{4-g}+y_{6-g}+\dots+y_{g-2}\big) &= \frac{\epsilon_0}{2}(-1)^{2-g}\left( \sum_{j=2-g}^{g-1} z_{1,j} \right) \\
&= \pm\frac{1}{2} I^\#(X_1,\nu_1)({\bf 1}).
\end{align*}
Now the left side vanishes since each of the $y_j$ appearing there is zero, so it follows that the map $I^\#(X_1,\nu_1)$ is zero as well.  But then Proposition~\ref{prop:first-l-space-slope} says that $S^3_1(K)$ is an instanton L-space, and this is a contradiction.
\end{proof}

\begin{corollary} \label{cor:genus-1-l-space}
Suppose $K\subset S^3$ is an instanton L-space knot of genus 1. Then $K$ is the right-handed trefoil.
\end{corollary}

\begin{proof}
Proposition~\ref{prop:first-l-space-slope} says that $S^3_1(K)$ must be an instanton L-space, so we apply Proposition~\ref{prop:2-surgery-l-space}.
\end{proof}

\begin{remark}The proof of Proposition~\ref{prop:2-surgery-l-space} also shows that if $g=g(K)$ is odd and greater than 2, then $S^3_3(K)$ is not an instanton L-space.  Indeed, we know from the proposition in this case that $S^3_2(K)$ is not an instanton L-space, so if $S^3_3(K)$ is, then Proposition~\ref{prop:first-l-space-slope} says that \[\dim I^\#_\godd(S^3_0(K),\mu) = 3,\] which implies, by symmetry, that the nonzero $y_j$ are precisely $y_{1-g},$ $y_0$, and $y_{g-1}$.
The fact that $g$ is odd then means that the element \[y_{2-g}+y_{4-g}+y_{6-g}+\dots+y_{g-2}\] is again zero, which implies that $I^\#(X_1,\nu_1)({\mathbf 1}) = 0$, as in the proof of Proposition~\ref{prop:2-surgery-l-space}. But this implies by the proposition that $S^3_1(K)$ is an instanton L-space, a contradiction.
\end{remark}

\section{Instanton L-space knots are strongly quasipositive} \label{sec:sqp}
 The goal of this section is to prove the  two theorems below, which, together with Theorem \ref{thm:l-space-knots-are-fibered},  will complete the proof of Theorem \ref{thm:main-surgery-reduction}. 
 
  \begin{theorem}
\label{thm:sqp} Suppose $K\subset S^3$ is a nontrivial instanton L-space knot. Then $K$ is strongly quasipositive.
\end{theorem}

\begin{theorem}
\label{thm:bound} Suppose $K\subset S^3$ is a nontrivial instanton L-space knot. Then $S^3_r(K)$ is an instanton L-space for some rational $r$ iff $r\geq 2g(K)-1$.
\end{theorem}

We will begin with and spend the most time proving Theorem \ref{thm:sqp}, proceeding as outlined in the introduction.
 Namely, if $K$ is an instanton L-space knot then it is fibered, by Theorem \ref{thm:l-space-knots-are-fibered}. The corresponding fibration specifies an open book decomposition of $S^3$, and hence a contact structure $\xi_K$ on $S^3$ by Thurston and Winkelnkemper's construction \cite{thurston-winkelnkemper}. To prove that $K$ is strongly quasipositive, it suffices to show  that  $\xi_K$ is the unique tight contact structure $\xi_{\std}$ on $S^3$, as  recorded in the proposition below.

\begin{proposition}
\label{prop:sqp} Suppose $K \subset S^3$ is a fibered knot.  Then $K$ is strongly quasipositive iff $\xi_K$ is tight, in which case the maximal self-linking number of $K$ is   $\maxsl(K) = 2g(K)-1$.
\end{proposition}

\begin{proof}
The relationship between strong quasipositivity and tightness was proved by Hedden in  \cite[Proposition~2.1]{hedden-positivity}.  Hedden's proof relies on the Giroux correspondence, but this is not necessary: it was given an alternate proof and generalized to fibered knots in other 3-manifolds by Baker--Etnyre--van Horn-Morris \cite[Corollary~1.12]{bev-cabling}. The  claim that if $\xi_K$ is tight then $\maxsl(K) = 2g(K)-1$ comes from the Bennequin inequality $\maxsl(K) \leq 2g(K)-1$ together with the elementary fact that the connected binding $B$ of an open book  supporting any contact structure $\xi$ is transverse in $\xi$ with self-linking number $2g(B)-1$.
\end{proof}

We will study   $\xi_K$ via cabling. Let   $K_{p,q}$ denote  the $(p,q)$-cable of $K$, where $p$ and $q$ are relatively prime and $q \geq 2$ (this cable is  the simple closed curve on $\partial N(K)\subset S^3$ representing the class $\mu^p\lambda^q$, for $\mu$ the meridian of $K$ and $\lambda$ the longitude). 
It is well-known that if $K$ is fibered then so is $K_{p,q}$, in which case we can talk about the corresponding contact structure $\xi_{K_{p,q}}$. For positive cables, this contact structure agrees with $\xi_K$, as below.

\begin{proposition}
\label{prop:cabletight} Suppose $K \subset S^3$ is a fibered knot.  Then $\xi_K\cong \xi_{K_{p,q}}$ for $p$ and $q$   positive.\end{proposition}

\begin{proof}
This was proven by Hedden in  \cite[Theorem~1.2]{hedden-cabling} using the Giroux correspondence.  An alternate proof which does not rely on this correspondence was given by Baader--Ishikawa in \cite[\S3]{baader-ishikawa}.
\end{proof}

We will also use  the following, which states that the $(p,q)$-cable of an instanton L-space knot is itself an instanton L-space knot for $\frac{p}{q}$ sufficiently large.

\begin{lemma} \label{lem:cable-l-space}
Suppose $K\subset S^3$ is a nontrivial knot, and fix positive, coprime integers $p$ and $q$ with $q \geq 2$ and $\frac{p}{q} > 2g(K)-1.$  Then $K$ is an instanton L-space knot iff $K_{p,q}$ is.
\end{lemma}

\begin{proof}
We note that $\frac{pq-1}{q^2} > \frac{p-1}{q} \geq 2g(K)-1$, and also that
\[ g(K_{p,q}) \leq g(T_{p,q}) + q\cdot g(K), \]
which after some rearranging yields 
\[ pq-1 \geq 2g(K_{p,q})-1 + q\left(\frac{p-1}{q} - (2g(K)-1)\right) \geq 2g(K_{p,q})-1. \]
Thus,  Proposition~\ref{prop:first-l-space-slope} implies that $K$ is an instanton L-space knot iff $\frac{pq-1}{q^2}$-surgery on $K$ is an instanton L-space, and likewise that $K_{p,q}$ is an instanton L-space knot iff $S^3_{pq-1}(K_{p,q})$ is an instanton L-space.  The lemma now follows immediately from the relation
\[ S^3_{pq-1}(K_{p,q}) \cong S^3_{(pq-1)/q^2}(K), \]
which was originally proved by Gordon \cite[Corollary~7.3]{gordon}.
\end{proof}

Our strategy in proving Theorem \ref{thm:sqp} is  roughly as follows: if $K$ is a nontrivial instanton L-space knot then so  is a  positive cable $K_{p,q}$ with  $\frac{p}{q}$ sufficiently large, by Lemma \ref{lem:cable-l-space}. We will use this together with the fact that $K_{p,q}$  is a Murasugi sum of the form $K_{1,q}\ast T_{p,q}$ to prove that the contact structure $\xi_{K_{1,q}}$ is tight (for $q=2$, though it holds more generally). This will imply that $\xi_K$ is tight and hence that $K$ is strongly quasipositive, by Propositions \ref{prop:cabletight} and \ref{prop:sqp}. Our proof that $\xi_{K_{1,q}}$ is tight  makes use of a variation of our instanton contact class from \cite{bs-instanton}, developed in the next  subsection.

\subsection{Open books and framed instanton homology}

In this subsection, we briefly describe a variant of the contact class
\[ \iinvt(\xi) \in \SHItfun(-M,-\Gamma) \]
which we constructed for sutured contact manifolds in \cite{bs-instanton}.  Here, we specialize our construction to closed contact 3-manifolds, as in \cite{bs-stein}, and define this variant as a subspace of framed instanton homology, rather than an element of sutured instanton homology up to rescaling. We explain the reason for this  in Remark \ref{rmk:whysubspace}.

The following definitions are all taken from \cite[\S2.3]{bs-stein}.

\begin{definition}
An \emph{abstract open book} is a triple $(S,h,{\bf c})$ consisting of a surface $S$ with nonempty boundary, a diffeomorphism $h: S\to S$ such that $h|_{\partial S} = \id_{\partial S}$, and a collection of disjoint, properly embedded arcs ${\bf c} = \{c_1,\dots,c_{b_1(S)}\} \subset S$ such that $S \ssm {\bf c}$ deformation retracts onto a point.
\end{definition}

\begin{definition} \label{def:m-open-book}
An abstract open book $(S,h,{\bf c})$ determines a contact 3-manifold $M(S,h,{\bf c})$ with convex boundary $S^2$ as follows.  We take a contact handlebody $H(S) \cong S \times [-1,1]$ with a tight, $[-1,1]$-invariant contact structure such that after rounding corners, the boundary $\partial H(S)$ is a convex surface identified with the double of $S$ and its dividing set is $\Gamma = \partial S$.  We then attach contact $2$-handles to $H(S)$ along the collection $\gamma(h,{\bf c})$ of curves
\[ \gamma_i = \big(c_i\times\{1\}\big) \cup \big(\partial c_i \times [-1,1]\big) \cup \big(h(c_i)\times \{-1\}\big) \]
for $1 \leq i \leq b_1(S)$, and denote the result by $M(S,h,{\bf c})$.
\end{definition}

\begin{definition}
An \emph{open book decomposition} of a based contact 3-manifold $(Y,\xi,p)$ is a tuple
\[ \cB = (S,h,{\bf c}, f) \]
consisting of an abstract open book $(S,h,{\bf c})$ and a contactomorphism
\[ f: M(S,h,{\bf c}) \to (Y(p), \xi|_{Y(p)}), \]
in which $Y(p)$ is the complement of a Darboux ball around $p$.
\end{definition}

Let $\cB = (S,h,{\bf c}, f)$ be an open book decomposition of $(Y,\xi,p)$.  We choose a compact, oriented surface $T$ and an identification $\partial T \cong \partial S$, let $R=S\cup T$, and extend $h$ to $R$ so that $h|_T = \id_T$.  We also form the mapping torus
\[ R \times_\phi S^1 = \frac{R\times[-1,3]}{(x,3) \sim (\phi(x),-1)} \]
for some diffeomorphism $\phi: R\to R$.  Then $S\times[-1,1]$ is naturally a submanifold of $R\times_\phi S^1$, so we can view the collection of  curves $\gamma(h,{\bf c})$ of Definition~\ref{def:m-open-book} as living inside $R\times_\phi S^1$. We proved the following in {\cite[Proposition~2.16]{bs-stein}}.

\begin{proposition}
Performing $\partial(S\times[-1,1])$-framed surgery on each $\gamma_i \in \gamma(h,{\bf c})$ inside $R\times_\phi S^1$ produces a manifold which is canonically diffeomorphic to
\[ M(S,h,{\bf c}) \cup_\partial ((R\times_{h^{-1} \circ \phi} S^1) \ssm B^3) \]
up to isotopy.
\end{proposition}

Using the diffeomorphism $f: M(S,h,{\bf c}) \to Y(p)$, we define a cobordism
\[ V_{\cB,\phi}: R\times_\phi S^1 \to Y \# (R\times_{h^{-1}\circ\phi} S^1) \]
by attaching 2-handles to $(R\times_\phi S^1)\times[0,1]$ along each of the $\partial(S\times[-1,1])$-framed curves $\gamma_i \in \gamma(h,{\bf c})$.  Finally, we fix a closed curve
\[ \alpha \in R\times_\phi S^1, \]
disjoint from a neighborhood of $S\times[-1,1]$, such that $\alpha \cap (R\times[-1,1])$ is an arc $\{t_0\} \times[-1,1]$ for some $t_0 \in \inr(T)$.  We also use $\alpha$ to denote its image in $Y\#(R\times_{h^{-1}\circ\phi}S^1)$.

We now recall from Proposition~\ref{prop:fiberedframed} that
\begin{equation}\label{eqn:evenoddsplit} I^\#_\geven(-R\times_{\phi} S^1,\alpha|{-}R) \cong I^\#_\godd(-R\times_{\phi} S^1,\alpha|{-}R) \cong \C, \end{equation} where we are using the notation from the end of \S\ref{ssec:framed-homology},
and we define
\[ \fiinvt(\cB,R,\phi) \subset I^\#(-Y \#(-R\times_{h^{-1}\circ\phi}S^1), \alpha|{-}R) \]
to be the image of the map induced by the cobordism $(-V_{\cB,\phi}, \alpha\times[0,1])$, restricted to the top eigenspaces of $-R$, which we denote as in \S\ref{ssec:framed-homology} by 
\begin{equation}\label{eqn:mapVB} I^\#(-V_{\cB,\phi},\alpha\times[0,1]|{-}R): I^\#(-R\times_{\phi}S^1, \alpha|{-}R) \to I^\#(-Y \#(-R\times_{h^{-1}\circ\phi}S^1), \alpha|{-}R). \end{equation}
Since $I^\#(-V_{\cB,\phi},\alpha\times[0,1])$ is homogeneous with respect to the $\Z/2\Z$-grading, by Proposition \ref{prop:grading-shift}, the subspace $\fiinvt(\cB,R,\phi)$ is the direct sum of two subspaces of dimension at most 1, one in even grading and one in odd grading.

\begin{remark}
\label{rmk:whysubspace} The reason we define $\fiinvt(\cB,R,\phi)$ as the  image of $I^\#(-R\times_{\phi}S^1, \alpha|{-}R)$ under the map \eqref{eqn:mapVB} rather than as the image of \emph{some element}  under this map (which would be more in line with the definition of the contact invariant in \cite{bs-instanton}), is that the latter requires choosing an element of $I^\#(-R\times_{\phi}S^1, \alpha|{-}R)$. There are two natural choices, in light of \eqref{eqn:evenoddsplit}, corresponding to the generators in even and odd gradings, but it is not clear which one is preferred. Moreover, it is not clear whether the excision isomorphisms relating the Floer groups associated to different \emph{closures} of $S\times[-1,1]$ (corresponding to  different choices of $T$ and $\phi$) preserve the $\Z/2\Z$-grading, which means that it is not clear whether one can actually define an invariant contact element in this setting by choosing the even or odd generator.
\end{remark}

\begin{definition}
A \emph{positive stabilization} of the abstract open book $(S,h,{\bf c})$ is an open book
\[ (S', h' = D_\beta \circ h, {\bf c}' = {\bf c} \cup \{c_0\}), \]
where $S'$ is formed by attaching a 1-handle $H_0$ to $S$ with cocore $c_0$, and $D_\beta$ is a right-handed Dehn twist about some closed curve $\beta \subset S'$ which intersects $c_0$ transversely in a single point.  There is then a canonical contactomorphism
\[ q: M(S',h',{\bf c}') \xrightarrow{\sim} M(S,h,{\bf c}) \]
up to isotopy, and thus an open book decomposition $\cB = (S,h,{\bf c},f)$ of $(Y,\xi,p)$ can be positively stabilized to an open book decomposition
\[ \cB' = (S',h',{\bf c}',f' = f\circ q). \]
See \cite[\S2.3]{bs-instanton} for details.
\end{definition}

\begin{proposition} \label{prop:stabilization-invariance}
Let $\cB' = (S',h' = D_\beta \circ h,{\bf c'},f')$ be a positive stabilization of $\cB = (S,h,{\bf c},f)$, and embed $S'$ in a closed surface $R$ as above.  Then
\[ \fiinvt(\cB',R,\phi) = \fiinvt(\cB,R,D_\beta^{-1}\circ\phi) \]
for any diffeomorphism $\phi: R \to R$.
\end{proposition}

\begin{proof}
The proof is identical to that of \cite[Proposition~4.5]{bs-instanton}. The point is that  we can compute $\fiinvt(\cB',R,\phi)$ as the image of the  composition of a certain cobordism map 
\[ \Psi_{\phi,\beta}: I^\#(-R \times_\phi S^1,\alpha|{-}R) \to I^\#(-R \times_{D_\beta^{-1}\circ\phi} S^1,\alpha|{-}R) \]
with the cobordism map $I^\#(-V_{\cB,D_\beta^{-1}\circ\phi},\alpha\times[0,1]|{-}R)$ whose image is the subspace \[\fiinvt(\cB,R,D_\beta^{-1}\circ\phi)\subset I^\#(-Y \# (-R \times_{h^{-1}\circ (D_\beta^{-1}\circ\phi)} S^1), \alpha|{-}R).\] (Note  that $h^{-1}\circ (D_\beta^{-1}\circ\phi) = (D_\beta \circ h)^{-1} \circ \phi = (h')^{-1}\circ\phi$.)
Now, the map $\Psi_{\phi,\beta}$ is induced by the cobordism obtained by   attaching a 2-handle along a copy of $\beta$, viewed as a curve in some fiber $-R$.  It thus fits into a surgery exact triangle in which the remaining entry is zero because the surgery compresses $-R$, making it homologous to a surface of strictly lower genus.   We conclude that $\Psi_{\phi,\beta}$ is an isomorphism when restricted to the top eigenspaces of $\mu(-R)$, from which the proposition follows; see \cite{bs-instanton} for more details.
\end{proof}

We wish to use the subspaces $\fiinvt(\cB,R,\phi)$ defined above to obstruct overtwistedness (see Corollary \ref{cor:ot}). The following proposition is the technical key behind this.

\begin{proposition} \label{prop:overtwisted-vanishing}
Let $\cB = (S,h,{\bf c},f)$ be an open book decomposition of a based contact manifold $(Y,\xi,p)$.  Suppose that some page contains a nonseparating, nullhomologous Legendrian knot $\Lambda$ for which the contact framing of $\Lambda$ agrees with its page framing, and 
\[ \ltb(\Lambda) \geq 2g(\Lambda) > 0. \]
Let $R$ be a surface containing $S$ as a subsurface, as above.  Then $\fiinvt(\cB,R,\phi) = \{0\}$ for all diffeomorphisms $\phi: R\to R$.
\end{proposition}

\begin{proof}
This is proved by combining the arguments of \cite[Proposition~4.6]{bs-instanton} and \cite[Theorem~4.10]{bs-instanton}, without any substantial changes.  We outline these arguments below for the reader's benefit.

First, we form $(Y_-,\xi_-)$ by Legendrian surgery along $\Lambda$, and let the cobordism
\[ X: Y \to Y_- \]
denote its trace.  Then $(Y_-,\xi_-)$ is supported by an open book
\[ \cB_- = (S,D_\Lambda\circ h,{\bf c},f_-), \]
and if we let $\Lambda_-$ be the image of a Legendrian push-off of $\Lambda$ in $(Y_-,\xi_-)$, then we can recover $(Y,\xi)$ by performing contact $(+1)$-surgery on $\Lambda_-$.  The trace of this contact $(+1)$-surgery is diffeomorphic to the reversed cobordism $X^\dagger: -Y_- \to -Y$, and it follows exactly as in \cite[Proposition~4.6]{bs-instanton} that the map
\[ I^\#(-Y_- \# (-R \times_{h^{-1}\circ\phi} S^1), \alpha|{-R}) \to I^\#(-Y \# (-R\times_{h^{-1}\circ\phi}S^1), \alpha|{-}R) \]
induced by the cobordism
\[ X^{\dagger,\bowtie} := X^\dagger \bowtie \big( (-R\times_{h^{-1}\circ\phi}S^1) \times [0,1]\big) \]
sends $\fiinvt(\cB_-,R,D_\Lambda\circ\phi)$ to $\fiinvt(\cB,R,\phi)$.

Now we let $\Sigma \subset Y$ be a Seifert surface for $\Lambda$ of minimal genus $g=g(\Lambda)$, and cap it off inside $X$ to form a closed surface $\hat\Sigma$ of genus $g$ and self-intersection $\ltb(\Lambda)-1 \geq 2g-1$.  This surface persists inside $X^{\dagger,\bowtie}$, and since we have $\hat\Sigma\cdot\hat\Sigma > 2g(\hat\Sigma)-2$ the adjunction inequality says that
\[ I^\#(X^{\dagger,\bowtie},\alpha\times[0,1]|{-}R) = 0. \]
But $\fiinvt(\cB,R,\phi)$ lies in the image of this map, so it must be zero.
\end{proof}

\begin{remark}
Strictly speaking, the adjunction inequality as stated in Proposition~\ref{prop:adjunction-inequality} requires that $b_1(X^{\dagger,\bowtie})=0$, though we can appeal more generally to \cite{km-embedded2} here.  Alternatively, if $Y$ is a rational homology sphere then $b_1(X^\dagger)=0$, and  one can   argue  that the connected sum of either $-Y$ or $-Y_-$ with $(-R\times_{h^{-1}\circ\phi}S^1)\#T^3$ has ``strong simple type'' in the sense of Mu\~noz \cite{munoz-basic}, specifically when restricting to the $(2g(R)-2,2)$-eigenspace of $\mu(-R),\mu(\pt)$, and then repeat verbatim the proof of the structure theorem which led to Proposition~\ref{prop:adjunction-inequality}.  This case will suffice for our purposes, in which  $Y=S^3$.
\end{remark}

\begin{corollary}
\label{cor:ot} Suppose $\xi$ is an overtwisted contact structure on a based manifold $(Y,p)$. Then there is a supporting open book decomposition of $(Y,\xi,p)$,
\[ \cB_{\mathrm{ot}} = (S_{\mathrm{ot}},h_{\mathrm{ot}},{\bf c}_{\mathrm{ot}},f_{\mathrm{ot}}), \]
such that \[\fiinvt(\cB_{\mathrm{ot}},R,\phi) = \{0\}\] for any $R$ containing $S_{\mathrm{ot}}$ as a subsurface, and all diffeomorphisms $\phi: R\to R$. 
\end{corollary}

\begin{proof}
 That $\xi$ is overtwisted implies that  we can find  a Legendrian right-handed trefoil $\Lambda$ contained in a ball in $Y$ with $\ltb(\Lambda) = 2$ (see  \cite[Proof of Theorem 4.1]{bs-instanton}), violating the Thurston--Bennequin inequality.  There exists a supporting open book decomposition
\[ \cB_{\mathrm{ot}} = (S_{\mathrm{ot}},h_{\mathrm{ot}},{\bf c}_{\mathrm{ot}},f_{\mathrm{ot}}) \]
in which $\Lambda$ lies in a  page as a nonseparating curve, with contact framing equal to its page framing. The corollary then follows from Proposition \ref{prop:overtwisted-vanishing}.
\end{proof}

\subsection{Open books from cables of L-space knots} \label{ssec:ob-cables} For the rest of this subsection, let us fix a nontrivial instanton L-space knot $K$ and positive, coprime integers $p$ and $q$ with $q\geq 2$ and $\frac{p}{q} > 2g(K)-1$, so that $K_{p,q}$ is also an instanton L-space knot, by Lemma \ref{lem:cable-l-space}, and \[\xi_{K_{p,q}}\cong \xi_{K_{1,q}}\cong\xi_K,\] by Proposition \ref{prop:cabletight}.
Let \[\cB_{p,q} = (S_{p,q}, h_{p,q}, {\bf c}_{p,q},f_{p,q})\] be an open book decomposition of $(S^3,\xi_{K_{p,q}})$ with binding $K_{p,q}$ which encodes  the fibration associated to the fibered knot $K_{p,q}$.  That is, $S_{p,q}$ is a fiber surface of $K_{p,q}$, $h_{p,q}: S_{p,q} \xrightarrow{} S_{p,q}$ is the monodromy of the fibration, and ${\bf c}_{p,q}$ is some basis of arcs for $S_{p,q}$.  We will prove that
\[ \fiinvt(\cB_{1,q},R,h) \neq \{0\} \]
for an appropriate choice of $R \supset S_{1,q}$ and $h$ (see Proposition \ref{prop:v-bowtie-injective}). We will then use this in combination with Proposition \ref{prop:stabilization-invariance} and Corollary \ref{cor:ot} to conclude that $\xi_{K_{1,q}}\cong\xi_{K}$ is tight (though will only carry this out explicitly  for $q=2$), proving Theorem \ref{thm:sqp}.

 The reason we consider cables is that they can be deplumbed. More precisely, Neumann--Rudolph  proved in \cite[\S4.3, Theorem]{neumann-rudolph} that $K_{p,q}$ is a Murasugi sum of $K_{1,q}$ with the torus knot $T_{p,q}$; see \cite[Figure~4.2]{neumann-rudolph} or \cite[Figure~1]{bev-cabling}. We can express this  as a Murasugi sum of abstract open books, by  \[ (S_{p,q},h_{p,q}) \cong (S_{1,q},h_{1,q}) \ast (\Sigma_{p,q},\phi_{p,q}), \] where $(\Sigma_{p,q},\phi_{p,q})$ is an  open book for the fibration associated with $T_{p,q}$. But $(\Sigma_{p,q},\phi_{p,q})$ can itself be constructed by plumbing together $1-\chi(\Sigma_{p,q}) = 2g(T_{p,q})$ positive Hopf bands; see, e.g.,\ \cite[\S9.1]{ozbagci-stipsicz}. Thus,  $(S_{p,q},h_{p,q})$ is obtained from the open book
\begin{equation}\label{eqn:Smurasugi} (S_{p,q},h_{1,q}) := (S_{1,q}, h_{1,q}) \ast (\Sigma_{p,q}, \id) \end{equation}
by adding a right-handed Dehn twist to the monodromy along the core of each Hopf band used to construct $\Sigma_{p,q}$.  Note that the open book $(S_{p,q},h_{1,q})$ defined in \eqref{eqn:Smurasugi}  supports the contact manifold
\[ \left(S^3 \# \left(\#^{2g(T_{p,q})} (S^1\times S^2)\right), \xi_K \# \xi_\std\right), \]
since $(S_{1,q},h_{1,q})$ supports the contact structure $\xi_{K_{1,q}} \cong \xi_K$ on $S^3$ and any open book of the form $(S,\id)$ supports the unique tight contact structure $\xi_\std$ on  \[\#^{1-\chi(S)} (S^1\times S^2).\] For notational simplicity, let \[Y := \#^{2g(T_{p,q})}(S^1\times S^2),\] which we identify with  the manifold supported by the open book $(S_{p,q}, h_{1,q})$, and let $J\subset Y$ denote the binding of this open book. Then $g(J)=g(K_{p,q})$, since $J$ and $K_{p,q}$  both bound minimal genus Seifert surfaces $\Sigma\subset Y$ and $\Sigma'\subset S^3$ which are identified with $S_{p,q}$.

Let \begin{equation*}\label{eqn:cobordismcable}(X_0,\nu_0): (S^3,0) \to (S^3_{0}(K_{p,q}),\mu)\end{equation*} be the cobordism given by the trace of $0$-surgery on $K_{p,q}$, where $\mu$ is a meridian of $K_{p,q}$ and $\nu_0$ is the surface used in \S\ref{sec:L-spaces}. As mentioned above, the monodromy  $h_{p,q}$ is obtained from  $h_{1,q}$ by adding right-handed Dehn twists along curves \begin{equation}\label{eqn:corecurves}c_1,\dots,c_{g(T_{p,q})}\end{equation}  in $\Sigma_{p,q}\subset S_{p,q}$ corresponding to cores in a plumbing description of $\Sigma_{p,q}$. It follows that $S^3$ is obtained from $Y$ by performing $-1$-surgeries on copies of these $c_i$  on $g(T_{p,q})$ parallel pages of the open book decomposition $(S_{p,q}, h_{1,q})$ of $Y$, with respect to the page framings of these curves. The cobordism $(X_0,\nu_0)$ above therefore fits into a commutative diagram \begin{equation} \label{eq:cobordism-diagram}
\xymatrix{
(Y,0) \ar[rr]^-{(X,\nu)} \ar[d] && (Y_0(J),\mu) \ar[d] \\
(S^3,0) \ar[rr]^-{(X_0,\nu_0)} && (S^3_0(K_{p,q}),\mu),
}
\end{equation} with the cobordism \[(X,\nu): (Y,0) \to (Y_{0}(J),\mu)\] given by the trace of $0$-surgery on $J$, where $\mu$ is a meridian of $J$, and $\nu$ is the corresponding surface. The vertical arrows in this diagram correspond to the cobordisms obtained as the trace of $-1$-surgeries on the $c_i$, viewed as curves in $Y$ and $Y_0(J)$. Commutativity follows from the fact that these curves are disjoint from $J$ and $\mu$.

The first step toward proving that $\xi_{K_{1,q}}$ is tight is the following lemma.

\begin{lemma} \label{lem:y0j-cobordism-injective}
Let  $\hat\Sigma'\subset Y_0(J)$ denote the closed surface obtained by capping off the Seifert surface $\Sigma' \cong S_{p,q}$ for $J$, and let \[ (X^\dagger,\nu^\dagger):  (-Y_0(J),\mu) \to (-Y,0) \] be the cobordism obtained by turning $(X,\nu)$ upside-down.  Then the induced map
\begin{equation} \label{eq:y0j-cobordism}
I^\#(X^\dagger,\nu^\dagger|{-}\hat\Sigma'): I^\#(-Y_0(J),\mu|{-}\hat\Sigma') \to I^\#(-Y)
\end{equation}
is nonzero.
\end{lemma}

\begin{proof}
The commutative diagram \eqref{eq:cobordism-diagram} gives rise to the commutative diagram 
\begin{equation} \label{eq:open-book-monodromy-diagram}
\xymatrix{
I^\#(-S^3_0(K_{p,q}),\mu|{-}\hat\Sigma) \ar[rrr]^-{I^\#(X_0^\dagger,\nu_0^\dagger|{-}\hat\Sigma)} \ar[d] &&& I^\#(-S^3) \ar[d] \\
I^\#(-Y_0(J),\mu|{-}\hat\Sigma') \ar[rrr]^-{I^\#(X^\dagger,\nu^\dagger|{-}\hat\Sigma')} &&& I^\#(-Y),
}
\end{equation}
whose arrows are given by the maps  induced by the cobordisms in \eqref{eq:cobordism-diagram} turned upside-down and restricted to the top eigenspaces of $\mu(\hat\Sigma')$ and $\mu(\hat\Sigma)$, where $\hat\Sigma$ is the closed surface in $ S^3_0(K_{p,q})$ obtained by capping off the Seifert surface $\Sigma\cong S_{p,q}$ for $K_{p,q}$. 

We  claim that the rightmost vertical map
\[ I^\#(-S^3) \to I^\#(-Y) = I^\#(-\#^{2g(T_{p,q})}(S^1\times S^2)) \]
is injective.  Indeed, each  curve $c_i$ is dual to a 2-sphere in a unique $S^1\times S^2$ summand of $Y = \#^{2g(T_{p,q})}(S^1\times S^2)$, which at the level of open books comes from plumbing an annulus with trivial monodromy onto $(S_{1,q},h_{1,q})$. It follows that the map $I^\#(-S^3) \to I^\#(-Y)$ is induced by 0-surgery on a $2g(T_{p,q})$-component unlink, and so an argument using the surgery exact triangle \eqref{eq:triangle-untwisted} exactly as in \cite[\S7.6]{scaduto} shows that it must be injective.

Moreover, the map in the top row of \eqref{eq:open-book-monodromy-diagram} is  dual to the map
\begin{equation}\label{eqn:mapX0cable} I^\#(X_0,\nu_0;t_{0,1-g(K_{p,q})}): I^\#(S^3) \to I^\#(S^3_0(K_{p,q}),\mu;s_{1-g(K_{p,q})}) \end{equation} in the notation of \S\ref{sec:L-spaces}, where in particular \[s_{1-g(K_{p,q})}: H_2(S^3_0(K_{p,q}))\to 2\Z\] is the homomorphism defined by 
\[s_{1-g(K_{p,q})}([\hat\Sigma]) = 2-2g(K_{p,q})=2-2g(\hat\Sigma).\]  Since $K_{p,q}$ is an instanton L-space knot, Theorem~\ref{thm:odd-dim-g-1} and Proposition~\ref{prop:l-space-eigenspaces} (together with the conjugation symmetry of Theorem~\ref{thm:eigenspace-decomposition}) imply that the map \eqref{eqn:mapX0cable} is injective, with image spanned by the nonzero element 
\[ y_{1-g(K_{p,q})} \in I^\#_\godd(S^3_0(K_{p,q}),\mu;s_{1-g(K_{p,q})}). \]
Thus, the dual $I^\#(X_0^\dagger,\nu_0^\dagger|{-}\hat\Sigma)$ of this map is also nonzero.

Since the diagram \eqref{eq:open-book-monodromy-diagram} is commutative and the composition of top and right maps is nonzero, we conclude that the same must be true for the composition 
of the left and bottom maps.  In particular, the bottom map
\[ I^\#(X^\dagger,\nu^\dagger|{-}\hat\Sigma'): I^\#(-Y_0(J),\mu|{-}\hat\Sigma') \to I^\#(-Y) \]
must also be nonzero, as claimed.
\end{proof}

To prove that the subspace $\fiinvt(\cB_{1,q},R,h)$ is nontrivial for some $R$ and $h$, we  relate the cobordism map in Lemma \ref{lem:y0j-cobordism-injective} with the map which defines this subspace. First note that the surface $S_{p,q}=S_{1,q}\ast \Sigma_{p,q}$ can also be expressed as \[S_{p,q}= S_{1,q}\cup T,\]
where $T$ is a surface of genus $g(T_{p,q})$ and one boundary component, along which it is glued to $\partial S_{1,q}$. The $0$-surgery $Y_0(J)$ is then given by the mapping torus,
\[ Y_0(J) = (S_{1,q} \cup T) \times_h S^1, \]
where the map \[h: S_{1,q} \cup T \to S_{1,q} \cup T\] is equal to $h_{1,q}$ on $S_{1,q}$ and the identity on $T$.  Let us define
\[ R := S_{1,q} \cup T, \] and observe that under the identification $Y_0(J)=R\times_h S^1$, the capped-off Seifert surface $\hat\Sigma'$  in  Lemma \ref{lem:y0j-cobordism-injective} is identified with a copy of the fiber $R$. In particular,  $g(R) = g(K_{p,q})=g(\hat\Sigma')$.  Recall that the subspace \[\fiinvt(\cB_{1,q},R,h)\subset I^\#(-S^3\#(-R\times S^1),\alpha|{-}R)\] is the image of the  map \[ I^\#(-V_{\cB_{1,q},h},\alpha\times[0,1]|{-}R): I^\#(-R\times_h S^1,\alpha|{-}R) \to I^\#(-S^3\#(-R\times S^1),\alpha|{-}R) \] induced by the cobordism
\[ V_{\cB_{1,q},h}: R \times_h S^1 \to S^3 \# (R \times S^1), \]
obtained by attaching $\partial(S_{1,q}\times[-1,1])$-framed $2$-handles to $(R \times_{h} S^1)\times[0,1]$ along each of the $2g(K_{1,q})$ curves $\gamma_j\times\{1\}$ for \[\gamma_j \in \gamma(h_{1,q},{\bf c}_{1,q}),\] where ${\bf c}_{1,q}$ is a basis of arcs for $S_{1,q}$ and  $\alpha$ is the meridian $\mu$ of $J$ in $-Y_0(J) = -R\times_h S^1.$ 

\begin{proposition} \label{prop:v-bowtie-injective}
The cobordism map
\[ I^\#(-V_{\cB_{1,q},h},\alpha\times[0,1]|{-}R): I^\#(-R\times_h S^1,\alpha|{-}R) \to I^\#(-S^3\#(-R\times S^1),\alpha|{-}R) \]
 above is nonzero.  In particular, 
$\fiinvt(\cB_{1,q},R,h) \neq \{0\}. $
\end{proposition}

\begin{proof}
We  prove that the cobordism \[(X^\dagger,\nu^\dagger):(-Y_0(J),\mu)\to(-Y,0)\] of Lemma \ref{lem:y0j-cobordism-injective} factors through $(-V_{\cB_{1,q},h},\alpha\times[0,1])$. We can write the former as \[(X^\dagger,\nu^\dagger):(-R\times_h S^1,\alpha)\to(-\#^{2g(T_{p,q})}(S^1\times S^2),0),\] per the discussion preceding the proposition.  For the claimed factorization, first note that $X^\dagger$ is the cobordism given by the trace of $0$-surgery  on a section $\{\pt\} \times S^1$ of \[-T \times S^1 \subset -R\times_h S^1.\]
In the surgered manifold $-\#^{2g(T_{p,q})}(S^1\times S^2)$, the induced curves $\gamma_j \in \gamma(h_{1,q},{\bf c}_{1,q})$ are each just unknots with page framing equal to their $0$-framing. Let \[W:-\#^{2g(T_{p,q})}(S^1\times S^2) \to -\#^{2g(K_{p,q})}(S^1\times S^2)\] be the  $2$-handle cobordism given by the trace of $0$-surgeries on these unknots, and let \[Z:-\#^{2g(K_{p,q})}(S^1\times S^2) \to -\#^{2g(T_{p,q})}(S^1\times S^2)\] be the $3$-handle cobordism cancelling these $2$-handles. Then the composition 
\[ -R \times_h S^1 \xrightarrow{X^\dagger} -\#^{2g(T_{p,q})}(S^1\times S^2) \xrightarrow{W} -\#^{2g(K_{p,q})}(S^1\times S^2)\xrightarrow{Z} -\#^{2g(T_{p,q})}(S^1\times S^2) \] is simply $X^\dagger$. 

We can now change the order of the first two sets of $2$-handle attachments since they are performed along the disjoint curves $\{\pt\} \times S^1$ and $\gamma_1,\dots,\gamma_{2g(K_{1,q})}\in \gamma(h_{1,q},{\bf c}_{1,q})$. Performing the latter set of $2$-handle attachments first results in the cobordism \[(-V_{\cB_{1,q},h},\alpha\times[0,1]):(-R\times_h S^1,\alpha)\to (-R\times S^1,\alpha).\] This shows that $X^\dagger$ can also be expressed as a composition of the form \[ -R \times_h S^1 \xrightarrow{-V_{\cB_{1,q},h}} -R\times S^1 \xrightarrow{W'}  -\#^{2g(T_{p,q})}(S^1\times S^2) \] for some cobordism \[(W',\nu'): (-R\times S^1,\alpha) \to (-\#^{2g(T_{p,q})}(S^1\times S^2),0)\] involving $2$- and $3$-handle attachments. Therefore, 
\[ I^\#(X^\dagger,\nu^\dagger|{-}R) = I^\#(W',\nu'|{-}R) \circ I^\#(-V_{\cB_{1,q},h},\alpha\times[0,1]|{-}R). \] 
Lemma~\ref{lem:y0j-cobordism-injective} says that the map $I^\#(X^\dagger,\nu^\dagger|{-}R)$ is nontrivial (recall that $\hat\Sigma$ is identified with a copy of the fiber $R$ here); hence, so is  the map \[I^\#(-V_{\cB_{1,q},h},\alpha\times[0,1]|{-}R).\]  Since its image is by definition $\fiinvt(\cB_{1,q},R,h)$, this subspace is nonzero.
\end{proof}

\begin{remark} \label{rem:sqp-which-cables} As a special case of the constructions above, let  $q=2$ and  $p=2k+1$ with $k \geq 2g(K)-1$ (so that  $\frac{p}{q}\geq 2g(K)-1$).  Then \[ g(R) = g(K_{1,2}) + g(T_{2k+1,2}) = 2g(K) + k. \]
Thus by varying $k$, we can arrange for $g(R)$ to be any integer which is at least $4g(K)-1$.  Moreover, the map $h: R \to R$  of Proposition \ref{prop:v-bowtie-injective}  is the identity on $\Sigma_{2k+1,2}\subset R$, and its restriction $h|_{S_{1,2}\subset R} = h_{1,2}$ does not depend on the choice of $k$.
\end{remark}

We may now prove Theorem \ref{thm:sqp}.

\begin{proof}[Proof of Theorem \ref{thm:sqp}]
According to Propositions \ref{prop:sqp} and \ref{prop:cabletight}, it suffices to show that $\xi_{K_{1,2}}$ is not overtwisted. Supposing otherwise, there exists by Corollary \ref{cor:ot} a supporting open book decomposition for $(S^3,\xi_{K_{1,2}})$ given by \[ \cB_{\mathrm{ot}} = (S_{\mathrm{ot}},h_{\mathrm{ot}},{\bf c}_{\mathrm{ot}},f_{\mathrm{ot}}), \] such that \begin{equation}\label{eqn:zeroclass}\fiinvt(\cB_{\mathrm{ot}},R,\phi) = \{0\}\end{equation} for any $R$ containing $S_{\mathrm{ot}}$ as a subsurface, and all diffeomorphisms $\phi: R\to R$. The Giroux correspondence \cite{giroux-icm} asserts that there is an open book decomposition
\[ \cB_+ = (S_+,h_+,{\bf c}_+,f_+) \]
which is simultaneously a positive stabilization of both $\cB_{\mathrm{ot}}$ and the open book decomposition
\[ \cB_{1,2} = (S_{1,2},h_{1,2},{\bf c}_{1,2},f_{1,2}) \]
described at the beginning of this subsection. 

Let us now embed $S_+$ in a closed, connected surface $R$ of genus at least $4g(K)-1$ (see Remark~\ref{rem:sqp-which-cables}) so that the complement $R\ssm S_+$ of the image  is connected with positive genus; this induces   embeddings $S_{\mathrm{ot}}, S_{1,2} \hookrightarrow R$ satisfying the same conditions. The vanishing \eqref{eqn:zeroclass} together with Proposition~\ref{prop:stabilization-invariance} then  implies that 
\[ \fiinvt(\cB_+,R,\phi) = \{0\} \]
for all $\phi$ as well. On the other hand,  Proposition~\ref{prop:v-bowtie-injective} says that \[\fiinvt(\cB,R,h) \neq \{0\},\] where $h|_{S_{1,2}} = h_{1,2}$ and $h|_{R\ssm S_{1,2}}=\id$,  so another application of Proposition~\ref{prop:stabilization-invariance} says  that \[\fiinvt(\cB_+,R,h') \neq \{0\}\] for some corresponding $h'$, which is a contradiction.
\end{proof}

\subsection{Instanton L-space slopes}\label{ssec:lspaceslopes}
Suppose $K\subset S^3$ is a nontrivial instanton L-space knot. We proved in \cite[Theorem~4.20]{bs-stein} that there exists a positive integer $N$ such that $S^3_r(K)$ is an instanton L-space iff $r\in [N,\infty)\cap \Q$, and we showed in Proposition~\ref{prop:first-l-space-slope} that \[N\leq 2g(K)-1.\] Furthermore, it follows from Theorem~\ref{thm:sqp} and Proposition~\ref{prop:sqp} that \begin{equation}\label{eqn:maxsl2g}\maxsl(K) = 2g(K)-1.\end{equation} Lidman, Pinz\'on-Caicedo, and Scaduto  prove as part of  their forthcoming work in \cite{lpcs} that $K$ being a nontrivial instanton L-space knot together with \eqref{eqn:maxsl2g}  implies that \begin{equation}\label{eqn:N2g}N\geq 2g(K)-1.\end{equation} Thus, $N=2g(K)-1$, proving Theorem \ref{thm:bound}. The argument for the bound \eqref{eqn:N2g}, which is only  part of the main work in \cite{lpcs}, is simple enough that we reproduce it below for completeness. 

\begin{proof}[Proof of Theorem \ref{thm:bound}] As discussed above, all that remains is to prove the inequality \eqref{eqn:N2g}. Note by Proposition~\ref{prop:first-l-space-slope} that this is equivalent to  \begin{equation}\label{eqn:in0bound}\dim I^\#_\godd(S^3_0(K),\mu)\geq 2g(K)-1.\end{equation} Since $\chi(I^\#(S^3_0(K),\mu))=0$ by Proposition \ref{prop:euler}, we have that \[\dim I^\#_\godd(S^3_0(K),\mu)=\dim I^\#_\geven(S^3_0(K),\mu),\] which implies that \eqref{eqn:in0bound} is equivalent to \begin{equation}\label{eqn:in0boundtotal}\dim I^\#(S^3_0(K),\mu)\geq 4g(K)-2.\end{equation} The latter is what we will prove below.

Let $g=g(K)$. Since $\maxsl(K) = 2g-1$, we can find a Legendrian representative $\Lambda$ of $K$ in the standard contact $S^3$ with classical invariants
\[ (\ttb(\Lambda), r(\Lambda)) = (\tau_0,r_0), \qquad \tau_0 - r_0 = 2g-1 \] (any Legendrian approximation of the max-$\ssl$ transverse representative of $K$ will do).
For $n \geq 1-\tau_0$, we can positively stabilize this Legendrian $k$ times and negatively stabilize it $\tau_0+n-1-k$ times to get a Legendrian representative with
\[ (\ttb,r) = (1-n, 2-2g-n+2k), \qquad 0 \leq k \leq \tau_0+n-1. \]
For odd $n \gg 0$, these values of $r$ include every positive odd number between $1$ and $n+2g-2$.

Fixing such a large value of $n$, we perform Legendrian surgery on these knots $\Lambda_i$ with
\[ (\ttb(\Lambda_i), r(\Lambda_i)) = (1-n, 2i-1), \qquad 1 \leq i \leq \frac{n+2g-1}{2}, \]
to get contact structures
\[ \xi_1,\dots,\xi_{(n+2g-1)/2} \]
on \[Y:=S^3_{-n}(K).\] Let \[W:=X_{-n}(K)\] be the trace of this $-n$-surgery, and $\hat\Sigma \subset W$ the union of a Seifert surface for $K$ with the core of the 2-handle. Then each $\xi_i$ admits a Stein filling $(W,J_i)$ with \[\langle c_1(J_i),[\hat\Sigma]\rangle = r(\Lambda_i) = 2i-1.\]  We can also take contact structures
\[ \bar{\xi}_i = T(Y) \cap \bar{J}_i T(Y), \qquad 1 \leq i \leq \frac{n+2g-1}{2}, \]
which are filled by $W$ with the conjugate Stein structure $\bar{J}_i$ for each $i$.  These satisfy \[\langle c_1(\bar{J}_i), [\hat\Sigma]\rangle = -(2i-1),\] so we have exhibited $n+2g-1$ Stein structures
\[ J_1,J_2,\dots,J_{(n+2g-1)/2}, \bar{J}_1,\bar{J}_2, \bar{J}_{(n+2g-1)/2} \]
on $W$ which are all distinguished by their first Chern classes. This implies by \cite[Theorem 1.6]{bs-stein} that the associated contact invariants \begin{equation}\label{eqn:contactclasses}\Theta(\xi_1),\dots,\Theta(\xi_{(n+2g-1)/2}),\Theta(\bar\xi_1),\dots,\Theta(\bar\xi_{(n+2g-1)/2})\end{equation} are linearly independent as elements of  the sutured instanton homology of the complement of a ball in $-Y$ (with suture consisting of one circle). Concretely, we showed in \cite{bs-stein} that these contact invariants can all be thought of simultaneously as elements of \[I_*(-Y\#(R\times S^1)|{-}R)_{\alpha\cup \eta}\] for some fixed closed surface $R$, where $\alpha=\{\pt\}\times S^1$ for some $\pt\in R$, and $\eta$ is a homologically essential curve in some copy of the fiber $R$. By removing a $4$-ball from   $W^\dagger$, we may view it as a cobordism \[W^\dagger:-Y\to -S^3.\] Let $W^{\dagger,\bowtie}$ be the cobordism  \[W^{\dagger,\bowtie} := W^\dagger \bowtie \big( (-R\times_{}S^1) \times [0,1]\big): -Y\#(R\times S^1)\to -S^3\#(R\times S^1).\] The functoriality of our contact invariants under   maps induced by Stein cobordisms implies \cite{bs-instanton,bs-stein} that the map induced by $(W^{\dagger,\bowtie},(\alpha\cup\eta)\times[0,1])$,
\[I_*(-Y\#(R\times S^1)|{-}R)_{\alpha\cup \eta}\to I_*(-S^3\#(R\times S^1)|{-}R)_{\alpha\cup \eta}\cong \C,\]
sends each of the contact classes \eqref{eqn:contactclasses} to a generator. This implies that these classes all have the same mod $2$ grading, since cobordism maps are homogeneous.

Now, Kronheimer and Mrowka showed in \cite{km-excision} that \[I^\#(-Y)\cong I_*(-Y\#(R\times S^1)|{-}R)_{\alpha\cup \eta}\] as part of their proof that sutured instanton homology is independent of the closure. This isomorphism is a composition of excision isomorphisms, and is therefore homogeneous with respect to the mod $2$ grading. So, it identifies the contact classes in \eqref{eqn:contactclasses} with $n+2g-1$ linearly independent elements \[x_1,\dots,x_{(n+2g-1)/2},\bar x_1,\dots,\bar x_{(n+2g-1)/2}\in I^\#(-Y):=I^\#(-S^3_{-n}(K))\] which all have the same mod 2 grading. The fact that \[\chi(I^\#(-S^3_{-n}(K))) = n\] then forces there to be at least $2g-1$ other linearly independent elements in $I^\#(-S^3_{-n}(K))$ in the other mod 2 grading. In conclusion, we have shown that \[\dim I^\#(-S^3_{-n}(K)) \geq n+4g-2.\] By an easy application of the surgery exact triangles 
\begin{equation*} \label{eq:triangle-0-cobordismsneg}
\dots \to I^\#(-S^3) \to I^\#(-S^3_0(K),\mu) \to I^\#(-S^3_{-1}(K)) \to \dots
\end{equation*}
and
\begin{equation*} \label{eq:triangle-n-cobordismsneg}
\dots \to I^\#(-S^3) \to I^\#(-S^3_{-k}(K)) \to I^\#(-S^3_{-(k+1)}(K)) \to \dots,
\end{equation*}
we conclude that \[\dim I^\#(-S^3_{0}(K),\mu) =\dim I^\#(S^3_{0}(K),\mu) \geq 4g-2,\] as desired.
\end{proof}

\begin{proof}[Proof of Theorem \ref{thm:main-surgery-reduction}]
This is simply a combination of Theorems~\ref{thm:l-space-knots-are-fibered}, \ref{thm:sqp}, and \ref{thm:bound}.
\end{proof}

Observe that Theorem \ref{thm:bound} and Proposition~\ref{prop:first-l-space-slope} together imply that if $K$ is a  nontrivial instanton L-space knot, then 
\[ \dim I^\#_\godd(S^3_0(K),\mu) = 2g(K)-1. \]
We know from Proposition~\ref{prop:l-space-eigenspaces} that each $I^\#(S^3_0(K),\mu;s_i)$ is spanned by the element $y_i = I^\#(X_0,\nu_0;t_{0,i})$ for $|i| \leq g(K)-1$, and is zero for $i$ outside this range.  The $y_i$ span a space of dimension equal to the number of nonzero $y_i$, since these are linearly independent.  Thus we can characterize instanton L-space knots in terms of 2-handle cobordism maps as follows.

\begin{corollary} \label{cor:l-space-cobordism-maps}
 A nontrivial knot $K \subset S^3$ is an instanton L-space knot iff the cobordism maps
\[ I^\#(X_0,\nu_0;t_{0,i}): I^\#(S^3) \to I^\#_\godd(S^3_0(K),\mu;s_i) \]
are isomorphisms for all $i$ in the range $1-g(K) \leq i \leq g(K)-1$. \qed
\end{corollary}

As in Remark~\ref{rem:l-space-eigenspaces-bundle}, this corollary does not require using a nontrivial bundle on $S^3_0(K)$ as specified by $\mu$; we can also conclude that $K$ is an instanton L-space knot iff the maps
\[ I^\#(X_0;t_{0,i}): I^\#(S^3) \to I^\#_\godd(S^3_0(K);s_i) \]
are isomorphisms for $1-g(K) \leq i \leq g(K)-1$.

\section{$SU(2)$ representation varieties of  Dehn surgeries} \label{sec:small-surgeries}

In this section, we apply our results about instanton L-space knots and  their  L-space surgeries to questions about the $SU(2)$ representation varieties of Dehn surgeries on knots in $S^3$.  As explained in the introduction, our main tool is the following \cite[Corollary~4.8]{bs-stein}; we have already used this (in the introduction) to show that Theorem \ref{thm:main-surgery-reduction} implies Theorems~\ref{thm:main-surgery} and \ref{thm:main-su2-averse}.

\begin{proposition}\label{prop:l-space-su2-cyclic}
Suppose $K\subset S^3$ is a nontrivial knot  and $r=\frac{m}{n} > 0$ is a rational  number such that $\Delta_K(\zeta^2) \neq 0$ for any $m$th root of unity $\zeta$. If $S^3_r(K)$ is $SU(2)$-abelian then  $S^3_r(K)$ is an instanton L-space.
\end{proposition}

\begin{corollary} \label{cor:prime-power-numerator}
Suppose $K\subset S^3$ is a nontrivial knot and $r=\frac{m}{n} > 0$ is a rational number in which $m$ is a prime power.  If $S^3_{r}(K)$ is $SU(2)$-abelian then $S^3_{r}(K)$ is an instanton L-space.
\end{corollary}

\begin{proof}
Write $m=p^e$ for some prime $p$ and  $e\geq 1$. If $S^3_{r}(K)$ is not an instanton L-space then Proposition~\ref{prop:l-space-su2-cyclic} says that there is  some $p^e$th root of unity $\zeta$ for which $\Delta_K(\zeta^2)=0$.  In other words, if $\Phi_k(t)$ is the cyclotomic polynomial of order $k$, then $\Phi_{p^e}(t)$ divides $\Delta_K(t^2)$ as elements of $\Z[t,t^{-1}]$.  Setting $t=1$, it follows that $\Phi_{p^e}(1)=p$ divides $\Delta_K(1) = 1$, a contradiction.
\end{proof}

\begin{remark} \label{rmk:dense}
The set of rational numbers $r=\frac{m}{n}>0$ with $m$ a prime power is dense in $[0,\infty)$. Indeed, given \emph{any} rational $s=\frac{a}{b}>0$, the rational numbers
\[ s \pm \frac{1}{kb} = \frac{ka \pm 1}{kb}, \quad k \geq 2 \]
are also positive  and approach $s$ as $k$ goes to infinity, and for either choice of sign there are infinitely primes congruent to $\pm 1 \pmod{a}$, hence infinitely many $k$ such that the numerator $ka \pm 1$ is prime.
\end{remark}

\begin{proposition} \label{prop:cyclic-2-3} Suppose $K\subset S^3$ is a nontrivial knot and 
 $r = \frac{m}{n}\in (0,3)$ is a rational number with $m$  a prime power.  Then $S^3_r(K)$ is not $SU(2)$-abelian.\end{proposition}

\begin{proof}
Suppose for a contradiction that $S^3_r(K)$ is  $SU(2)$-abelian. Then  $S^3_r(K)$ is an instanton L-space, by Corollary~\ref{cor:prime-power-numerator}. It follows that $S^3_{\lfloor r \rfloor}(K)$ is an instanton L-space as well, by \cite[Theorem~4.20]{bs-stein}. This is impossible by definition for $r\in(0,1)$, so let us assume that $r\in[1,3)$, in which case $\lfloor r \rfloor$ is equal to $1$ or $2$.  Proposition~\ref{prop:2-surgery-l-space} then tells us that $K$ is the right-handed trefoil $T_{2,3}$. So to finish the proof,  we need only  show that $S^3_r(T_{2,3})$ is not $SU(2)$-abelian for any  $r\in [1,3)$.

We can assume $r$ is not an integer.  Indeed, when $r=1$ we know that the fundamental group of $S^3_1(T_{2,3})=-\Sigma(2,3,5)$ is already a non-abelian subgroup of $SU(2)$.  In the case $r=2$, it is not hard to show that the trefoil group admits an irreducible $SU(2)$ representation $\rho$ satisfying $\rho(\mu) = e^{\pi i/4}$ and $\rho(\lambda) = e^{3\pi i/2}$ as unit quaternions, and then $\rho(\mu^2\lambda)=1$, so $\rho$ descends to an irreducible representation $\pi_1(S^3_2(T_{2,3})) \to SU(2)$.

To prove the claim for positive $r\not\in\Z$, we note that Dehn surgery of slope $s=\frac{6k+1}{k}$ on $T_{2,3}$ gives a lens space, which is $SU(2)$-abelian, for all integers $k \geq 1$ \cite{moser}.  Suppose now that $r = \frac{m}{n}$ is another $SU(2)$-abelian surgery slope for $T_{2,3}$, for some fixed $r < 3$. Then Lin \cite{lin} proved that the distance $\Delta(r,s)$ between these slopes is at most the sum of the absolute values of their numerators, i.e.,
\[ (6k+1)n - km \leq (6k+1) + m \quad\Longleftrightarrow\quad (k+1)m \geq (6k+1)(n-1). \]
Since $n \geq 2$, the right side is at least $(6k+1)\frac{n}{2}$, hence
\[ r = \frac{m}{n} \geq \frac{6k+1}{2(k+1)} = 3 - \frac{5}{2(k+1)}. \]
But this is false for large enough $k$, so we have a contradiction.
\end{proof}

We now address Question~\ref{q:km-su2}, asked by Kronheimer and Mrowka in \cite[\S4.3]{km-su2}, about when $3$-surgery and $4$-surgery on a nontrivial knot can be $SU(2)$-abelian, proving Theorem \ref{thm:main-34-surgery} at the end of this section. We begin with the following observation.

\begin{lemma} \label{lem:4-surgery-det-1}
If $K\subset S^3$ is a knot for which $S^3_4(K)$ is $SU(2)$-abelian then $\det(K) = 1$.
\end{lemma}

\begin{proof}
Suppose for a contradiction that $\det(K) > 1$.  Klassen \cite[Theorem~10]{klassen} proved that there are \[\frac{1}{2}(\det(K)-1) > 0\] conjugacy classes of nonabelian homomorphisms \[\rho: \pi_1(S^3\ssm K) \to SU(2)\] with image in the binary dihedral group \[D_\infty = \{e^{i\theta}\} \cup \{e^{i\theta}j\},\] where here we view $SU(2)$ as the unit quaternions.  Letting $\rho$ be any such homomorphism, we  observe that $\rho(\mu)$ cannot lie in the normal subgroup $\{e^{i\theta}\}$; otherwise, since the meridian $\mu$ normally generates $\pi_1(S^3\ssm K)$, this would force the image of $\rho$ to be abelian. So $\rho(\mu)$ is a purely imaginary quaternion. We therefore have that $\rho(\mu^2)=-1$ and thus $\rho(\mu^4)=1$.  We also claim that $\rho(\lambda)=1$; indeed, the longitude $\lambda$ belongs to the second commutator subgroup of $\pi_1(S^3\ssm K)$, so its image lies in $D_\infty'' = \{e^{i\theta}\}' = \{1\}$.  Putting all of this together, we have that  $\rho(\mu^4\lambda) = 1$, so $\rho$ induces a representation $\pi_1(S^3_4(K)) \to SU(2)$ with the same image as $\rho$, and this is nonabelian, a contradiction.
\end{proof}

\begin{proposition} \label{prop:34-surgery}
Suppose  $K \subset S^3$ is a nontrivial knot of genus $g$.  Then 
\begin{itemize}
\item $S^3_4(K)$ is not $SU(2)$-abelian, and 
\item $S^3_3(K)$ is not  $SU(2)$-abelian unless $K$ is  fibered and strongly quasipositive and $g=2$.
\end{itemize}
\end{proposition}

\begin{proof}
Suppose that $S^3_r(K)$ is $SU(2)$-abelian for $r$ equal to either $3$ or $4$.  Then Corollary~\ref{cor:prime-power-numerator} says that $S^3_r(K)$ is an instanton L-space, so by Theorem~\ref{thm:main-surgery-reduction}, $K$ is fibered and strongly quasipositive and $2g-1 \leq r < 5$, i.e., $g \leq 2$.

To rule out $g=1$, we note that by Corollary~\ref{cor:genus-1-l-space}, $K$ would have to be the right-handed trefoil.  But $S^3_4(T_{2,3})$ is not $SU(2)$-abelian by Lemma~\ref{lem:4-surgery-det-1} since $\det(T_{2,3})=3$, and neither is $S^3_3(K)$ since $\pi_1(S^3_3(T_{2,3}))$ is the binary tetrahedral group (see \cite[\S10.D]{rolfsen}), which is already a nonabelian subgroup of $SU(2)$.  Thus $g=2$.

In the case $r=3$ there is nothing left to prove, so we may now assume that $r=4$. Again by Lemma~\ref{lem:4-surgery-det-1}, we must have $\det(K)=1$.  Note that $K$ is not the right-handed trefoil since $g=2$, which means that $S^3_1(K)$ and $S^3_2(K)$ are not instanton L-spaces, by Proposition \ref{prop:2-surgery-l-space}. This implies by Proposition~\ref{prop:first-l-space-slope} that \[\dim I^\#_\godd(S^3_0(K),\mu)\geq 3.\] But  $g=2$ also implies that $S^3_3(K)$ is also an instanton L-space, by Theorem \ref{thm:main-surgery-reduction}, in which case another application of  Proposition~\ref{prop:first-l-space-slope} tells us that the inequality above is an equality. Since $\chi(I^\#(S^3_0(K),\mu)) = 0$ by Proposition \ref{prop:euler},  we have
\[ \dim I^\#(S^3_0(K),\mu) = 2\dim I^\#_\godd(S^3_0(K),\mu) = 6. \]
Thus, to demonstrate a contradiction, it will suffice to show that $\dim I^\#(S^3_0(K),\mu) \geq 8$.

Since $K$ is fibered of genus 2, the Alexander polynomial of $K$ has the form
\[ \Delta_K(t) = at^2 + bt + (1-2a-2b) + bt^{-1} + at^{-2} \]
for some integers $a=\pm1$ and $b$.  We compute from this that \[\det(K) = |\Delta_K(-1)| = |1-4b|,\] and since $\det(K)=1$ we must have $b=0$.  Letting $\hat\Sigma \subset S^3_0(K)$ be a capped-off Seifert surface for $K$ with $g(\hat\Sigma)=2$, we now apply Theorem~\ref{thm:lim} to see that the  Euler characteristics of the $(2j,2)$-eigenspaces of $\mu(\hat\Sigma),\mu(\pt)$ acting on $I_*(S^3_0(K))_\mu$ satisfy
\[ \sum_{j=-1}^1 \chi(I_*(S^3_0(K),\hat\Sigma,j)_\mu)t^j = \frac{\Delta_K(t)-1}{t-2+t^{-1}} = a(t + 2 + t^{-1}). \]
The $(\pm 2,2)$-eigenspaces therefore have dimension at least $1$ each, and the $(0,2)$-eigenspace has dimension at least $2$. Since these eigenspaces are isomorphic to the $(\pm 2 \dot i, -2)$-eigenspaces and to the $(0,-2)$-eigenspace, respectively, by Lemma \ref{lem:symmetry}, we conclude that \[\dim I_*(S^3_0(K))_\mu \geq 8.\] Since $g(\hat\Sigma)=2$ and $\mu\cdot \hat\Sigma = \pm 1$, we may apply Corollary~\ref{cor:connected-sum-low-genus} to see that \[\dim I^\#(S^3_0(K),\mu) = \dim I_*(S^3_0(K))_\mu \geq 8,\] which is the desired contradiction.
\end{proof}

\begin{proof}[Proof of Theorem~\ref{thm:main-34-surgery}]
The slopes $r\in (2,3)$ are handled by Proposition~\ref{prop:cyclic-2-3}, while Proposition~\ref{prop:34-surgery} addresses $r=3$ and $r=4$.
\end{proof}

\section{$A$-polynomials of torus knots} \label{sec:a-polynomial}

\subsection{Basic properties of $A$-polynomials}
The goal of this section is to use our preceding results to prove that a slight  enhancement of the $A$-polynomial detects infinitely many torus knots, including the trefoil, per Theorem \ref{thm:main-a-poly}.
 We begin by recalling the definition of the $A$-polynomial $A_K(M,L)$ by Cooper, Culler, Gillet, Long, and Shalen \cite{ccgls}.  

 Let $X(K)$ denote the variety of characters of representations
\[ \pi_1(S^3\ssm N(K)) \to SL(2,\C), \]
and let $X(\partial N(K))$ be the $SL(2,\C)$ character variety of the boundary torus, with restriction map
\[ i^*: X(K) \to X(\partial N(K)). \]
The representation variety $\Hom(\pi_1(T^2),SL(2,\C))$ has a subvariety $\Delta$ of diagonal representations, with a branched double covering
\[ t: \C^\ast \times \C^\ast \xrightarrow{\sim} \Delta \twoheadrightarrow X(\partial N(K)) \]
sending a pair $(M,L)$ to the character of the representation $\rho$ with
\begin{align*}
\rho(\mu) &= \begin{pmatrix} M & 0 \\ 0 & M^{-1} \end{pmatrix}, &
\rho(\lambda) &= \begin{pmatrix} L & 0 \\ 0 & L^{-1} \end{pmatrix}.
\end{align*}
We let $V \subset X(\partial N(K))$ be the union of the closures $\overline{i^*(X)}$ as $X$ ranges over irreducible components of $X(K)$ such that $i^*(X)$ has complex dimension $1$.  Then
\[ \overline{V}(K) = t^{-1}(V) \subset \C^\ast \times \C^\ast \]
is an algebraic plane curve, and we take
\[ A_K(M,L) \in \C[M^{\pm 1},L^{\pm 1}] \]
to be its defining polynomial.  It is normalized to have integer coefficients and no repeated factors, and is well-defined up to multiplication by powers of $M$ and $L$.

The $A$-polynomial of any knot $K\subset S^3$ always has a factor of $L-1$.  This corresponds to a 1-dimensional curve $X_{\mathrm{red}}$ of characters of reducible representations $\rho$ such that $\rho(\mu) = \left(\begin{smallmatrix} M & 0 \\ 0 & M^{-1} \end{smallmatrix}\right)$ and $\rho(\lambda) = \left(\begin{smallmatrix} 1 & 0 \\ 0 & 1 \end{smallmatrix}\right)$ as $M$ ranges over $\C^\ast$; clearly $t^{-1}(\overline{i^*(X_{\mathrm{red}})}) = \C^\ast \times \{1\} = \{L = 1\}$.  We will work with a slight modification of the $A$-polynomial, following Ni--Zhang \cite{ni-zhang}: we define
\[ \Airr_K(M,L) \in \Z[M^{\pm1},L^{\pm1}] \]
to be the defining polynomial of
\[ t^{-1}\left(\bigcup_{X\neq X_{\mathrm{red}}} \overline{i^*(X)}\right) \subset \C^\ast \times \C^\ast, \]
where we take the union over irreducible components $X \neq X_{\mathrm{red}}$ for which $\dim_\C i^*(X)=1$; the difference is that we now explicitly exclude $X_{\mathrm{red}}$, and so there may not be a factor of $L-1$.  By convention we take $\Airr_U(M,L)=1$.

\begin{remark}
$\Airr_K(M,L)$ is equal to either $A_K(M,L)$ or $A_K(M,L)/(L-1)$, depending respectively on whether some  component containing irreducibles contributes a factor of $L-1$ or not.
\end{remark}

Using work of Kronheimer and Mrowka \cite{km-su2}, Dunfield--Garoufalidis \cite{dunfield-garoufalidis} and Boyer--Zhang \cite{boyer-zhang} proved the following.

\begin{theorem} \label{thm:a-unknot}
If $K$ is not the unknot, then $\Airr_K(M,L)$ has an irreducible factor other than $L-1$.  Thus $A_K(M,L)=L-1$ iff $K$ is unknotted.
\end{theorem}

Our goal is to prove similar results characterizing torus knots in terms of their A-polynomials.  We begin with the following computation.

\begin{proposition} \label{prop:a-torus-knot}
If $K$ is the $(p,q)$-torus knot, then $\Airr_K(M,L)$ divides $M^{2pq}L^2-1$.
\end{proposition}

\begin{proof}
The knot complement is Seifert fibered, with generic fiber $\sigma = \mu^{pq}\lambda$ being central in $\pi_1(S^3 \ssm N(K))$.  Its image under any representation
\[ \rho: \pi_1(S^3 \ssm N(K)) \to SL(2,\C) \]
commutes with the entire image of $\rho$.  If we assume that $\rho$ sends the peripheral subgroup $\langle \mu,\lambda\rangle$ to diagonal matrices, then $\rho(\sigma)$ is diagonal, and if any other element $\rho(g)$ is not diagonal then this forces $\rho(\sigma)=\pm I$.  In other words, we have $\rho(\mu^{pq}\lambda)=\pm I$ unless $\rho$ is reducible, in which case the character $\tr(\rho)$ lies in $X_{\mathrm{red}}$.

If $X \subset X(K)$ is an irreducible component other than $X_{\mathrm{red}}$ satisfying $\dim i^*(X) = 1$, then all but at most finitely many $\chi \in X$ are the characters of irreducible representations $\rho$, so that $\chi(\mu^{pq}\lambda)$ must be identically either $2$ or $-2$ on $X$, corresponding to $\rho(\mu^{pq}\lambda) = I$ or $-I$ respectively.  But then $t^{-1}(\overline{i^*(X)})$ lies in the zero set of either $M^{pq}L-1$ or $M^{pq}L+1$, from which the proposition follows.
\end{proof}

\begin{remark}By \cite[Proposition~2.7]{ccgls}, the $A$-polynomial of $T_{p,q}$ has a factor of $M^{pq}L+1$, so in light of Proposition~\ref{prop:a-torus-knot}, the only ambiguity in $\Airr_{T_{p,q}}(M,L)$ is whether it  also contains a factor of $M^{pq}L-1$.\end{remark}

For any knot $K \subset S^3$, we will let $\cN(K)$ denote the Newton polygon in the $(L,M)$-plane of the polynomial $\Airr_K(M,L)$.  This is the convex hull of all points $(a,b) \in \Z^2$ such that the monomial $M^bL^a$ has nonzero coefficient in $\Airr_K(M,L)$; it is well-defined up to translation.  We can reinterpret Theorem~\ref{thm:a-unknot} as the statement that $\cN(K)$ is not a single point unless $K$ is the unknot.

\begin{definition} \label{def:a-thin}
We say that a nontrivial knot $K \subset S^3$ is \emph{$r$-thin} for some $r\in\Q$ if $\cN(K)$ is contained in a line segment of slope $r$.
\end{definition}

\begin{example}
Proposition~\ref{prop:a-torus-knot} says that $\cN(T_{p,q})$ is contained in the line segment from $(0,0)$ to $(2,2pq)$, of slope $pq$, so $T_{p,q}$ is $pq$-thin.
\end{example}

\begin{proposition} \label{prop:hyperbolic-thick}
If $K \subset S^3$ is a hyperbolic knot then $K$ is not $r$-thin for any $r$.
\end{proposition}

\begin{proof}
This is essentially \cite[Proposition~2.6]{ccgls}.  The key observation is that $X(K)$ contains a 1-dimensional irreducible component $X$, one of whose points is the character of a discrete faithful representation, such that the function $I_\gamma = \tr(\rho(\gamma))$ is not constant on $X$ for any peripheral element $\gamma$.  If $K$ were $\frac{p}{q}$-thin then $\Airr_K(M,L)$ would be (up to a monomial factor) a polynomial of the form $f(M^pL^q)$ for some $f\in \Z[t]$.  But then $I_{M^pL^q}$ could only take finitely many values on $X$, corresponding to the roots of $f(t)$, and this is a contradiction.
\end{proof}

\subsection{$SU(2)$-averse knots}
Sivek and Zentner make the following definition in \cite{sivek-zentner}.

\begin{definition}
A nontrivial knot $K \subset S^3$ is \emph{$SU(2)$-averse} if the set
\[ \cS(K) = \left\{ \frac{p}{q} \in \Q \ \middle|\ S^3_{p/q}(K) \mathrm{\ is\ }SU(2)\mathrm{{-}abelian}\right\} \]
is infinite.
\end{definition}
One of the main theorems in \cite{sivek-zentner} is the following.
\begin{theorem} \label{thm:main-averse}
If $K$ is $SU(2)$-averse, then $\cS(K) \subset \R$ is bounded and has a single accumulation point $r(K)$, which is a rational number with $|r(K)| > 2$.  Supposing that $r(K)$ is positive, if we let
\[ n = \lceil r(K) \rceil - 1 \]
then $S^3_n(K)$ is an instanton L-space.
\end{theorem}

We call $r(K)$ the \emph{limit slope} of $K$.  The assumption that $r(K) > 0$ is a minor one, since if $K$ is $SU(2)$-averse then so is its mirror $\mirror{K}$, with $r(\mirror{K}) = -r(K)$.  

Recall from Theorem~\ref{thm:main-su2-averse} (proved in the introduction using Theorem \ref{thm:main-surgery-reduction}) that if $K$ is an $SU(2)$-averse knot, then $K$ fibered and strongly quasipositive with  $r(K) > 2g(K)-1$. 

\begin{proposition} \label{prop:thin-implies-averse}
Let $K \subset S^3$ be a nontrivial $r$-thin knot, in the sense of Definition~\ref{def:a-thin}.  Then $K$ is $SU(2)$-averse with limit slope $r$.  In particular, $K$ is fibered and $|r| > 2g(K)-1$.
\end{proposition}

\begin{proof}
Let $r=\frac{p}{q}$ with $p$ and $q$ relatively prime.  The assumption that $K$ is $r$-thin says that up to a monomial factor, we can write
\[ \Airr_K(M,L) = f(M^pL^q) \]
for some polynomial $f \in \Z[t]$.

If $X$ is an irreducible component of $X(K)$, then $i^*(X)$ has complex dimension either $0$ or $1$, see e.g.\ \cite[Lemma~2.1]{dunfield-garoufalidis}.  Thus if $X$ does not contribute to the plane curve defining $\Airr_K(M,L)$ then either $X$ is the curve $X_{\mathrm{red}}$ of reducible characters or $i^*(X) \subset X(\partial N(K))$ is a point.  Since the latter happens for only finitely many $X$, it follows that the set $S$ of points $(M,L) \subset \C^\ast \times \C^\ast$ such that
\begin{itemize}
\item there is an irreducible representation $\rho: \pi_1(S^3\ssm N(K)) \to SL(2,\C)$ with
\begin{align*}
\rho(\mu) &= \begin{pmatrix} M & 0 \\ 0 & M^{-1} \end{pmatrix}, &
\rho(\lambda) &= \begin{pmatrix} L & 0 \\ 0 & L^{-1} \end{pmatrix};
\end{align*}
\item $\Airr_K(M,L) \neq 0$, or equivalently $f(M^pL^q) \neq 0$;
\end{itemize}
is finite.

Restricting to the subgroup $SU(2)$, every representation $\rho: \pi_1(S^3 \ssm N(K)) \to SU(2)$ is conjugate to one such that
\begin{align*}
\rho(\mu) &= \begin{pmatrix} e^{i\alpha} & 0 \\ 0 & e^{-i\alpha} \end{pmatrix}, &
\rho(\lambda) &= \begin{pmatrix} e^{i\beta} & 0 \\ 0 & e^{-i\beta} \end{pmatrix}
\end{align*}
for some constants $\alpha$ and $\beta$.  (Indeed, every element of $SU(2)$ is diagonalizable, and $\rho(\mu)$ and $\rho(\lambda)$ can be simultaneously diagonalized because they commute.)  Let
\[ T \subset U(1) \times U(1) \subset \C^\ast \times \C^\ast \]
be the set of all such pairs $(e^{i\alpha},e^{i\beta})$ arising from irreducible $SU(2)$ representations $\rho$.  Then we have
\[ \Airr_K(e^{i\alpha},e^{i\beta}) = f(e^{i(p\alpha+q\beta)}) = 0 \]
on all of $T$ except for the finitely many points of $S$.  In particular, $e^{i(p\alpha+q\beta)}$ can only take finitely many values on $T \ssm S$, namely the roots of $f(t)$, so \cite[Theorem~8.2]{sivek-zentner} tells us that $K$ is $SU(2)$-averse with limit slope $\frac{p}{q} = r$.  We apply Theorem~\ref{thm:main-su2-averse} to conclude.
\end{proof}

Finally, we recall the following facts about $SU(2)$-averse satellite knots, proved in {\cite[Theorem~1.7]{sivek-zentner}}.

\begin{theorem} \label{thm:averse-satellites}
Let $K = P(C)$ be a nontrivial, $SU(2)$-averse satellite, and let $w \geq 0$ be the winding number of the pattern $P \subset S^1\times D^2$.
\begin{itemize}
\item If $P(U)$ is not the unknot, then it is $SU(2)$-averse with limit slope $r(K)$.
\item If $w \neq 0$, then the companion $C$ is $SU(2)$-averse, and $r(K) = w^2r(C)$.
\end{itemize}
\end{theorem}

\subsection{Detecting torus knots}

Ni and Zhang \cite{ni-zhang} proved that the combination of   the polynomial $\Airr_K(M,L)$ and the knot Floer homology $\hfkhat(K)$ suffice to detect torus knots.  The reason they needed $\hfkhat(K)$ was to show that $K$ is fibered, and to determine its Seifert genus and Alexander polynomial.  However, we have seen that the $(p,q)$-torus knot $T_{p,q}$ is $pq$-thin, and that $r$-thin knots are fibered for any $r$, so in many cases we do not actually need $\hfkhat(K)$.  We make this precise below.

\begin{lemma} \label{lem:torus-impostor-satellite}
Suppose that $K$ is an $r$-thin knot for some $r \in \Q$, but that $K$ is not isotopic to a torus knot.  Then $K$ is $SU(2)$-averse with limit slope $r$, and it is both fibered and a satellite knot.  If we write $K=P(C)$ then
\begin{itemize}
\item the satellite pattern $P \subset S^1\times D^2$ has positive winding number $w \geq 1$;
\item the companion $C$ is fibered and $\frac{r}{w^2}$-thin; and
\item the knot $P(U)$ is fibered, and if it is not the unknot then it is $r$-thin;
\item the Seifert genera of $K$, $P(U)$ and $C$ are related by
\[ g(K) = w\cdot g(C) + g(P(U)). \]
\end{itemize}
Moreover, if $P(U)$ is the unknot then $w \geq 3$.
\end{lemma}

\begin{proof}
Proposition~\ref{prop:thin-implies-averse} says that $K$ is fibered.  In addition, $K$ is not the unknot by Theorem~\ref{thm:a-unknot} or a torus knot by assumption, and hyperbolic knots are not $r$-thin by Proposition~\ref{prop:hyperbolic-thick}, so $K$ must be a satellite.  

For the claims that $w \geq 1$ and that $C$ and $P(U)$ are both fibered, we note by \cite[Proposition~5.5]{burde-zieschang} that since $K$ is fibered, the commutator subgroup of $\pi_1(S^3 \ssm N(K))$ is free on $2g(K)$ generators.  Since it is finitely generated, we conclude from \cite[Corollary~4.15]{burde-zieschang} that the winding number must be nonzero.  It follows from Theorem~\ref{thm:averse-satellites} that $C$ is $SU(2)$-averse, and that either $P(U)$ is the unknot or it is also $SU(2)$-averse, and in any case Proposition~\ref{prop:thin-implies-averse} tells us that these must both be fibered.  Now we have the relation
\[ \Delta_K(t) = \Delta_{P(U)}(t) \cdot \Delta_C(t^w), \]
and since each of these knots is fibered the claim about their Seifert genera follows by computing the degrees of each of these Alexander polynomials.

For the thinness of $C$ and $P(U)$, Ni--Zhang \cite[Lemma~2.6]{ni-zhang} proved that $\Airr_{P(U)}(M,L)$ divides $\Airr_K(M,L)$ in $\Z[x,y]$, so if $P(U)$ is not the unknot (meaning that $\Airr_{P(U)}(M,L) \neq 1$ by Theorem~\ref{thm:a-unknot}) then it is also $r$-thin.  They also proved \cite[Proposition~2.7]{ni-zhang} that every irreducible factor
\[ f_C(M,L) \mid \Airr_C(M,L) \]
contributes a factor
\[ f_K(M,L) = \begin{cases} \mathrm{Red}\left[\mathrm{Res}_{\bar{L}}\left(f_C(M^w,\bar{L}),\bar{L}^w-L\right) \right], & \deg_L f_C(M,L) > 0 \\ f_C(M^w,L), & \deg_L f_C(M,L)=0 \end{cases} \]
to $\Airr_K(M,L)$.  Here $\mathrm{Res}_{\bar{L}}$ denotes the resultant that eliminates the variable $\bar{L}$, and the reduced polynomial $\mathrm{Red}(p(M,L))$ is obtained from $p(M,L)$ by removing all repeated factors.  Since $K$ is $r$-thin, the case $\deg_L f_C(M,L)=0$ does not occur and it follows that $C$ is $\frac{r}{w^2}$-thin.

Finally, suppose that $P(U) = U$ but that $w=1$.  Since $K$ is fibered, Hirasawa, Murasugi, and Silver \cite[Corollary~1]{hms} proved in this case that the pattern $P$ must be the core of $S^1\times D^2$, contradicting the fact that $K$ is a nontrivial satellite, so we must have $w \geq 2$ instead.  Similarly, if $P(U)=U$ and $w=2$ then they proved that $P$ must be a $(\pm1,2)$-cable (see \cite[p.~420]{hms}), but then \cite[Theorem~10.6]{sivek-zentner} says that $K=P(C)$ cannot be $SU(2)$-averse, so this is also impossible.  We conclude that if $P(U)=U$ then $w \geq 3$.
\end{proof}

\begin{proposition} \label{prop:thinness-values}
Suppose that $K$ is an $r$-thin knot for some $r\in\Q$.  Then $r$ is a nonzero integer with at least two distinct prime divisors; in other words, $r=pq$ for some nontrivial torus knot $T_{p,q}$.
\end{proposition}

\begin{proof}
We may assume that $r$ is nonnegative, by replacing $K$ with its mirror $\mirror{K}$ if needed.  Proposition~\ref{prop:thin-implies-averse} says that $K$ is $SU(2)$-averse with limit slope $r$, hence by Theorem~\ref{thm:main-averse} we know that $S^3_n(K)$ is an instanton L-space where $n=\lceil r\rceil - 1$.  If $0 \leq r \leq 3$ then it follows that $S^3_2(K)$ is an instanton L-space, so $K$ is the right-handed trefoil by Proposition~\ref{prop:2-surgery-l-space}; but then we should have $r=6$, so this is a contradiction and in fact $r > 3$.

We now suppose that the set
\[ R = \left\{ r\in\Q \mid r>0,\ r\mathrm{-thin\ knots\ exist},\ r\neq pq\mathrm{\ for\ any\ }T_{p,q}\right\} \]
is nonempty, and let $r_0 = \inf R$; then $r_0 \geq 3$ by the above argument.  We fix some $r \in R$ such that $r_0 \leq r < r_0+1$, and we choose $K$ to have minimal Seifert genus among all $r$-thin knots; by definition $K$ is not a torus knot.

By Lemma~\ref{lem:torus-impostor-satellite} we can write $K$ as a satellite $P(C)$, and then
\[ g(K) = w\cdot g(C) + g(P(U)). \]
Thus $g(P(U)) < g(K)$, and if $P(U) \neq U$ then we know that $P(U)$ is also $r$-thin, contradicting the minimality of $g(K)$.  We must therefore have $P(U)=U$, and so Lemma~\ref{lem:torus-impostor-satellite} says that $w \geq 3$.

Now since $K$ is $r$-thin and $w \geq 3$, the companion knot $C$ has thinness
\[ \frac{r}{w^2} \leq \frac{r}{9} < \frac{r_0+1}{9} < r_0, \]
the last inequality holding since $r_0 \geq 3 > \frac{1}{8}$.  If $r$ is not an integer with at least two distinct prime factors then neither is $\frac{r}{w^2}$, so $\frac{r}{w^2} \in R$; but since $\frac{r}{w^2} < \inf R$ we have a contradiction.  So in fact $K$ cannot exist, and we conclude that the set $R$ is empty, as desired.
\end{proof}

We can now prove that $\Airr_K(M,L)$ detects the trefoils. 

\begin{theorem} \label{thm:apoly-trefoil}
If $K \subset S^3$ is $6$-thin, then $K$ is isotopic to the right-handed trefoil.
\end{theorem}

\begin{proof}
Suppose that $K$ is $6$-thin but not isotopic to $T_{2,3}$, and that $K$ minimizes Seifert genus among such knots.  In this case Lemma~\ref{lem:torus-impostor-satellite} says that $K$ is a satellite of the form $K = P(C)$ with winding number $w \geq 1$, and that $C$ is $\frac{6}{w^2}$-thin; by Proposition~\ref{prop:thinness-values} this forces $w=1$.  Given this, Lemma~\ref{lem:torus-impostor-satellite} now also implies that $P(U)$ is a nontrivial, $6$-thin knot, and that
\[ g(K) = g(C) + g(P(U)). \]
Since $g(K)$ was assumed minimal, both $C$ and $P(U)$ must be isotopic to $T_{2,3}$.  Thus $K$ is an instanton L-space knot of genus $2$, with
\[ \Delta_K(t) = \Delta_{P(U)}(t) \Delta_C(t^w) = (t-1+t^{-1})^2, \]
or equivalently
\begin{equation} \label{eq:alex-fake-trefoil}
\frac{\Delta_K(t)-1}{t-2+t^{-1}} = t+t^{-1}.
\end{equation}
Let $\hat\Sigma \subset S^3_0(K)$ be a capped-off Seifert surface for $K$ with $g(\hat\Sigma)=2$. Then 
Theorem~\ref{thm:lim} together with \eqref{eq:alex-fake-trefoil} tells us that the $(2j,2)$-eigenspaces of $\mu(\hat\Sigma),\mu(\pt)$ acting on $I_*(S^3_0(K))_\mu$ have Euler characteristics given by
\[ \chi(I_*(S^3_0(K),\hat\Sigma,j)_\mu) = \begin{cases} 1 & j=\pm1 \\ 0 & j \neq \pm1. \end{cases} \]
The $(\pm 2,2)$-eigenspaces, which are isomorphic by Lemma \ref{lem:symmetry} (apply the map $\phi^2$), therefore  have the same dimension  $1+2m$ for some integer $m\geq 0$, and the $(0,2)$-eigenspace has dimension  $2k$ for some integer $k\geq 0$. Since these eigenspaces are isomorphic to the $(\pm 2i, -2)$-eigenspaces and  the $(0,-2)$-eigenspace, respectively, by Lemma \ref{lem:symmetry}, we have  \[\dim I_*(S^3_0(K))_\mu = 4+8m+4k.\] Since $g(\hat\Sigma)=2$ and $\mu\cdot \hat\Sigma=\pm 1$, we may apply Corollary~\ref{cor:connected-sum-low-genus} to conclude that \[\dim I^\#(S^3_0(K),\mu)=\dim I_*(S^3_0(K))_\mu = 4+8m+4k. \]
Together with the fact that $\chi(I^\#(S^3_0(K),\mu))=0$, this implies 
\[ \dim I^\#_\godd(S^3_0(K),\mu) = \frac{1}{2}\dim I_*(S^3_0(K))_\mu= 2+4m+2k. \]
Now, Proposition~\ref{prop:first-l-space-slope} says that $S^3_{2g(K)-1}(K) = S^3_3(K)$ is an instanton L-space, which then implies by the same proposition that \[3\geq  \dim I^\#_\godd(S^3_0(K),\mu) = 2+4m+2k, \] so $m=k=0$ and  \[\dim I^\#_\godd(S^3_0(K),\mu)=2.\]But the latter implies by Proposition~\ref{prop:first-l-space-slope} that $S^3_2(K)$ is an instanton L-space, and  Proposition~\ref{prop:2-surgery-l-space}  tells us in this case that $K$ is the unknot or the right-handed trefoil,  a contradiction.
\end{proof}

\begin{proposition} \label{prop:squarefree-thin}
Let $r > 12$ be an integer such that one of the following holds:
\begin{itemize}
\item $r$ is square-free and odd, with at least two distinct prime divisors;
\item $r=p^2q$ for distinct primes $p$ and $q$, with $q \geq 3$; or
\item $r=p^2q^2$ for distinct primes $p$ and $q$.
\end{itemize}
Then every $r$-thin knot is a torus knot.
\end{proposition}

\begin{proof}
Suppose that there are non-torus knots which are $r$-thin, and let $K$ be such a knot with the smallest possible genus.  By Lemma~\ref{lem:torus-impostor-satellite}, we know that $K$ is a satellite, say $K=P(C)$ where $P$ has winding number $w \geq 1$, and that the companion $C$ is $\frac{r}{w^2}$-thin.  If $w \geq 2$ then by Proposition~\ref{prop:thinness-values} there are no $\frac{r}{w^2}$-thin knots, so we must have $w=1$.

Since $w=1$, we know from Lemma~\ref{lem:torus-impostor-satellite} that $P(U)$ is not the unknot, that both $C$ and $P(U)$ are $r$-thin, and that
\[ g(K) = g(C) + g(P(U)). \]
Both $g(C)$ and $g(P(U))$ are positive and strictly less than $g(K)$, but $g(K)$ was assumed minimal, so $C$ and $P(U)$ must be nontrivial torus knots, say $C=T_{a,b}$ and $P(U)=T_{c,d}$ with $ab=cd=r$.

The assumptions on $r$ each imply that $a,b \geq 3$, so we have
\[ g(C) = \frac{(a-1)(b-1)}{2} = \frac{ab+1}{2} - \frac{a+b}{2} \geq \frac{r+1}{2} - \frac{3+r/3}{2} = \frac{r}{3}-1. \]
Likewise $g(P(U))$ satisfies the same bound, so $g(K) = g(C)+g(P(U))$ satisfies
\[ 2g(K)-1 \geq 2\left(\frac{2r}{3}-2\right)-1 = \frac{4r}{3}-5. \]
Since $K$ is $r$-thin we know from Proposition~\ref{prop:thin-implies-averse} that $r > 2g(K)-1$, or equivalently $r-1 \geq 2g(K)-1$ since $r$ is an integer.  But then we have $r-1 \geq \frac{4r}{3}-5$, or $r \leq 12$, and this is a contradiction.
\end{proof}

We say that the $\Airr$-polynomial \emph{detects} a knot $K$ if for any knot $K'$, we have $\Airr_K(M,L) = \Airr_{K'}(M,L)$ iff $K$ is isotopic to $K'$.

\begin{corollary} \label{cor:torus-knot-detection}
Let $p$ and $q$ be distinct odd primes.  Then the $\Airr$-polynomial detects each of the torus knots $T_{p,q}$, $T_{p^2,q}$, $T_{p^2,q^2}$, $T_{4,q}$ if $q > 3$, and $T_{4,q^2}$.
\end{corollary}

\begin{proof}
Write any of the given torus knots as $T_{a,b}$.  Then $r=ab$ satisfies the hypotheses of Proposition~\ref{prop:squarefree-thin}, so if $\Airr_K(M,L)=\Airr_{T_{a,b}}(M,L)$ then $K$ must be a torus knot $T_{c,d}$ with $r=cd$.  But in each case $T_{a,b}$ is the unique $r$-thin torus knot, so in fact $K = T_{a,b}$.
\end{proof}

\begin{proof}[Proof of Theorem~\ref{thm:main-a-poly}]
The trefoil case is Theorem~\ref{thm:apoly-trefoil}, and the remaining cases are included in Corollary~\ref{cor:torus-knot-detection}.
\end{proof}

\bibliographystyle{alpha}
\bibliography{References}

\newcommand{\etalchar}[1]{$^{#1}$}
\begin{thebibliography}{BEVHM12}

\bibitem[BD95]{braam-donaldson}
P.~J. Braam and S.~K. Donaldson.
\newblock Floer's work on instanton homology, knots and surgery.
\newblock In {\em The {F}loer memorial volume}, volume 133 of {\em Progr.
  Math.}, pages 195--256. Birkh\"auser, Basel, 1995.

\bibitem[BEVHM12]{bev-cabling}
Kenneth~L. Baker, John~B. Etnyre, and Jeremy Van Horn-Morris.
\newblock Cabling, contact structures and mapping class monoids.
\newblock {\em J. Differential Geom.}, 90(1):1--80, 2012.

\bibitem[BI09]{baader-ishikawa}
Sebastian Baader and Masaharu Ishikawa.
\newblock Legendrian graphs and quasipositive diagrams.
\newblock {\em Ann. Fac. Sci. Toulouse Math. (6)}, 18(2):285--305, 2009.

\bibitem[BN90]{boyer-nicas}
S.~Boyer and A.~Nicas.
\newblock Varieties of group representations and {C}asson's invariant for
  rational homology {$3$}-spheres.
\newblock {\em Trans. Amer. Math. Soc.}, 322(2):507--522, 1990.

\bibitem[BS16]{bs-instanton}
John~A. Baldwin and Steven Sivek.
\newblock Instanton {F}loer homology and contact structures.
\newblock {\em Selecta Math. (N.S.)}, 22(2):939--978, 2016.

\bibitem[BS18]{bs-stein}
John Baldwin and Steven Sivek.
\newblock Stein fillings and {SU}(2) representations.
\newblock {\em Geom. Topol.}, 22(7):4307--4380, 2018.

\bibitem[BZ03]{burde-zieschang}
Gerhard Burde and Heiner Zieschang.
\newblock {\em Knots}, volume~5 of {\em De Gruyter Studies in Mathematics}.
\newblock Walter de Gruyter \& Co., Berlin, second edition, 2003.

\bibitem[BZ05]{boyer-zhang}
Steven Boyer and Xingru Zhang.
\newblock Every nontrivial knot in {$S^3$} has nontrivial {$A$}-polynomial.
\newblock {\em Proc. Amer. Math. Soc.}, 133(9):2813--2815, 2005.

\bibitem[CCG{\etalchar{+}}94]{ccgls}
D.~Cooper, M.~Culler, H.~Gillet, D.~D. Long, and P.~B. Shalen.
\newblock Plane curves associated to character varieties of {$3$}-manifolds.
\newblock {\em Invent. Math.}, 118(1):47--84, 1994.

\bibitem[DG04]{dunfield-garoufalidis}
Nathan~M. Dunfield and Stavros Garoufalidis.
\newblock Non-triviality of the {$A$}-polynomial for knots in {$S^3$}.
\newblock {\em Algebr. Geom. Topol.}, 4:1145--1153, 2004.

\bibitem[DK90]{donaldson-kronheimer}
S.~K. Donaldson and P.~B. Kronheimer.
\newblock {\em The geometry of four-manifolds}.
\newblock Oxford Mathematical Monographs. The Clarendon Press, Oxford
  University Press, New York, 1990.
\newblock Oxford Science Publications.

\bibitem[Don02]{donaldson-book}
S.~K. Donaldson.
\newblock {\em Floer homology groups in {Y}ang-{M}ills theory}, volume 147 of
  {\em Cambridge Tracts in Mathematics}.
\newblock Cambridge University Press, Cambridge, 2002.
\newblock With the assistance of M. Furuta and D. Kotschick.

\bibitem[Flo90]{floer-surgery}
Andreas Floer.
\newblock Instanton homology, surgery, and knots.
\newblock In {\em Geometry of low-dimensional manifolds, 1 ({D}urham, 1989)},
  volume 150 of {\em London Math. Soc. Lecture Note Ser.}, pages 97--114.
  Cambridge Univ. Press, Cambridge, 1990.

\bibitem[Fr{\o}02]{froyshov}
Kim~A. Fr{\o}yshov.
\newblock Equivariant aspects of {Y}ang-{M}ills {F}loer theory.
\newblock {\em Topology}, 41(3):525--552, 2002.

\bibitem[Fr{\o}04]{froyshov-inequality}
Kim~A. Fr{\o}yshov.
\newblock An inequality for the {$h$}-invariant in instanton {F}loer theory.
\newblock {\em Topology}, 43(2):407--432, 2004.

\bibitem[FS90]{fs-seifert}
Ronald Fintushel and Ronald~J. Stern.
\newblock Instanton homology of {S}eifert fibred homology three spheres.
\newblock {\em Proc. London Math. Soc. (3)}, 61(1):109--137, 1990.

\bibitem[FS95]{fs-structure}
Ronald Fintushel and Ronald~J. Stern.
\newblock Donaldson invariants of {$4$}-manifolds with simple type.
\newblock {\em J. Differential Geom.}, 42(3):577--633, 1995.

\bibitem[FS96]{fs-blowup}
Ronald Fintushel and Ronald~J. Stern.
\newblock The blowup formula for {D}onaldson invariants.
\newblock {\em Ann. of Math. (2)}, 143(3):529--546, 1996.

\bibitem[Fuk96]{fukaya}
Kenji Fukaya.
\newblock Floer homology of connected sum of homology {$3$}-spheres.
\newblock {\em Topology}, 35(1):89--136, 1996.

\bibitem[Gab87]{gabai-foliations3}
David Gabai.
\newblock Foliations and the topology of {$3$}-manifolds. {III}.
\newblock {\em J. Differential Geom.}, 26(3):479--536, 1987.

\bibitem[Ghi08]{ghiggini}
Paolo Ghiggini.
\newblock Knot {F}loer homology detects genus-one fibred knots.
\newblock {\em Amer. J. Math.}, 130(5):1151--1169, 2008.

\bibitem[Gir02]{giroux-icm}
Emmanuel Giroux.
\newblock G\'{e}om\'{e}trie de contact: de la dimension trois vers les
  dimensions sup\'{e}rieures.
\newblock In {\em Proceedings of the {I}nternational {C}ongress of
  {M}athematicians, {V}ol. {II} ({B}eijing, 2002)}, pages 405--414. Higher Ed.
  Press, Beijing, 2002.

\bibitem[Gor83]{gordon}
C.~McA. Gordon.
\newblock Dehn surgery and satellite knots.
\newblock {\em Trans. Amer. Math. Soc.}, 275(2):687--708, 1983.

\bibitem[Hed08]{hedden-cabling}
Matthew Hedden.
\newblock Some remarks on cabling, contact structures, and complex curves.
\newblock In {\em Proceedings of {G}\"{o}kova {G}eometry-{T}opology
  {C}onference 2007}, pages 49--59. G\"{o}kova Geometry/Topology Conference
  (GGT), G\"{o}kova, 2008.

\bibitem[Hed10]{hedden-positivity}
Matthew Hedden.
\newblock Notions of positivity and the {O}zsv\'{a}th-{S}zab\'{o} concordance
  invariant.
\newblock {\em J. Knot Theory Ramifications}, 19(5):617--629, 2010.

\bibitem[HMS08]{hms}
Mikami Hirasawa, Kunio Murasugi, and Daniel~S. Silver.
\newblock When does a satellite knot fiber?
\newblock {\em Hiroshima Math. J.}, 38(3):411--423, 2008.

\bibitem[Kla91]{klassen}
Eric~Paul Klassen.
\newblock Representations of knot groups in {${\rm SU}(2)$}.
\newblock {\em Trans. Amer. Math. Soc.}, 326(2):795--828, 1991.

\bibitem[KM95a]{km-structure}
P.~B. Kronheimer and T.~S. Mrowka.
\newblock Embedded surfaces and the structure of {D}onaldson's polynomial
  invariants.
\newblock {\em J. Differential Geom.}, 41(3):573--734, 1995.

\bibitem[KM95b]{km-embedded2}
P.~B. Kronheimer and T.~S. Mrowka.
\newblock Gauge theory for embedded surfaces. {II}.
\newblock {\em Topology}, 34(1):37--97, 1995.

\bibitem[KM04a]{km-su2}
P.~B. Kronheimer and T.~S. Mrowka.
\newblock Dehn surgery, the fundamental group and {SU{$(2)$}}.
\newblock {\em Math. Res. Lett.}, 11(5-6):741--754, 2004.

\bibitem[KM04b]{km-p}
P.~B. Kronheimer and T.~S. Mrowka.
\newblock Witten's conjecture and property {P}.
\newblock {\em Geom. Topol.}, 8:295--310, 2004.

\bibitem[KM10]{km-excision}
Peter Kronheimer and Tomasz Mrowka.
\newblock Knots, sutures, and excision.
\newblock {\em J. Differential Geom.}, 84(2):301--364, 2010.

\bibitem[KM11a]{km-unknot}
P.~B. Kronheimer and T.~S. Mrowka.
\newblock Khovanov homology is an unknot-detector.
\newblock {\em Publ. Math. Inst. Hautes \'Etudes Sci.}, (113):97--208, 2011.

\bibitem[KM11b]{km-yaft}
P.~B. Kronheimer and T.~S. Mrowka.
\newblock Knot homology groups from instantons.
\newblock {\em J. Topol.}, 4(4):835--918, 2011.

\bibitem[KMOS07]{kmos}
P.~Kronheimer, T.~Mrowka, P.~Ozsv\'{a}th, and Z.~Szab\'{o}.
\newblock Monopoles and lens space surgeries.
\newblock {\em Ann. of Math. (2)}, 165(2):457--546, 2007.

\bibitem[Lim10]{lim}
Yuhan Lim.
\newblock Instanton homology and the {A}lexander polynomial.
\newblock {\em Proc. Amer. Math. Soc.}, 138(10):3759--3768, 2010.

\bibitem[Lin16]{lin}
Jianfeng Lin.
\newblock {$\rm{SU}(2)$}-cyclic surgeries on knots.
\newblock {\em Int. Math. Res. Not. IMRN}, (19):6018--6033, 2016.

\bibitem[LPCS20]{lpcs}
Tye Lidman, Juanita Pinz{\'o}n-Caicedo, and Christopher Scaduto.
\newblock Framed instanton homology of surgeries on {L}-space knots.
\newblock arXiv:2003.03329, 2020.

\bibitem[Mos71]{moser}
Louise Moser.
\newblock Elementary surgery along a torus knot.
\newblock {\em Pacific J. Math.}, 38:737--745, 1971.

\bibitem[Mu{\~n}99a]{munoz-hff}
Vicente Mu{\~n}oz.
\newblock Fukaya-{F}loer homology of {$\Sigma\times {\bf S}^1$} and
  applications.
\newblock {\em J. Differential Geom.}, 53(2):279--326, 1999.

\bibitem[Mu{\~n}99b]{munoz-ring}
Vicente Mu{\~n}oz.
\newblock Ring structure of the {F}loer cohomology of {$\Sigma\times{\bf
  S}^1$}.
\newblock {\em Topology}, 38(3):517--528, 1999.

\bibitem[Mu{\~n}00]{munoz-basic}
Vicente Mu{\~n}oz.
\newblock Basic classes for four-manifolds not of simple type.
\newblock {\em Comm. Anal. Geom.}, 8(3):653--670, 2000.

\bibitem[Ni07]{ni-hfk}
Yi~Ni.
\newblock Knot {F}loer homology detects fibred knots.
\newblock {\em Invent. Math.}, 170(3):577--608, 2007.

\bibitem[Ni08]{ni-hm}
Yi~Ni.
\newblock Addendum to: ``{K}nots, sutures and excision''.
\newblock arXiv:0808.1327, 2008.

\bibitem[Ni09a]{ni-hfk-erratum}
Yi~Ni.
\newblock Erratum: {K}not {F}loer homology detects fibred knots.
\newblock {\em Invent. Math.}, 177(1):235--238, 2009.

\bibitem[Ni09b]{ni-hf}
Yi~Ni.
\newblock Heegaard {F}loer homology and fibred 3-manifolds.
\newblock {\em Amer. J. Math.}, 131(4):1047--1063, 2009.

\bibitem[NR87]{neumann-rudolph}
Walter Neumann and Lee Rudolph.
\newblock Unfoldings in knot theory.
\newblock {\em Math. Ann.}, 278(1-4):409--439, 1987.

\bibitem[NZ17]{ni-zhang}
Yi~Ni and Xingru Zhang.
\newblock Detection of knots and a cabling formula for {$A$}-polynomials.
\newblock {\em Algebr. Geom. Topol.}, 17(1):65--109, 2017.

\bibitem[OS04]{ozbagci-stipsicz}
Burak Ozbagci and Andr\'{a}s~I. Stipsicz.
\newblock {\em Surgery on contact 3-manifolds and {S}tein surfaces}, volume~13
  of {\em Bolyai Society Mathematical Studies}.
\newblock Springer-Verlag, Berlin; J\'{a}nos Bolyai Mathematical Society,
  Budapest, 2004.

\bibitem[OS05a]{osz-contact}
Peter Ozsv\'{a}th and Zolt\'{a}n Szab\'{o}.
\newblock Heegaard {F}loer homology and contact structures.
\newblock {\em Duke Math. J.}, 129(1):39--61, 2005.

\bibitem[OS05b]{osz-lens}
Peter Ozsv\'{a}th and Zolt\'{a}n Szab\'{o}.
\newblock On knot {F}loer homology and lens space surgeries.
\newblock {\em Topology}, 44(6):1281--1300, 2005.

\bibitem[OS11]{osz-rational}
Peter~S. Ozsv\'{a}th and Zolt\'{a}n Szab\'{o}.
\newblock Knot {F}loer homology and rational surgeries.
\newblock {\em Algebr. Geom. Topol.}, 11(1):1--68, 2011.

\bibitem[Rol90]{rolfsen}
Dale Rolfsen.
\newblock {\em Knots and links}, volume~7 of {\em Mathematics Lecture Series}.
\newblock Publish or Perish, Inc., Houston, TX, 1990.
\newblock Corrected reprint of the 1976 original.

\bibitem[Sca15]{scaduto}
Christopher~W. Scaduto.
\newblock Instantons and odd {K}hovanov homology.
\newblock {\em J. Topol.}, 8(3):744--810, 2015.

\bibitem[SZ17]{sivek-zentner}
Steven Sivek and Raphael Zentner.
\newblock {$SU(2)$}-cyclic surgeries and the pillowcase.
\newblock arXiv:1710.01957, 2017.

\bibitem[TW75]{thurston-winkelnkemper}
W.~P. Thurston and H.~E. Winkelnkemper.
\newblock On the existence of contact forms.
\newblock {\em Proc. Amer. Math. Soc.}, 52:345--347, 1975.

\end{thebibliography}

\end{document}